\documentclass[12pt,a4paper,reqno]{amsart}
\usepackage{amsmath}
\usepackage{amssymb}
\usepackage{enumitem}
\usepackage{geometry}
\usepackage{mathtools,nccmath}
\usepackage{amsthm}
\usepackage{tikz}
\usetikzlibrary{shapes}
\usepackage{mathrsfs}
\usepackage{hyperref}
\usetikzlibrary{arrows}
\usepackage[utf8]{inputenc}
\numberwithin{equation}{section}
\usepackage{tikz-cd}
\usepackage[utf8]{inputenc}
\usepackage[T1]{fontenc}
\usepackage{tikz-cd}
\usepackage{txfonts}

\usepackage[%
shortalphabetic, 
abbrev, 
backrefs, 
]{amsrefs}
\hypersetup{%
	linktoc=all, 
	colorlinks=true,
	citecolor=magenta,
}

\NewDocumentCommand\Crefnameitem { m m m O{\textup} O{(\roman*)}} {%
	\Crefname{#1enumi}{#2}{#3} 
	\AtBeginEnvironment{#1}{%
		\crefalias{enumi}{#1enumi}%
		\setlist[enumerate,1]{
			label={#4{#5}.},
			ref={#5}
		}%
	}  
}

\usetikzlibrary{%
	cd, 
	shapes, 
	arrows, 
}
\NewDocumentCommand\placeholder{}{\:\cdot\:} 
\NewDocumentCommand\NewPairedDelimiterS{mmm}{%
	\DeclarePairedDelimiterX{#1}[1]{#2}{#3}%
	{\ifblank{##1}{\placeholder}{##1}}%
}
\NewDocumentCommand\NewPairedDelimiterSS{mmmO{,}}{%
	\DeclarePairedDelimiterX{#1}[2]{#2}{#3}%
	{\ifblank{##1}{\placeholder}{##1}%
		#4%
		\ifblank{##2}{\placeholder}{##2}}%
}
\NewPairedDelimiterS\abs\lvert\rvert 
\NewPairedDelimiterS\norm\lVert\rVert 
\NewPairedDelimiterS\lrangle\langle\longrightarrowngle 
\NewPairedDelimiterS\paren\lparen\rparen 
\NewPairedDelimiterS\lrbrack\lbrack\rbrack 
\NewPairedDelimiterS\lrbrace\lbrace\rbrace 
\NewPairedDelimiterS\ceil\lceil\rceil 
\NewPairedDelimiterS\floor\lfloor\rfloor 
\NewPairedDelimiterS\bra\langle\vert 
\NewPairedDelimiterS\ket\vert\rangle 
\NewPairedDelimiterS\normord{\mathopen{:}}{\mathclose{:}} 
\NewPairedDelimiterSS\braket\langle\rangle[%
\,\delimsize\vert\,\mathopen{}%
] 
\NewPairedDelimiterSS\pairing\langle\rangle
\NewPairedDelimiterSS\inner\lparen\rparen
\NewPairedDelimiterSS\Liebracket\lbrack\rbrack

\newtheorem{thm}{Theorem}[section]
\newtheorem*{theorem*}{Theorem}

\newtheorem{prop}[thm]{Proposition}
\newtheorem{lm}[thm]{Lemma}

\newtheorem{coro}[thm]{Corollary}

\newcommand\numberthis{\addtocounter{equation}{1}\tag{\theequation}}

\theoremstyle{definition}

\newtheorem{df}[thm]{Definition}
\newtheorem{remark}[thm]{Remark}
\newtheorem{construction}[thm]{Construction}

\newtheorem{innercustomgeneric}{\customgenericname}
\providecommand{\customgenericname}{}
\newcommand{\newcustomtheorem}[2]{%
	\newenvironment{#1}[1]
	{%
		\renewcommand\customgenericname{#2}%
		\renewcommand\theinnercustomgeneric{##1}%
		\innercustomgeneric
	}
	{\endinnercustomgeneric}
}
\newcustomtheorem{customthm}{Theorem}

\newcommand{\bt}{\boxtimes}
\newcommand{\fusion}[3]{{\binom{#3}{#1\;#2}}}

\newcommand{\al}{\alpha}

\newcommand{\Hom}{{\rm Hom}\,}

\newcommand{\End}{{\rm End}\,}
\newcommand{\Res}{{\rm Res}\,}

\newcommand{\Z}{\mathbb{Z}}

\newcommand{\N}{\mathbb{N}}

	\def\Res{{\rm Res}}
	\def\wt{{\rm wt}}

	\newcommand{\la}{\lambda}
	
	\def\C{{\mathbb C}}
	\def\P{{\mathbb P}}

	\def\Z{{\mathbb Z}}
	
	\def\N{{\mathbb N}}

	\def\1{{\em 1}}
	
	\def \End{{\rm End}}
	\def \Hom{{\rm Hom}}

	\def \b{\beta}
	
	\def \h{\mathfrak{h}}

	\def \ra {\rightarrow}
	\def \g{\mathfrak{g}}

	\def \ssq{\subseteq}
	\def \C {\mathbb{C}}
	
	\def\Id{\mathrm{Id}}
	\def\op{\oplus}
	
	\def\vac{\mathbf{1}}
	\def\spn{\mathrm{span}}
	\def \vac{\mathbf{1}}
	
	\def\bs{\backslash}
	\def\om{\omega}

	\def\o{\otimes}
	
	\def\<{\langle}
	\def\>{\rangle}
	
	\def\Om{\Omega}

	\def\ds{\dots}

	\def\gl{\mathfrak{gl}}
	\def\adm{\mathsf{Adm}}
	\def\sl{\mathfrak{sl}}
	\def\wphi{\widetilde{\varphi}}
	\def\ord{\mathrm{ord}}
	\def\Mod{\mathsf{Mod}}
	\newcommand \WV {\begin{smallmatrix}W\\WV\end{smallmatrix}}

	\def\<{\langle}
	\def\>{\rangle}

	\allowdisplaybreaks
	
\begin{document}
		\title{One-Point Restricted Conformal Blocks and the Fusion Tensor Product}

		\author{Jianqi Liu}
		\address{Department of Mathematics, University of Pennsylvania, Philadelphia, PA, 19104}
		\email{jliu230@sas.upenn.edu}

		\begin{abstract}
			We investigate a one-point restriction of conformal blocks on $(\P^1,\infty,1,0)$ associated with modules over a vertex operator algebra. By restricting the module attached to the point $\infty$ to its bottom degree, we obtain a new formula for computing fusion rules in terms of a left $A(V)$-module $M^1\odot M^2$ over the Zhu algebra $A(V)$. As a consequence, for strongly rational VOAs, the construction of $M^1\odot M^2$ induces the fusion tensor product on the module category $\Mod(A(V))$.


		\end{abstract}
		
		\maketitle
		
		\tableofcontents
		
		\section{Introduction}
		The space of conformal blocks on the three-pointed genus-zero smooth curve defined by modules over a vertex operator algebra (VOA) $V$ is isomorphic to the vector space of intertwining operators among these modules, whose dimension is the fusion rule \cite{NT05,FBZ04}. Using the restriction technique of conformal blocks \cite{GLZ24}, we obtain a new hom-space description of the space of intertwining operators 	$I\fusion{M^1}{M^2}{M^3}$ using a new left module $M^1\odot M^2$ over the degree-zero Borcherds' Lie algebra $L(V)_0$ \cite{B86} or the Zhu algebra $A(V)$ \cite{Z96}. If $V$ is a strongly rational VOA, and $M^1$ and $M^2$ are generalized Verma modules associated to the $A(V)$-modules $S^1$ and $S^2$, respectively, we prove that $M^1\odot M^2$ is exactly the fusion tensor product $S^1\bt S^2$ in the module category of Zhu algebra $\Mod(A(V))$. 
		
		Motivated by finding a mathematical rigorous definition of the WZNW-conformal field theory, Tsuchiya, Ueno, and Yamada introduced the notions of coinvariants (covacua) and conformal blocks (vacua) on stable algebraic curves defined by highest-weight representations of affine Kac-Moody algebras $\hat{\g}$ of non-generic level $k\in \Z_{>0}$ \cite{TUY89}. From the VOA point of view, the representation theory of affine Lie algebras is in parallel with the representation theory of affine VOAs of the same level \cite{FZ92}. With this key observation, 
		the notions of coinvariants and conformal blocks in WZNW-model were generalized to the general VOA case by Zhu \cite{Z94}, Frenkel-Ben-Zvi \cite{FBZ04}, and Nagatomo-Tsuchiya \cite{NT05} for smooth curves, and by Damiolini-Gibney-Tarasca \cite{DGT24} for general stable curves. 
		The space of three-pointed conformal blocks associated to the datum $$\Sigma((M^3)',M^1,M^2)=(\P^1,(\infty, 1,0),(1/z,z-1,z), ((M^3)', M^1,M^2))$$ 
        is canonically isomorphic to the space of intertwining operators $I\fusion{M^1}{M^2}{M^3}$  \cite{FHL93} among these $V$-modules. These spaces are the building blocks of the space of conformal blocks on higher-genus algebraic curves via the factorization theorem \cite{DGT24}. The dimension of these spaces, namely the fusion rules, are not only one of the fundamental concept in the conformal field theory (CFT), but also carry important information about the rank of the vector bundle on the moduli space $\overline{\mathcal{M}}_{g,n}$, parametrizing $n$-pointed genus-$g$ stable curves, defined by the VOA-conformal blocks \cite{DGT24,DG23}. 
		
		The structure of the vector space $I\fusion{M^1}{M^2}{M^3}$ has been studied extensively in the VOA theory. Frenkel-Zhu's fusion rules theorem states that this vector space is isomorphic to the hom-space $\Hom_{A(V)}(A(M^1)\o_{A(V)}\Om(M^2), \Om(M^3))$ \cite{FZ92}, where $A(M^1)$ is a bimodule over Zhu algebra $A(V)$. Li revised this theorem by adding the assumptions that $M^2$ and $(M^3)'$ are generalized Verma modules \cite{DLM98} associated to their bottom degrees $\Om(M^2)$ and $\Om((M^3)')$, respectively \cite{Li99}. The author proved a variant of this theorem using the technique of three-pointed correlation functions defined by intertwining operators  \cite{Liu23}, which was further generalized to the $g$-twisted case by Gao, the author, and Zhu by developing a theory of twisted correlation functions \cite{GLZ23}.
      Furthermore, we introduced a notion of (twisted) restricted conformal blocks in \cite{GLZ24}. With this new notion, we noticed that the hom space $\Hom_{A(V)}(A(M^1)\o_{A(V)}\Om(M^2), \Om(M^3))$ can be identified with the space of two-point restricted conformal blocks defined on the projective line $(\P^1,(\infty, 1,0),(1/z,z-1,z), (\Om((M^3)'), M^1,\Om(M^2)))$, wherein the $V$-modules $(M^3)'$ and $M^2$ attached to $\infty$ and $0$ are restricted to their bottom degrees $\Om((M^3)')$ and $\Om(M^2)$, respectively, and the fusion rules theorem can be interpreted as a theorem about extending a restricted conformal block to a regular conformal block.  The fusion rules theorem had generalizations from various aspects in the theory of VOAs. For instance, Dong and Ren generalized it to the higher-level and modular representation case in \cite{DR13,DR14} using a bimodule over the higher-level Zhu algebra $A_N(V)$ \cite{DLM98(2)}; Huang and Yang generalized it to the logarithmic intertwining operator case in \cite{HY12}; Huang also gave an interpretation of the space of logarithmic intertwining operators case using a bimodule over his associative algebra $A^\infty(V)$ in \cite{H22,H24}. All of these generalizations were influenced by the idea of using a bimodule over certain variants of the associative algebra $A(V)$ to describe the space of intertwining operators.

		
In this paper, instead of the restricting the modules attached at $(\infty,0)$ to bottom degrees, we consider the case when only the module $(M^3)'$ attached at $\infty$ is restricted to its bottom-degree $\Om((M^3)')$ and investigate the space of conformal blocks associated to the datum  $$\Sigma(\Om((M^3)'), M^1, M^2)=(\P^1,(\infty, 1,0),(1/z,z-1,z), (\Om((M^3)'), M^1,M^2)).$$
    The space of conformal blocks associated to this one-point restricted datum turns out to be more interesting than the two-point restricted datum. It has an intimate relation with the tensor structure on $\Mod(A(V))$ for strongly rational VOAs $V$ and leads to a new fusion rules theorem.  
    
    

	To state our results precisely, and describe how they are proved, we set a small amount of notation. 
The chiral (current) Lie algebra $\mathcal{L}_{\P^1\bs \{\infty,1,0\}}(V)$ associated to the datum $(\P^1,\infty,1,0)$ defined by the VOA $V$ has spanning elements represented by $a\o f(z)$, where $f(z)$ is a rational function on $\P^1$ with poles at $\infty,1,0$. Relations between spanning elements are given by $L(-1)a\o f(z)=-a\o \frac{d}{dz}(f(z))$. 
The chiral Lie algebra $\mathcal{L}_{\P^1\bs \{\infty,1,0\}}(V)$ has natural actions on the $V$-modules $(M^3)'$, $M^1$, and $M^2$ given by the expansions of $f(z)$ around the poles.  
The space of {\em three-pointed conformal blocks on $\P^1$}, denoted by $\mathscr{C}\left(\Sigma((M^3)', M^1, M^2)\right)$, consists of linear functionals $f\in (M^3)'\o_\C M^1\o_\C M^2\ra \C$ that are invariant under the action of the chiral Lie algebra $\mathcal{L}_{\P^1\bs \{\infty,1,0\}}(V)$ on the tensor product \cite{Z94,NT05}.

		We define the {\em $\infty$-restricted chiral Lie algebra $\mathcal{L}_{\P^1\bs\{0,1\}}(V)_{\leq 0}$} to be the Lie subalgebra of $\mathcal{L}_{\P^1\bs \{\infty,1,0\}}(V)$ spanned by elements $a\o f(z)$ whose action on the tensor product leaves the subspace $\Om((M^3)') \o_\C M^1\o _\C M^2$ invariant (see Definition~\ref{def:inftyrestrictedchiralLie}). Then we define a {\em $\infty$-restricted conformal block} to be a linear functional $f:\Om((M^3)') \o_\C M^1\o _\C M^2\ra \C$ that is invariant under the action of $\mathcal{L}_{\P^1\bs\{0,1\}}(V)_{\leq 0}$. The space of such conformal blocks is denoted by $\mathscr{C}_{\mathrm{res}}\left(\Sigma(\Om((M^3)'), M^1, M^2)\right)$ (see Definition~\ref{def:restrictedcfb}). The Lie algebra $\mathcal{L}_{\P^1\bs\{0,1\}}(V)_{\leq 0}$ has an ideal $\mathcal{L}_{\P^1\bs\{0,1\}}(V)_{< 0}$ consisting of elements whose action vanishes on $\Om((M^3)')$, and such that the quotient algebra $\mathcal{L}_{\P^1\bs\{0,1\}}(V)_{\leq 0}/\mathcal{L}_{\P^1\bs\{0,1\}}(V)_{< 0}$ is isomorphic to the degree zero Borcherds Lie algebra $L(V)_0$ (see Lemma~\ref{lm:propertyofinftychiralLie}). Let $$M^1\odot M^2:=M^1\o_\C M^2/\mathcal{L}_{\P^1\bs\{0,1\}}(V)_{< 0}.(M^1\o_\C M^2),$$ which is a module over the Lie algebra $L(V)_0$. This quotient space also occurred in a construction of fusion tensor product of modules over $C_2$-cofinite VOAs by Tsuchiya and Wood (see \cite[(2.84)]{TW13}), wherein the chiral Lie algebra $\mathcal{L}_{\P^1\bs \{\infty,1,0\}}(V)$ is referred to as the current Lie algebra associated to $(\P^1, \infty,1,0)$, and is denoted by $\g^{\P}(V)$. 
		
		Using a set of spanning elements of $\mathcal{L}_{\P^1\bs\{0,1\}}(V)_{<0}$, we can show that $M^1\odot M^2$ is spanned by the equivalent classes $v_1\odot v_2$ of the elements $v_1\o v_2\in M^1\o_\C M^2$, subject to the following relations (see Definition~\ref{df:M1odotM2}): 
		
		\begin{align*}
			\sum_{j\geq 0}\binom{\wt a-1}{j} a(j-1)v_1\odot v_2&=v_1\odot \sum_{j\geq 0} a(\wt a-1+j)v_2,\\
			\sum_{j\geq 0} \binom{\wt a-k}{j} a(j)v_1\odot v_2&=-v_1\odot a(\wt a-k)v_2,\quad k\geq 2,
		\end{align*}
		where $a\in V$, $v_1\in M^1$, and $v_2\in M^2$. Although there is no $A(V)$-bimodule appear in $M^1\odot M^2$, it is closely related to the left $A(V)$-module $A(M^1)\o_{A(V)}\Om(M^2)$ in the fusion rules theorem \cite{FZ92}. The following is our first main theorem (see Theorem~\ref{thm:AVaction} and Corollary~\ref{corc:comparison}): 
		
		\begin{customthm}{A}\label{thm:A}
		Let $M^1$ and $M^2$ be $V$-modules. If either $M^1$ or $M^2$ is generated by its bottom degree, then $M^1\odot M^2$ is a left $A(V)$-module with respect to the following action:
			$$	A(V)\times \left(M^1\odot M^2\right)\ra M^1\odot M^2,\quad [a].(v_1\odot v_2)=(a\ast v_1-v_1\ast a)\odot v_2+ v_1\odot o(a)v_2,
			$$
			where  $[a]\in A(V)$, $v_1\in M^1$ and $v_2\in M^2$. In particular, if $M^2$ is generated by $\Om(M^2)$, then there is an epimorphism of left $A(V)$-modules $A(M^1)\o_{A(V)}\Om(M^2)\twoheadrightarrow M^1\odot M^2$. 
		\end{customthm}
		
		We prove Theorem~\ref{thm:A} by relating $O(V)$ and $O(M)$ in the definitions of $A(V)$ and $A(M)$ \cite{Z96,FZ92} with the properties of the Lie algebra  $\mathcal{L}_{\P^1\bs\{0,1\}}(V)_{\leq 0}$ and its ideal $\mathcal{L}_{\P^1\bs\{0,1\}}(V)_{< 0}$.  In general, the $L(V)_0$-module $M^1\odot M^2$ is a proper quotient of $A(M^1)\o _{A(V)}\Om(M^2)$ (see Section~\ref{Sec:5.4.2}).  But if the VOA $V$ is rational and $C_2$-cofinite, and $M^1,M^2$ are irreducible $V$-modules, then the $A(V)$-module $M^1\odot M^2$ is isomorphic to $A(M^1)\o _{A(V)}\Om(M^2)$, and they are both the bottom degree of Huang-Lepowsky's $P(z)$-tensor product module $M^1\bt_{P(z)}M^2$ (see Corollary~\ref{coro:bottomdegree}). 
		
		The $L(V)_0$-module $M^1\odot M^2$ can be used to describe the space of intertwining operators. The following is our second main theorem (see Theorem~\ref{thm:fusion}): 
		
		\begin{customthm}{B}\label{thm:B}
       Let $M^1,M^2,M^3$ be ordinary $V$-modules. Suppose that the contragredient module $(M^3)'$ is isomorphic to the generalized Verma module $\bar{M}(\Om((M^3)'))$ associated with the left $A(V)$-module $\Om((M^3)')$, and that $\Om((M^3)')^\ast\cong \Om(M^3)$. Assume further that either $M^1$ or $M^2$ is generated by its bottom degree $\Om(M^1)$ or $\Om(M^2)$, respectively. Then
	\begin{equation}\label{eq:fusionintro}
		I\fusion{M^1}{M^2}{M^3}\cong \Hom_{A(V)}(M^1\odot M^2,\Om(M^3)).
	\end{equation}
	In particular, if $V$ is rational, then \eqref{eq:fusion} holds for any irreducible $V$-modules $M^1,M^2$, and $M^3$.
		\end{customthm}
		
		In comparison with the Frenkel--Zhu fusion rules theorem, Theorem~\ref{thm:B} relaxes the condition that $M^2$ is a generalized Verma module, which makes an essential difference when the VOA $V$ is not rational. Moreover, if we only assume that $(M^3)'$ is generated by its bottom degree $\Om((M^3)')$ and $\Om((M^3)')^\ast\cong\Om(M^3)$, then we obtain an estimate of the fusion rule (see Lemma~\ref{lm:estimate}):
\[
	N\fusion{M^1}{M^2}{M^3}\leq \dim \Hom_{A(V)}(M^1\odot M^2,\Om(M^3)).
\]
This estimate is in general sharper than the usual estimate
$N\fusion{M^1}{M^2}{M^3}\leq \dim \Hom_{A(V)}(A(M^1)\o_{A(V)} \Om(M^2),\Om(M^3))$.
We use Li's example of modules over the universal Virasoro VOA \cite{Li99}
to illustrate this fact in Section~\ref{Sec:5.4.2}. 
		

		Instead of using the language of correlation functions, we use the language of conformal blocks to prove Theorem~\ref{thm:B}. The proof is much shorter than the correlation function arguments \cite{Liu23,GLZ23}. It is well-known that the space of intertwining operators is isomorphic to the space of conformal blocks $\mathscr{C}\left(\Sigma((M^3)', M^1, M^2)\right)$ \cite{FBZ04,NT05}. Using some basic facts about the representation theory of Lie algebras, we can also show that the space $\Hom_{A(V)}(M^1\odot M^2,\Om(M^3))$ can be identified with the space of $\infty$-restricted conformal blocks $\mathscr{C}_{\mathrm{res}}\left(\Sigma(\Om((M^3)'), M^1, M^2)\right)$ (see Proposition~\ref{prop:inftycfbHom}). Then Theorem~\ref{thm:B} is equivalent to showing that  $\mathscr{C}\left(\Sigma((M^3)', M^1, M^2)\right)\cong \mathscr{C}_{\mathrm{res}}\left(\Sigma(\Om((M^3)'), M^1, M^2)\right)$. The proof can be carried out using an explicit set of spanning elements of $\mathcal{L}_{\P^1\bs \{\infty,1,0\}}(V)$ (see Theorem~\ref{thm:main}). There is an alternative proof of a more general version of the isomorphism between restricted and unrestricted conformal blocks in \cite{GLZ24} using the Riemann-Roch theorem of algebraic curves. The proof we present in this paper is purely algebraic. 

One interesting consequence of Theorem~\ref{thm:B} is that it shows the contracted tensor product $\odot$ gives the fusion tensor product of $A(V)$-modules for strongly rational VOAs. Indeed, Dong--Li--Mason constructed a pair of adjoint functors
\[
(\bar{M}(-)\dashv \Om(-)):\Mod(A(V))\rightleftarrows \mathsf{Adm}(V),
\]
where $\bar{M}(-)$ is the generalized Verma module functor, and proved that this pair defines an adjoint equivalence of categories when $V$ is rational \cite{DLM98}. Huang--Lepowsky constructed a $P(z)$-tensor product $M^1\bt_{P(z)}M^2$ for VOA-modules using the theory of intertwining operators \cite{HL95}. Huang proved the associativity of the $P(z)$-tensor product for strongly rational VOAs by establishing the convergence of correlation functions defined by compositions of intertwining operators via differential equations \cite{H05}. Furthermore, $\mathsf{Adm}(V)$ is a modular tensor category when $V$ is strongly rational \cite{H08}. The modular tensor category structure on $\mathsf{Adm}(V)$ then transfers to a modular tensor category structure on $\Mod(A(V))$ under the equivalence functor $\Om(-)$. The following theorem is our final main result (see Theorem~\ref{coro:bottomdegree}).

\begin{customthm}{C}\label{mainC}
    Let $V$ be a strongly rational VOA, and let $M^1,M^2$ be irreducible $V$-modules. Then
$	
\Om(M^1\bt_{P(z)}M^2)\cong M^1\odot M^2
$	
	as left $A(V)$-modules, and $\bar{M}(M^1\odot M^2)\cong M^1\bt_{P(z)} M^2$ as $V$-modules. Moreover, the fusion tensor product of two irreducible $A(V)$-modules $S,T$ is given by
	\begin{equation}\label{coro1tensor}
		S\bt T\cong\bar{M}(S)\odot \bar{M}(T).
	\end{equation}
\end{customthm}
We verify the correctness of Theorem~\ref{mainC} using the standard examples of VOAs, including the level-one affine/lattice VOA $L_1(\sl_2)$, and the critical Ising model Virasoro VOA $L(\frac{1}{2},0)$. The fusion tensor products given by \eqref{coro1tensor} turn out to be both the correct and standard ones, see Proposition~\ref{prop:affinelattice} and Proposition~\ref{prop:Ising}.

		
		This paper is organized as follows. We first recall the basics of vertex operator algebras and related constructions, and then review the definition of three-pointed conformal blocks on $\P^1$ in Section~\ref{Sec:2}. In Section~\ref{Sec:3}, we introduce the notion of $\infty$-restricted chiral Lie algebras and discuss their basic properties and spanning elements. In Section~\ref{Sec:4}, using the results from the previous sections, we introduce the notions of $\infty$-restricted conformal blocks and the $L(V)_0$-module $M^1\odot M^2$, and prove Theorem~\ref{thm:A}. We prove Theorem~\ref{thm:B} in Section~\ref{sec:5} and discuss some of its consequences. In Section~\ref{sec:6}, we prove Theorem~\ref{mainC}. Finally, in Section~\ref{sec:7}, we present examples of the tensor product $M^1\odot M^2$ for standard examples of VOAs.


		

		
		

		
		
		\section{Space of VOA-conformal blocks on three-pointed $\P^1$}\label{Sec:2}
	Throughout this paper, we adopt the following conventions:
	\begin{itemize}
		\item All vector spaces and algebraic curves are defined over $\C$, the field of complex numbers.
		\item $\N$ represents the set of all natural numbers, including $0$.
	\end{itemize}
	
	In this section, we first review vertex operator algebras and related structures, and then introduce the notions of chiral Lie algebras ancillary to the three-pointed projective line $(\P^1,\infty,1,0)$ and the space of three-pointed conformal blocks using an algebraic language.
	
	\subsection{Preliminaries of VOAs}
	We recall the definitions of vertex operator algebras (VOAs) and related notions such as modules over VOAs, intertwining operators and fusion rules, contragredient modules, Borcherds' Lie algebra and Zhu's algebra, and generalized Verma modules over VOAs. These notions will be used later in this paper. We refer to \cite{B86,FLM88,FHL93,FZ92,Z96,DLM98,LL04} for more details.
	
	\subsubsection{Vertex operator algebras and modules}

 A VOA $V$ is said to be of {\em CFT-type} if $V=V_0\oplus V_+$, where $V_0=\C \vac$ and $V_+=\bigoplus_{n=1}^\infty V_n$. In this paper, unless stated otherwise, a VOA $V$ is always of CFT-type.
	
	\begin{df}
		Let $V$ be a VOA. An {\em admissible $V$-module} is a $\N$-graded vector space
$
		M=\bigoplus_{n=0}^\infty M(n),
$
		equipped with a linear map
		\[
		Y_M:V\ra \End(M)[\![z,z^{-1}]\!],\qquad
		Y_M(a,z)=\sum_{n\in \Z} a(n)z^{-n-1},
		\]
		called the {\em module vertex operator}, satisfying the following axioms:
		\begin{enumerate}
			\item (truncation property) For any $a\in V$ and $u\in M$, $Y_M(a,z)u\in M((z))$.
			\item (vacuum property) $Y_M(\vac,z)=\Id_M$.
			\item (Jacobi identity for $Y_M$) For any $a,b\in V$ and $u\in M$,
			\begin{equation}\label{eq:formalJacobi}
				\begin{aligned}
					z_0^{-1}\delta\left(\frac{z_1-z_2}{z_0}\right) &Y_M(a,z_1)Y_M(b,z_2)u
					- z_0^{-1}\delta\left(\frac{-z_2+z_1}{z_0}\right)Y_M(b,z_2)Y_M(a,z_1)u\\
					&= z_2^{-1}\delta\left(\frac{z_1-z_0}{z_2}\right) Y_M(Y(a,z_0)b,z_2)u.
				\end{aligned}
			\end{equation}
			\item ($L(-1)$-derivative property) $Y_M(L(-1)a,z)=\frac{d}{dz} Y_M(a,z)$ for any $a\in V$.
			\item (grading property) For any $a\in V$, $m\in \Z$, and $n\in \N$,
			\[
			a(m)M(n)\subseteq M(n+\wt a-m-1),
			\]
			i.e., $\wt(a(m))=\wt a-m-1$.
		\end{enumerate}
		We write $\deg v=n$ if $v\in M(n)$, and call it the {\em degree} of $v$. Denote the category of admissible $V$-modules by $\mathsf{Adm}(V)$.
		
		An admissible $V$-module $M$ is called {\em ordinary} if each degree-$n$ subspace $M(n)$ is a finite-dimensional eigenspace of $L(0)$ with eigenvalue $h+n$, where $h\in \C$ is called the {\em conformal weight} of $M$. In particular, if we write $L(0)v=(\wt v)\cdot v$ for $v\in M(n)$, then $\wt v=\deg v+h$.
		
		We abbreviate an {\em irreducible ordinary $V$-module} simply as an {\em irreducible $V$-module}. Submodules, quotient modules, and irreducible modules are defined in the usual categorical sense. The VOA $V$ is called {\em rational} if the category of admissible $V$-modules is semisimple \cite{DLM98}.
	\end{df}
	
	\begin{remark}
		We recall the following well-known facts about the Jacobi identity:
		\begin{enumerate}
			\item Using Cauchy's integral (or residue) theorem, one can rewrite the formal-variable Jacobi identity \eqref{eq:formalJacobi} into the {\em residue form} \cite{FLM88,FZ92,Z96}:
			\begin{equation}\label{eq:resJacobi}
				\begin{aligned}
					&\Res_{z=0} Y_{M}(a,z)Y_M(b,w)\,\iota_{z,w}(F(z,w))
					-\Res_{z=0} Y_M(b,w)Y_M(a,z)\,\iota_{w,z}(F(z,w))\\
					&=\Res_{z-w=0} Y_M(Y(a,z-w)b,w)\,\iota_{w,z-w}(F(z,w)),
				\end{aligned}
			\end{equation}
			where $F(z,w)=z^n w^m (z-w)^l$ with $m,n,l\in \Z$, and $\iota_{z,w}$, $\iota_{w,z}$, and $\iota_{w,z-w}$ are the expansion operations of a rational function in complex variables $z$ and $w$ on the domains $|z|>|w|$, $|w|>|z|$, and $|w|>|z-w|$, respectively.
			\item The Jacobi identity \eqref{eq:resJacobi} has the following component form:
			\begin{equation}\label{borcherds}
				\begin{aligned}
					&\sum_{i=0}^\infty \binom{l}{i} a(m+l-i) b(n+i)
					-\sum_{i=0}^\infty (-1)^{l+i} \binom{l}{i} b(n+l-i) a(m+i)\\
					&=\sum_{i=0}^\infty (a(l+i)b)(m+n-i),
				\end{aligned}
			\end{equation}
			where $a,b\in V$ and $m,n,l\in \Z$. This identity is also called the {\em Borcherds identity} \cite{B86}.
		\end{enumerate}
	\end{remark}

		
	\subsubsection{Intertwining operators and fusion rules}
	Intertwining operators among $V$-modules are generalizations of intertwining operators among modules over Lie algebras \cite{FHL93,FZ92,Li98}:
	\begin{df}\label{def:IO}
		Let $M^1,M^2,M^3$ be ordinary $V$-modules of conformal weights $h_1,h_2,h_3\in \C$, respectively. Let $h:=h_1+h_2-h_3$. An {\em intertwining operator of type $\fusion{M^1}{M^2}{M^3}$} is a linear map
		\[
		I(\cdot,z):M^1\ra \Hom(M^2,M^3)[\![z,z^{-1}]\!]z^{-h},\qquad
		I(v_1,z)=\sum_{n\in \Z}v_1(n)z^{-n-1-h},
		\]
		satisfying the following axioms:
		\begin{enumerate}
			\item (truncation property) For any $v_1\in M^1$ and $v_2\in M^2$, $v_1(n)v_2=0$ for $n\gg 0$.
			\item ($L(-1)$-derivative property) For any $v_1\in M^1$,
			$I(L(-1)v_1,z)=\frac{d}{dz} I(v_1,z)$.
			\item (Jacobi identity) For any $v_1\in M^1$, $v_2\in M^2$, and $a\in V$, one has
			\begin{align*}
				z_0^{-1}\delta\left(\frac{z_1-z_2}{z_0}\right) &Y_{M^3}(a,z_1)I(v_1,z_2)v_2
				- z_0^{-1}\delta\left(\frac{-z_2+z_1}{z_0}\right)I(v_1,z_2)Y_{M^2}(a,z_1)v_2\\
				&=z_2^{-1}\delta\left(\frac{z_1-z_0}{z_2}\right) I(Y_{M^1}(a,z_0)v_1,z_2)v_2.
			\end{align*}
		\end{enumerate}
		The vector space of intertwining operators of type $\fusion{M^1}{M^2}{M^3}$ is denoted by $I\fusion{M^1}{M^2}{M^3}$. Its dimension, denoted by $N\fusion{M^1}{M^2}{M^3}$, is called the {\em fusion rule among $M^1,M^2$, and $M^3$}.
	\end{df}
	
	Using the Jacobi identity and the $L(0)$-eigenspace property for $M^i(n)$, one can easily show that
$
	v_1(n)M^2(m)\subseteq M^3(\deg v_1-n-1+m),
$
	for any $v_1\in M^1$, $n\in \Z$, and $m\in \N$ \cite{FZ92}.
	
	\subsubsection{Contragredient modules}
	
	\begin{df}\cite{FHL93}
		Let $M$ be an ordinary $V$-module. Its {\em contragredient module} is the graded dual space
$
		M'=\bigoplus_{n=0}^\infty M(n)^\ast,
$
		with $Y_{M'}:V\ra \End(M')[\![z,z^{-1}]\!]$ given by
		\begin{equation}\label{eq:contra}
			\<Y_{M'}(a,z)v',v\>
			:=\<v',Y_M(e^{zL(1)}(-z^{-2})^{L(0)}a,z^{-1})v\>
			=\<v',Y'_M(a,z)v\>,
		\end{equation}
		where $\<\cdot,\cdot\>:M'\times M\ra \C$ is the natural pairing between graded vector spaces.
	\end{df}
	
	It was proved in \cite[Section 5]{FHL93} that $Y_{M'}$ defined by \eqref{eq:contra} satisfies the Jacobi identity \eqref{eq:formalJacobi}, and that $Y''_M(a,z)=Y_M(a,z)$.
	
	Moreover, if we write $Y'_{M}(a,z)=\sum_{n\in \Z}a'(n) z^{-n-1}$, then by taking the formal residue,
	\begin{equation}\label{eq:a'n}
		a'(n)=\sum_{j\geq 0}\frac{(-1)^{\wt a}}{j!} (L(1)^j a)(2\wt a-n-j-2).
	\end{equation}
	It follows that $a'(n)M(m)\subseteq M(-\wt a+n+1+m)$, i.e., $\wt (a'(n))=-\wt a+n+1$.

	\subsubsection{Borcherds' Lie algebra and Zhu's algebra associated to a VOA}\label{sec:2.1.4}
	Let $V$ be a VOA. Then $\hat{V}=V\otimes_\C \C[t,t^{-1}]$ is the tensor product vertex algebra \cite{FHL93}, with vacuum element $\vac\otimes 1$ and differential
$
	\nabla=L(-1)\otimes \Id+\Id\otimes \frac{d}{dt};
$
	see \cite{B86} for more details.
	
	\begin{df}\cite{B86}\label{def:BorLie}
		Let $V$ be a VOA. The Borcherds Lie algebra $L(V)$ is defined by
		\[
		L(V)=\hat{V}/\nabla (\hat{V})
		=\spn\{a_{[n]}=a\otimes t^n+\nabla (\hat{V}) :a\in V, n\in \Z\},
		\]
		with $(L(-1)a)_{[n]}=-n a_{[n-1]}$, and Lie bracket
		\begin{equation}\label{eq:bracketborcherdslie}
			[a_{[m]},b_{[n]}]=\sum_{j\geq 0}\binom{m}{j} (a(j)b)_{[m+n-j]},
			\quad a,b\in V,\ m,n\in \Z.
		\end{equation}
		
	\end{df}
	
	\begin{df}\cite{Z96}\label{def:AV}
		Let $V$ be a VOA. The {\em Zhu algebra $A(V)$} is defined by $A(V)=V/O(V)$, where
		\begin{equation}\label{eq:O(V)}
			O(V)=\spn\left\{ a\circ b=\Res_{z=1}Y(a,z-1)b\,\iota_{1,z-1}\left(\frac{z^{\wt a}}{(z-1)^2}\right): a,b\in V\right\}.
		\end{equation}
		The space $A(V)=\spn\{[a]=a+O(V):a\in V\}$ is an associative algebra with respect to
		\begin{equation}\label{eq:prodAV}
			[a]\ast [b]
			=\Res_{z=1} [Y(a,z-1)b]\iota_{1,z-1}\left(\frac{z^{\wt a}}{z-1}\right)
			=\sum_{j\geq 0}\binom{\wt a}{j} [a(j-1)b],
		\end{equation}
		for $a,b\in V$.
		The category of (left) $A(V)$-modules is denoted by $\mathsf{Mod}(A(V))$.
	\end{df}
	
	We remark the following facts about the Zhu algebra and the Borcherds Lie algebra; see \cite{B86,Z96,DLM98} for more details.
	\begin{enumerate}
		\item There is an anti-involution $\theta: A(V)\ra A(V)$ defined by
		$\theta([a])=[e^{L(1)}(-1)^{L(0)}a]$, with
		$\theta([a]\ast [b])=\theta([b])\ast \theta([a])$.
		\item There is a similar anti-involution $\theta: L(V)\ra L(V)$ defined by
		$\theta (a_{[n]})=a'_{[n]}$ via \eqref{eq:a'n}, with
		$\theta([a_{[m]},b_{[n]}])=[\theta(b_{[n]}),\theta(a_{[m]})]$.
		\item Let $M=\bigoplus_{n=0}^\infty M(n)$ be an admissible $V$-module (or, more generally, a weak $V$-module). Then the subspace of ``highest-weight vectors''
		\[
		\Om(M)=\{v\in M: a(n)v=0,\ \deg(a(n))=\wt a-n-1<0\}
		\]
		is a left $A(V)$-module via $A(V)\ra \End(\Om(M))$, $[a]\mapsto o(a):=a(\wt a-1)$.
		\item If $M$ is an irreducible admissible $V$-module, then $\Om(M)$ is an irreducible $A(V)$-module, which is also the bottom degree $M(0)$ of $M$, i.e., $\Om(M)=M(0)$. Throughout this paper, we refer to $\Om(M)$ as the {\em bottom degree} of the $V$-module $M$.
		\item Let $M$ be an ordinary $V$-module, and let $M'$ be its contragredient module.
		The bottom degree $M'(0)=M(0)^\ast$ is naturally a right $A(V)$-module. It is a left $A(V)$-module via the anti-involution. $A(V)\times M(0)^\ast\ra  M(0)^\ast, ([a],v')\mapsto [a].v'$, where
		\begin{equation}\label{eq:contraleftA(V)}
			\braket{[a]. v'}{v}
			:=\braket{v'.[\theta(a)]}{v}
			=\braket{v'}{[\theta(a)].v}
			=\braket{v'}{o(\theta(a))v},
		\end{equation}
		where $[a]\in A(V)$, $v'\in M(0)^\ast$, and $v\in M(0)$. 	If $M$ is an irreducible ordinary module, then $M'$ is also irreducible and $\Om(M')=M(0)^\ast$. 
		\item There is an epimorphism of Lie algebras
		\begin{equation}\label{eq:LieandAV}
			L(V)_0\twoheadrightarrow A(V)_{\mathrm{Lie}},\qquad
			a_{[\wt a-1]}\mapsto [a],\quad a\in V.
		\end{equation}
	\end{enumerate}
	
	\subsubsection{Generalized Verma module associated to an $A(V)$-module}\label{sec:2.1.5}
	Let $U$ be an irreducible left $A(V)$-module. Then $U$ is a module over the Lie algebra $L(V)_0$ via \eqref{eq:LieandAV}, which can be lifted to a module over $L(V)_0\oplus L(V)_+$ by letting $(L(V)_+)\cdot U=0$.
	
	Consider the induced module
$
	M(U):=U(L(V))\otimes_{U(L(V)_0\oplus L(V)_+)} U
	=U(L(V)_-)\otimes_\C U.
$
	Let $J$ be the $U(L(V))$-submodule of $M(U)$ generated by the coefficients of the Borcherds identity \eqref{borcherds}. It was proved in \cite{DLM98} that
$
	\bar{M}(U):=M(U)/J
$
	is an admissible $V$-module generated by $U$, with bottom degree $\bar{M}(U)(0)=U$. The module vertex operator is given by
	\[
	Y_{\bar{M}(U)}(a,z)=\sum_{n\in \Z} a(n)z^{-n-1}
	=\sum_{n\in \Z} a_{[n]}z^{-n-1},
	\]
	for any $a\in V$. Moreover, any admissible $V$-module $W$ generated by $W(0)=U$ is a quotient module of $\bar{M}(U)$.
	
	If we view $\bar{M}(-)$ and $\Om(-)$ as functors, then they form an adjoint pair of functors \cite[Theorem 6.1]{DLM98}:
	\[
	(\bar{M}(-)\dashv \Om(-)):\Mod(A(V))\rightleftarrows \mathsf{Adm}(V).
	\]
Moreover, we have the following fact about rationality:  
\begin{lm}\cite[Theorem 8.1]{DLM98}.\label{lmrationality}
A VOA $V$ is rational if and only if $A(V)$ is semisimple and $(\bar{M}(-)\dashv \Om(-)):\Mod(A(V))\rightleftarrows \mathsf{Adm}(V)$ is an adjoint equivalence.   
\end{lm}

	Let $M$ be an irreducible $V$-module. Then $\Om(M')=M(0)^\ast=\Om(M)^\ast$ is an irreducible left $A(V)$-module via \eqref{eq:contraleftA(V)}.
	The anti-involution $\theta: L(V)\ra L(V)$ of the Borcherds Lie algebra induces an isomorphism of associative algebras
	\[
	\theta: U(L(V))\ra U(L(V))^{\mathrm{op}}.
	\]
	Note that
$
	\theta(b_r(n_r)\cdots b_1(n_1))=b'_1(n_1)\cdots b'_r(n_r),
$
	where $b'(n)$ is given by \eqref{eq:a'n}, and
$
	\bar{M}(\Om(M'))
	=U(L(V))\cdot\Om(M')
	=\theta(U(L(V)))\cdot\Om(M').
$
	It follows that
	\begin{equation}\label{eq:spanverma}
		\bar{M}(\Om(M'))
		=\spn\{b'_1(n_1)\cdots b'_r(n_r)v'
		:r\geq 0,\ b_i\in V,\ n_i\in \Z,\ v'\in \Om(M') \}.
	\end{equation}
	
	Moreover, by carefully choosing the coefficients in the Borcherds identity \eqref{borcherds} for the vertex operator $Y_{\bar{M}(U)}$, one can also show that
	\[
	\bar{M}(\Om(M'))
	=\spn\{b'(n)v':b\in V,\ n\in \Z,\ v'\in \Om(M') \}
	\]
	\cite{LL04}. These facts will be used in Section~\ref{sec:5}.

	\subsection{The chiral Lie algebra ancillary to $(\P^1,\infty,1,0)$}
	The chiral Lie algebra $\mathcal{L}_{C\bs P_\bullet}(V)$ ancillary to a VOA $V$ and a stable $n$-pointed curve $(C,P_\bullet)$ was defined as
	$H^0(C\bs P_\bullet, \mathscr{V}_C\o \Omega_C/\mathrm{Im}\,\nabla)$; see \cite{FBZ04,BD04,DGT24}.
	In this paper, we are only interested in the smooth curve $C=\P^1$ with three marked points $P_\bullet=(\infty,1,0)$.
	In this case, the chiral Lie algebra $\mathcal{L}_{\P^1\bs\{\infty,1,0\}}(V)$ has the following simple description:
	
	\begin{df}\cite{NT05,DGT24}\label{def:chiralLie}
		The chiral (current) Lie algebra $\mathcal{L}_{\P^1\bs\{\infty,1,0\}}(V)$ ancillary to the smooth curve $\P^1\bs\{\infty,1,0\}$
		is defined by
		\begin{equation}\label{eq:defofchiralLie}
			\mathcal{L}_{\P^1\bs\{\infty,1,0\}}(V)
			:=\left(V\o \C[z^{\pm 1}, (z-1)^{\pm 1}]\right)/\mathrm{Im}\,\nabla,
		\end{equation}
		where $\nabla=L(-1)\o\Id+\Id\o \frac{d}{dz}$.
		Then, as a vector space,
		\begin{equation}\label{spnofchiralLie}
			\mathcal{L}_{\P^1\bs\{\infty,1,0\}}(V)
			=\spn\left\{a\o \frac{z^n}{(z-1)^m}: a\in V,\ m,n\in \Z\right\},
		\end{equation}
		with
$
		(L(-1)a)\o \frac{z^n}{(z-1)^m}
		=-a\o \frac{d}{dz}\left(\frac{z^n}{(z-1)^m}\right),
$
		where we use the same symbol for the equivalence class of
		$a\o \frac{z^n}{(z-1)^m}\in V\o \C[z^{\pm 1}, (z-1)^{\pm 1}]$ in the quotient space.
		The Lie bracket on $\mathcal{L}_{\P^1\bs\{\infty,1,0\}}(V)$ is given by
		\begin{equation}\label{eq:chiralLiebracket}
			\left[a\o \frac{z^n}{(z-1)^m},\, b\o \frac{z^s}{(z-1)^t}\right]
			=\sum_{i\geq 0}\sum_{j\geq 0} \binom{n}{i} \binom{-m}{j}\,
			a_{i+j} b\o \frac{z^{n+s-i}}{(z-1)^{m+t+j}},
		\end{equation}
		where $a,b\in V$ and $m,n,s,t\in \Z$.
	\end{df}
	
\begin{remark}
The original definition of the three-point chiral (current) Lie algebra on $\P^1$ by Nagatomo--Tsuchiya \cite{NT05} was
\[
\mathcal{L}_{\P^1\bs\{\infty,1,0\}}(V)
=\frac{\bigoplus_{\Delta=0}^\infty V_{\Delta}\o H^0(\P^1, \Om^{1-\Delta}(\ast(\infty,1,0)))}
{\nabla\bigoplus_{\Delta=0}^\infty V_{\Delta}\o H^0(\P^1, \Om^{-\Delta}(\ast(\infty,1,0)))}.
\]
Thus, elements in $\mathcal{L}_{\P^1\bs\{\infty,1,0\}}(V)$ are represented by $a\o f(z)\,(dz)^{1-\Delta}$, where $f(z)$ is a global meromorphic function on $\P^1$ with poles only at $\infty$, $1$, and $0$.
In this paper, since we focus only on the algebraic side of the theory, we omit the formal differential symbol $(dz)^{1-\Delta}$ in our definition~\eqref{eq:defofchiralLie} and \eqref{spnofchiralLie}.

\end{remark}
	
	The following proposition is a purely algebraic version of the chiral Lie algebra action on the space of coinvariants \cite{FBZ04,NT05,DGT24}.
	The proof is an immediate consequence of the Jacobi identity of VOAs together with \eqref{eq:chiralLiebracket}, so we omit it.
	
	\begin{prop}\label{prop:chiralLieaction}
		Let $M^1,M^2$, and $M^3$ be $V$-modules, and let
		$a\o\frac{z^n}{(z-1)^m}\in \mathcal{L}_{\P^1\bs\{\infty,1,0\}}(V)$.
		Then:
		\begin{enumerate}
			\item $(M^3)'$ is a module over the Lie algebra $\mathcal{L}_{\P^1\bs\{\infty,1,0\}}(V)$ via
			$\rho_\infty:\mathcal{L}_{\P^1\bs\{\infty,1,0\}}(V)\ra \gl((M^3)')$,
			\begin{equation}\label{eq:inftyaction}
				\begin{aligned}
					\rho_\infty\left(a\o\frac{z^n}{(z-1)^m}\right)(v'_3)
					&=\Res_{z=\infty} Y_{(M^3)'}(\vartheta(a),z^{-1})v'_3\,
					\iota_{z,1}\left(\frac{z^n}{(z-1)^m}\right)\\
					&=-\sum_{j\geq 0}\binom{-m}{j} (-1)^j a'(n-m-j)v'_3,\quad v'_3\in (M^3)',
				\end{aligned}
			\end{equation}
			where $\vartheta(a)=e^{zL(1)}(-z^{-2})^{L(0)}(a)$, and $a'(k)$ is given by \eqref{eq:a'n}.
			\item $M^1$ is a module over the Lie algebra $\mathcal{L}_{\P^1\bs\{\infty,1,0\}}(V)$ via
			$\rho_1:\mathcal{L}_{\P^1\bs\{\infty,1,0\}}(V)\ra \gl(M^1)$,
			\begin{equation}\label{eq:1action}
				\begin{aligned}
					\rho_1\left(a\o \frac{z^n}{(z-1)^m}\right)(v_1)
					&=\Res_{z=1} Y_{M^1}(a,z-1)v_1\,
					\iota_{1,z-1}\left(\frac{z^n}{(z-1)^m}\right)\\
					&=\sum_{j\geq 0} \binom{n}{j} a(j-m)v_1,\quad v_1\in M^1.
				\end{aligned}
			\end{equation}
			\item $M^2$ is a module over the Lie algebra $\mathcal{L}_{\P^1\bs\{\infty,1,0\}}(V)$ via
			$\rho_0:\mathcal{L}_{\P^1\bs\{\infty,1,0\}}(V)\ra \gl(M^2)$,
			\begin{equation}\label{eq:0action}
				\begin{aligned}
					\rho_{0}\left(a\o\frac{z^n}{(z-1)^m}\right)(v_2)
					&=\Res_{z=0} Y_{M^2}(a,z)v_2\,
					\iota_{1,z}\left(\frac{z^n}{(z-1)^m}\right)\\
					&=\sum_{j\geq 0} \binom{-m}{j}(-1)^{-m-j} a(n+j)v_2,\quad v_2\in M^2.
				\end{aligned}
			\end{equation}
		\end{enumerate}
		
		In particular, $(M^3)'\o_\C M^1\o_\C M^2$ is a tensor product module over the chiral Lie algebra
		$\mathcal{L}_{\P^1\bs\{\infty,1,0\}}(V)$, with module action given by
		\begin{equation}\label{eq:actionofchiralliealgebraelements}
			\begin{aligned}
				&\left(a\o \frac{z^n}{(z-1)^m}\right)\cdot (v'_3\o v_1\o v_2)\\
				&=-\sum_{j\geq 0}\binom{-m}{j} (-1)^j a'(n-m-j)v'_3\o v_1\o v_2
				+\sum_{j\geq 0} \binom{n}{j}v'_3\o a(j-m)v_1\o v_2\\
				&\ \ \ +\sum_{j\geq 0} \binom{-m}{j}(-1)^{-m-j} v'_3\o v_1\o a(n+j)v_2,
			\end{aligned}
		\end{equation}
		where $v'_3\o v_1\o v_2\in (M^3)'\o_\C M^1\o_\C M^2$.
	\end{prop}
	
	\subsection{Space of three-pointed conformal blocks on $\P^1$}
	
	Consider the following datum:
	\begin{equation}\label{eq:datum}
		\Sigma((M^3)', M^1, M^2)=\left(\P^1, (\infty, 1,0),(1/z, z-1,z), ((M^3)',M^1,M^2)\right),
	\end{equation}
	where $1/z$, $z-1$, and $z$ are the local coordinates around the points $\infty,1$, and $0$ on $\P^1$, respectively.
	The contragredient module $(M^3)'$ is attached to $\infty$, and the $V$-modules $M^1$ and $M^2$ are attached to the points $1$ and $0$, respectively.
	Recall that $(M^3)'\o_\C M^1\o_\C M^2$ is a module over the chiral Lie algebra $\mathcal{L}_{\P^1\bs\{\infty,1,0\}}(V)$ via \eqref{eq:actionofchiralliealgebraelements}.
	
	\begin{df}\label{def:cfb}\cite{NT05,FBZ04,DGT24}
		Let $V$ be a VOA, and let $M^1,M^2$, and $M^3$ be $V$-modules. The quotient space
		\begin{equation}\label{eq:defcoinv}
			\mathbb{V}\left(\Sigma((M^3)', M^1, M^2)\right)
			:=\frac{(M^3)'\otimes_\C M^1\otimes_\C M^2}
			{\mathcal{L}_{\P^1\bs\{\infty,1,0\}}(V)\cdot\left((M^3)'\otimes_\C M^1\otimes_\C M^2\right)}
		\end{equation}
		is called the {\em space of coinvariants} associated to the datum $\Sigma((M^3)', M^1, M^2)$.
		The dual space
		\begin{equation}\label{eq:defofconformalblocks}
			\mathscr{C}\left(\Sigma((M^3)', M^1, M^2)\right)
			:= \left(\mathbb{V}\left(\Sigma((M^3)', M^1, M^2)\right)\right)^\ast
		\end{equation}
		is called the {\em space of conformal blocks} associated to $\Sigma((M^3)', M^1, M^2)$.
		We refer to an element $\varphi \in \mathscr{C}\left(\Sigma((M^3)', M^1, M^2)\right)$ as a {\em (three-pointed) conformal block} associated to the datum $\Sigma((M^3)', M^1, M^2)$.
	\end{df}
	
	One can replace the marked point $1$ on $\P^1$ by another point $w\in \C^\times$ on the same chart containing $0$, and define the chiral Lie algebra $\mathcal{L}_{\P^1\bs\{\infty,w,0\}}(V)$ and its actions in the same way as in Definition~\ref{def:chiralLie} and Proposition~\ref{prop:chiralLieaction}, with $1$ replaced by $w$.
	In particular,
	\begin{equation}\label{eq:defwchiral}
		\mathcal{L}_{\P^1\bs\{\infty,w,0\}}(V)=\spn\left\{a\o \frac{z^n}{(z-w)^m}: a\in V,\ m,n\in \Z\right\},
	\end{equation}
	with $
	(L(-1)a)\o \frac{z^n}{(z-w)^m}
	=-a\o \frac{d}{dz}\left(\frac{z^n}{(z-w)^m}\right),$
	and the module action of $\mathcal{L}_{\P^1\bs\{\infty,w,0\}}(V)$ on $(M^3)'\o_\C M^1\o_\C M^2$ is given by
	\begin{equation}\label{eq:wchiralaction}
		\begin{aligned}
			&\left(a\o \frac{z^n}{(z-w)^m}\right)\cdot (v'_3\o v_1\o v_2)\\
			&=-\sum_{j\geq 0}\binom{-m}{j} (-1)^j w^j a'(n-m-j)v'_3\o v_1\o v_2
			+\sum_{j\geq 0} \binom{n}{j} w^{n-j}v'_3\o a(j-m)v_1\o v_2\\
			&\ \ \ +\sum_{j\geq 0} \binom{-m}{j}(-1)^{-m-j} w^{-m-j}v'_3\o v_1\o a(n+j)v_2.
		\end{aligned}
	\end{equation}
	
	Similar to \eqref{eq:defofconformalblocks}, for the datum
	\[
	\Sigma_{w}((M^3)', M^1, M^2):=\left(\P^1, (\infty, w,0),(1/z, z-w,z), ((M^3)',M^1,M^2)\right),
	\]
	we can define the space of conformal blocks $\mathscr{C}\left(\Sigma_{w}((M^3)', M^1, M^2)\right)$ as the vector space of linear functionals on $(M^3)'\o_\C M^1\o_\C M^2$ that are invariant under the action of the chiral Lie algebra $\mathcal{L}_{\P^1\bs\{\infty,w,0\}}(V)$.
	The conformal blocks associated to $\Sigma((M^3)', M^1, M^2)$ and $\Sigma_{w}((M^3)', M^1, M^2)$ are related by the following formula; see \cite[(eq.~4.11)]{GLZ24}:
	\begin{equation}\label{eq:cfbw1iso}
		\begin{aligned}
			\mathscr{C}\left(\Sigma((M^3)', M^1, M^2)\right)&\xrightarrow{\simeq}
			\mathscr{C}\left(\Sigma_{w}((M^3)', M^1, M^2)\right),\quad \varphi_1\mapsto \varphi_w,\\
			\braket{\varphi_w}{v'_3\o v_1\o v_2}&=\braket{\varphi_1}{w^{L(0)-h_3}v'_3\o w^{-L(0)+h_1}v_1\o w^{-L(0)+h_2}v_2}.
		\end{aligned}
	\end{equation}
	
	It is well known that there is a one-to-one correspondence between three-pointed conformal blocks and intertwining operators of VOAs; see \cite{TUY89,Z94,NT05,GLZ24}:
	
	\begin{prop}\label{prop:IOconformalblocks}
		Let $M^1,M^2$, and $M^3$ be $V$-modules. Then there is an isomorphism of vector spaces
		$\mathscr{C}\left(\Sigma((M^3)', M^1, M^2)\right) \cong I\fusion{M^1}{M^2}{M^3}$.
		In particular, the fusion rule $N\fusion{M^1}{M^2}{M^3}$ is equal to the dimension of the space of three-pointed conformal blocks on $\P^1$.
	\end{prop}


		\section{Restriction of the chiral Lie algebra at $\infty$}\label{Sec:3}
		
		In this Section, we introduce the notion of $\infty$-restricted chiral Lie algebra $\mathcal{L}_{\P^1\bs\{0,1\}}(V)_{\leq 0}$
		and its augmented ideal $\mathcal{L}_{\P^1\bs\{0,1\}}(V)_{< 0}$. We discuss their basic properties and give a short list of spanning elements of these Lie algebras. These properties will used to define the $\infty$-restricted conformal blocks in the next Section.

		\subsection{Spanning elements of $\infty$-restricted chiral Lie algebra}	
Let $f(z)=\frac{z^n}{(z-1)^m}$, with $m,n\in \Z$, be a global meromorphic function on $\P^1$ with poles at $0$, $1$, and $\infty$. Then
\[
\iota_{z,1}(f(z))=\sum_{j\geq 0}\binom{-m}{j} (-1)^j z^{n-m-j},
\]
and hence the order of the pole at $\infty$ is
$
\ord_\infty(f(z))=\ord_{0} f(1/z)=n-m.$

\begin{df}\label{def:inftyrestrictedchiralLie}
	Consider the following subspace of the chiral Lie algebra $\mathcal{L}_{\P^1\bs\{\infty,1,0\}}(V)$:
	\begin{equation}\label{eq:defofinftyrestrictchiral}
		\begin{aligned}
			\mathcal{L}_{\P^1\bs\{0,1\}}(V)_{\leq 0}
			&=\spn\{a\o f(z)\in \mathcal{L}_{\P^1\bs\{\infty,1,0\}}(V)
			: a\in V,\ \ord_\infty (f(z))\leq \wt a-1 \}\\
			&=\spn\left\{a\o\frac{z^{n}}{(z-1)^m}\in \mathcal{L}_{\P^1\bs\{\infty,1,0\}}(V)
			: a\in V,\ n-m\leq \wt a-1 \right\}.
		\end{aligned}
	\end{equation}
	We call this the {\em $\infty$-restricted chiral Lie algebra}.
\end{df}

\begin{lm}\label{lm:inftychiralLiestab}
	The following properties hold for $\mathcal{L}_{\P^1\bs\{0,1\}}(V)_{\leq 0}$ defined in \eqref{eq:defofinftyrestrictchiral}.
	\begin{enumerate}
		\item $\mathcal{L}_{\P^1\bs\{0,1\}}(V)_{\leq 0}$ is a Lie subalgebra of $\mathcal{L}_{\P^1\bs\{\infty,1,0\}}(V)$.
		\item $\mathcal{L}_{\P^1\bs\{0,1\}}(V)_{\leq 0}
		\subseteq \mathrm{Stab}_{(M^3)'\otimes_\C M^1\otimes_\C M^2}
		(\Om((M^3)')\o_\C M^1\o_\C M^2)$.
		That is,
		\[
		\left(\mathcal{L}_{\P^1\bs\{0,1\}}(V)_{\leq 0}\right)
		\cdot(\Om((M^3)')\o_\C M^1\o_\C M^2)
		\subseteq \Om((M^3)')\o_\C M^1\o_\C M^2.
		\]
	\end{enumerate}
\end{lm}

\begin{proof}
	Let $a\o \frac{z^{n}}{(z-1)^m}$ and $b\o \frac{z^s}{(z-1)^t}\in \mathcal{L}_{\P^1\bs\{1,0\}}(V)_{\leq 0}$, with $n-m\leq \wt a-1$ and $s-t\leq \wt b-1$.
	By \eqref{eq:chiralLiebracket},
	$[a\o \frac{z^n}{(z-1)^m},\, b\o \frac{z^s}{(z-1)^t}]
	=\sum_{i\geq 0}\sum_{j\geq 0} \binom{n}{i} \binom{-m}{j}
	a_{i+j} b\o \frac{z^{n+s-i}}{(z-1)^{m+t+j}},$
	with
	\begin{equation}\label{eq:estimate}
		\begin{aligned}
			(n+s-i)-(m+t+j)
			&=(n-m)+(s-t)-i-j
			\leq \wt a-1+\wt b-1-i-j\\
			&=\wt (a_{i+j}b)-1,
		\end{aligned}
	\end{equation}
	for all $i,j\geq 0$.
	Thus,
	$[a\o \frac{z^n}{(z-1)^m},\, b\o \frac{z^s}{(z-1)^t}]
	\in \mathcal{L}_{\P^1\bs\{1,0\}}(V)_{\leq 0}$,
	and $\mathcal{L}_{\P^1\bs\{1,0\}}(V)_{\leq 0}$ is a Lie subalgebra.
	
	Let $v'_3\o v_1\o v_2\in \Om((M^3)')\o_\C M^1\o_\C M^2$, and
	$a\o \frac{z^{n}}{(z-1)^m}\in \mathcal{L}_{\P^1\bs\{1,0\}}(V)_{\leq 0}$.
	Since
	$$\deg a'(n-m-j)=-\wt a+n-m-j+1\leq -j\leq 0$$
	for any $j\geq 0$, in view of \eqref{eq:defofinftyrestrictchiral}, it follows that
	$\sum_{j\geq 0}\binom{-m}{j} (-1)^j a'(n-m-j)v'_3\in \Om((M^3)')$.
	Then, by \eqref{eq:actionofchiralliealgebraelements}, we have
	\begin{align*}
		&\left(a\o \frac{z^n}{(z-1)^m}\right)\cdot (v'_3\o v_1\o v_2)\\
		&=-\sum_{j\geq 0}\binom{-m}{j} (-1)^j a'(n-m-j)v'_3\o v_1\o v_2
		+\sum_{j\geq 0} \binom{n}{j}v'_3\o a(j-m)v_1\o v_2\\
		&\ \ \ +\sum_{j\geq 0} \binom{-m}{j}(-1)^{-m-j} v'_3\o v_1\o a(n+j)v_2
		\in \Om((M^3)')\o_\C M^1\o_\C M^2.
	\end{align*}
	This shows
	$\mathcal{L}_{\P^1\bs\{0,1\}}(V)_{\leq 0}
	\ssq \mathrm{Stab}_{(M^3)'\otimes_\C M^1\otimes_\C M^2}
	(\Om((M^3)')\o_\C M^1\o_\C M^2)$.
\end{proof}

	\begin{prop}\label{prop:spanningelt}
		The Lie algebra $\mathcal{L}_{\P^1\bs\{0,1\}}(V)_{\leq 0}$ is spanned by the following elements:
		\begin{equation}\label{eq:spanninglements}
			a\o \frac{z^{\wt a}}{z-1},\quad a\o z^{\wt a-k},\quad
			a\in V\ \mathrm{homogeneous},\ k\geq 1.
		\end{equation}
	\end{prop}
	
	We create a table for the pairs $(n,m)$ such that $n-m\leq \wt a-1$:
	\begin{equation}\label{eq:listfornm}
		\begin{matrix}
			\ds & (\wt a-3,-2) &(\wt a-2,-1)&\textcolor{red}{(\wt a-1,0)}& \textcolor{red}{(\wt a,1)}&(\wt a+1,2)& \ds \\
			\ds &(\wt a-3,-1)&\textcolor{red}{(\wt a-2,0)}&(\wt a-1,1)& (\wt a,2)& (\wt a+1,3)&\ds \\
			\ds &\textcolor{red}{(\wt a-3,0)}&(\wt a-2,1)&(\wt a-1,2)& (\wt a,3)& (\wt a+1,4)&\ds \\
			\textcolor{red}{\ds}&(\wt a-3,1)& (\wt a-2,2)& (\wt a-1,3)& (\wt a,4) &(\wt a+1,5)&\ds \\
			\ds & \vdots &\vdots & \vdots&\vdots & \vdots &\dots\\
			\ds & (-2) & (-1)& (0) &(1) &(2) &\ds,
		\end{matrix}
	\end{equation}
	wherein the columns are labeled by the indices $(i)$, with $i\in \Z$.
	The pairs $(n,m)=(\wt a,1)$ and $(\wt a-k,0)$, with $k\geq 1$, corresponding to the elements \eqref{eq:spanninglements}, are marked in red in Table~\eqref{eq:listfornm}.
	
	\begin{proof}[proof of Proposition~\ref{prop:spanningelt}]
		Let $\g$ be the subspace of $\mathcal{L}_{\P^1\bs\{0,1\}}(V)_{\leq 0}$ spanned by the elements \eqref{eq:spanninglements}.
		We need to show that $a\o \frac{z^{n}}{(z-1)^m}$, with $(n,m)$ given by the non-red pairs in \eqref{eq:listfornm}, are contained in $\g$, for any $a\in V$.
		By abuse of language, we also say that the pair $(n,m)$ belongs to $\g$ if the corresponding element $a\o \frac{z^{n}}{(z-1)^m}$ belongs to $\g$.
		
		Fix $l\geq 0$. We claim that
		\begin{equation}\label{eq:claim}
			\mathrm{If}\ a\o \frac{z^{\wt a-l}}{z-1}\in \g,\ \forall a\in V,
			\quad \mathrm{then}\quad
			a\o \frac{z^{\wt a-l}}{(z-1)^m}\in \g,\ \forall a\in V,\ m\geq 1.
		\end{equation}
		i.e., if the pair $(\wt a-l,1)$ belongs to $\g$, then all the pairs lying on the same column $(l)$ that are below $(\wt a-l,1)$ are contained in $\g$.
		
		Indeed, since $a\o \frac{z^{\wt a-l}}{z-1}\in \g$ for all $a\in V$ and $\wt(L(-1)a)=\wt a+1$, by \eqref{eq:defofchiralLie} we have
		\begin{align*}
			0&\equiv L(-1)a\o \frac{z^{\wt a+1-l}}{z-1}
			=-a\o \frac{d}{dz}\left(\frac{z^{\wt a+1-l}}{z-1}\right)\\
			&=-a\o \frac{(\wt a-l)z^{\wt a-l}}{z-1}
			+a\o \frac{z^{\wt a-l}}{(z-1)^2}\\
			&\equiv a\o \frac{z^{\wt a-l}}{(z-1)^2} \pmod{\g}.
		\end{align*}
		Hence $(\wt a-l,2)$ belongs to $\g$ for all $a\in V$.
		Proceeding in this way, using induction on $m$, we can show that $(\wt a-l,m)$ belongs to $\g$ for all $m\geq 2$ and $a\in V$.
		This proves Claim~\eqref{eq:claim}.
		In particular, all the pairs on column $(1)$ are contained in $\g$.
		
		On the other hand, we observe that for any $a\in V$ and $l,m\in \Z$, we have
		\[
		a\o \frac{z^{\wt a-l}}{(z-1)^m}
		+a\o \frac{z^{\wt a-l}}{(z-1)^{m+1}}
		=a\o \frac{z^{\wt a-l+1}}{(z-1)^{m+1}}.
		\]
		We use the following graph for the pairs $(n,m)$ to illustrate this property:
		\begin{equation}\label{eq:additionproperty}
			\begin{tikzcd}
				(\wt a-l,m)\arrow[d,dash,"+"]&(\wt a-l+1,m+1)\\
				(\wt a-l,m+1)\arrow[ur,"="']&
			\end{tikzcd}
		\end{equation}
		Using \eqref{eq:additionproperty}, it is easy to see that all the pairs on column $(i)$, with $i\geq 2$, are contained in $\g$.
		Furthermore, apply \eqref{eq:additionproperty} to the triple
		\[
		\begin{matrix}
			\textcolor{red}{(\wt a-1,0)} & \textcolor{red}{(\wt a,1)}\\
			(\wt a-1,1) &
		\end{matrix}
		\]
		We have $(\wt a-1,1)\in \g$.
		Then by Claim~\eqref{eq:claim}, all the pairs on column $(0)$ are contained in $\g$.
		Now apply \eqref{eq:additionproperty} to the triple
		\[
		\begin{matrix}
			\textcolor{red}{(\wt a-2,0)} & (\wt a-1,1)\\
			(\wt a-2,1) &
		\end{matrix}
		\]
		We have $(\wt a-2,1)\in \g$.
		By Claim~\eqref{eq:claim} again, all the pairs on column $(-1)$ that are below $(\wt a-2,0)$ are contained in $\g$.
		Proceeding in this way, we can show that all the pairs below the ones marked in red in Table~\eqref{eq:listfornm} are contained in $\g$.
		By applying \eqref{eq:additionproperty} successively, starting with the triple
		\[
		\begin{matrix}
			(\wt a-2,-1) & \textcolor{red}{(\wt a-1,0)}\\
			\textcolor{red}{(\wt a-2,0)} &
		\end{matrix}
		\]
		we can show that all the pairs above the red ones are contained in $\g$ as well.
		Hence all the pairs in \eqref{eq:listfornm} are contained in $\g$.
	\end{proof}

		\subsection{The augmented ideal of the $\infty$-restricted chiral Lie algebra}
		Inspired by the definition of $O(V)$ in the Zhu algebra $A(V)$ \cite{Z96}, we let
		\begin{equation}
			\mathrm{Ann}_{\mathcal{L}_{\P^1\bs\{0,1\}}(V)_{\leq 0}}(\Om((M^3)')):=\left\{X\in \mathcal{L}_{\P^1\bs\{0,1\}}(V)_{\leq 0}: \rho_\infty(X)(\Om((M^3)'))=0 \right\},
		\end{equation}
		which is clearly an ideal of the Lie algebra $\mathcal{L}_{\P^1\bs\{0,1\}}(V)_{\leq 0}$. 
		
		\begin{df}\label{def:L<0}
		Consider the following subspace of $\mathcal{L}_{\P^1\bs\{0,1\}}(V)_{\leq 0}$
			\begin{equation}\label{eq:defofrestrictchiralLie}
				\begin{aligned}
						\mathcal{L}_{\P^1\bs\{0,1\}}(V)_{< 0}&=\spn\{a\o f(z)\in \mathcal{L}_{\P^1\bs\{\infty,1,0\}}(V)
: a\in V,\ \ord_\infty (f(z))< \wt a-1 \}\\
&=\spn\left\{a\o \frac{z^n}{(z-1)^m}\in \mathcal{L}_{\P^1\bs\{0,1\}}(V)_{\leq 0}:a\in V,\ n-m< \wt a-1\right\}.
				\end{aligned}			
			\end{equation}
We call it the {\em augmented ideal} of the  $\infty$-restricted chiral Lie algebra $\mathcal{L}_{\P^1\bs\{0,1\}}(V)_{\leq 0}$. 
		\end{df}
		
		The fact that $	\mathcal{L}_{\P^1\bs\{0,1\}}(V)_{< 0}$ is an ideal of $\mathcal{L}_{\P^1\bs\{0,1\}}(V)_{\leq 0}$ follows from a similar estimate as \eqref{eq:estimate}, we omit the details.  In the following table, spanning elements of $	\mathcal{L}_{\P^1\bs\{0,1\}}(V)_{< 0}$ correspond to the pairs $(n,m)$ that are lying below the horizontal line: 
		\begin{equation}\label{eq:listforindeal}
			\begin{matrix}
				\ds & (\wt a-3,-2) &(\wt a-2,-1)&\textcolor{red}{	(\wt a-1,0)}&(\wt a,1)&(\wt a+1,2)& \ds \\
				\hline
				\ds &(\wt a-3,-1)&(\wt a-2,0)&	(\wt a-1,1)& (\wt a,2)& (\wt a+1,3)&\ds \\
				\ds &(\wt a-3,0)&	(\wt a-2,1)&	(\wt a-1, 2) & (\wt a,3)& (\wt a+1,4)&\ds \\
				\ds&(\wt a-3,1)& (\wt a-2,2)& (\wt a-1,3)& (\wt a, 4) &(\wt a+1,5)&\ds \\
				\ds 	& \vdots &\vdots &	\vdots&\vdots & \vdots &\dots
			\end{matrix}
		\end{equation}
		
		\begin{lm}\label{lm:propertyofinftychiralLie}
			The ideal	$\mathcal{L}_{\P^1\bs\{0,1\}}(V)_{< 0}$ satisfies the following properties: 
			\begin{enumerate}
				\item 	$\mathcal{L}_{\P^1\bs\{0,1\}}(V)_{< 0}\ssq \mathrm{Ann}_{\mathcal{L}_{\P^1\bs\{0,1\}}(V)_{\leq 0}}(\Om((M^3)'))$. In particular, we have 
				\begin{equation}\label{eq:quotientiso}
					\mathcal{L}_{\P^1\bs\{0,1\}}(V)_{< 0}.\left(\Om((M^3)') \o_\C M^1\o_\C M^2\right)=\Om((M^3)')\o_\C \left(	\mathcal{L}_{\P^1\bs\{0,1\}}(V)_{< 0}.(M^1\o M^2)\right).
				\end{equation}
				\item The quotient Lie algebra $\mathcal{L}_{\P^1\bs\{0,1\}}(V)_{\leq 0}/\mathcal{L}_{\P^1\bs\{0,1\}}(V)_{<0}$ is isomorphic to the subalgebra $L(V)_0$ of the Borcherds' Lie algebra $L(V)$ in Definition~\ref{def:BorLie}. 
			\end{enumerate}
		\end{lm}
		\begin{proof}
			Let $a\o\frac{z^n}{(z-1)^m}\in \mathcal{L}_{\P^1\bs\{0,1\}}(V)_{< 0}$. 
			By \eqref{eq:inftyaction} we have
			$$\rho_\infty\left(a\o\frac{z^n}{(z-1)^m}\right)(v'_3)=-\sum_{j\geq 0}\binom{-m}{j} (-1)^j a'(n-m-j)v'_3=0,$$
			for any $v'_3\in \Om((M^3)')$, since $\deg a'(n-m-j)=-\wt a+n-m-j+1<-j\leq 0$ for any $j\geq 0$. Thus, $	\mathcal{L}_{\P^1\bs\{0,1\}}(V)_{< 0}$ annihilates $\Om((M^3)')$. 
			
			Next, we show that $\mathcal{L}_{\P^1\bs\{0,1\}}(V)_{\leq 0}/\mathcal{L}_{\P^1\bs\{0,1\}}(V)_{<0}=\spn\{a\o z^{\wt a-1}+\mathcal{L}_{\P^1\bs\{0,1\}}(V)_{<0}:a\in V \}$. The pair $(\wt a-1,0)$ that corresponds to $a\o z^{\wt a-1}$ is marked in red in \eqref{eq:listforindeal}. Apply \eqref{eq:additionproperty} to the following triple in table \eqref{eq:listforindeal}: 
			$$\begin{matrix}
				\textcolor{red}{(\wt a-1,0)} &(\wt a,1)\\
				\hline
				(\wt a-1,1) &
			\end{matrix}$$
			we see that  $a\o \frac{z^{\wt a}}{z-1}\equiv a\o z^{\wt a-1}\pmod{\mathcal{L}_{\P^1\bs\{0,1\}}(V)_{<0}}$ since the pair $(\wt a-1,1)$ belongs to $\mathcal{L}_{\P^1\bs\{0,1\}}(V)_{<0}$. Then apply \eqref{eq:additionproperty} to the triple:  
			$$\begin{matrix}
				(\wt a,1)&(\wt a+1,2)\\
				\hline
				(\wt a,2) &
			\end{matrix}$$
			we have $a\o \frac{z^{\wt a+1}}{(z-1)^2}\equiv a\o \frac{z^{\wt a}}{z-1}\equiv a\o z^{\wt a-1}\pmod{\mathcal{L}_{\P^1\bs\{0,1\}}(V)_{<0}}$, since the pair $(\wt a,2)$ belongs to $\mathcal{L}_{\P^1\bs\{0,1\}}(V)_{<0}$. 
			Proceed like this, we can show that the elements $a\o\frac{z^n}{(z-1)^m}$ corresponding to all the pairs lying on the first line of \eqref{eq:listforindeal} are congruent to $a\o z^{\wt a-1}$ modulo $\mathcal{L}_{\P^1\bs\{0,1\}}(V)_{<0}$. 
			
			Now we prove (2). First, we observe that $L(-1)a\o z^{\wt a}+\mathcal{L}_{\P^1\bs\{0,1\}}(V)_{<0}=-(\wt a) a\o z^{\wt a-1}+\mathcal{L}_{\P^1\bs\{0,1\}}(V)_{<0}$ by \eqref{eq:defofchiralLie}, and
			\begin{equation}\label{eq:intermed1}
				\begin{aligned}
					&[a\o z^{\wt a-1}+\mathcal{L}_{\P^1\bs\{0,1\}}(V)_{<0}, b\o z^{\wt b-1}+\mathcal{L}_{\P^1\bs\{0,1\}}(V)_{<0}]\\
					&=\sum_{i\geq 0} \binom{\wt a-1}{i} a_i b\o z^{\wt a-1+\wt b-1-i}+ +\mathcal{L}_{\P^1\bs\{0,1\}}(V)_{<0},
				\end{aligned}
			\end{equation}by \eqref{eq:chiralLiebracket}. On the other hand, by Definition~\ref{def:BorLie}, the Lie algebra $L(V)_0=\spn\{a_{[\wt a-1]}:a\in V \}$, subject to the relations $(L(-1)a)_{[\wt a]}=-(\wt a) a_{[\wt a-1]}$ and 
			\begin{equation}\label{eq:intermed2}
				[a_{[\wt a-1]},b_{[\wt b-1]}]=\sum_{j\geq 0} \binom{\wt a-1}{j} (a_jb)_{[\wt a-1+\wt b-1-j]}. 
			\end{equation}
			Comparing \eqref{eq:intermed1} and \eqref{eq:intermed2}, it is easy to see that there is an epimorphism of Lie algebras 
			\begin{equation}\label{eq:quotientiso0}
				\psi: L(V)_0\twoheadrightarrow \mathcal{L}_{\P^1\bs\{0,1\}}(V)_{\leq 0}/\mathcal{L}_{\P^1\bs\{0,1\}}(V)_{<0},\quad a_{[\wt a-1]}\mapsto a\o z^{\wt a-1}+\mathcal{L}_{\P^1\bs\{0,1\}}(V)_{<0}.
			\end{equation} Conversely, we can define a linear map 
			\begin{equation}\label{eq:quotientLieiso}
				\varphi: \mathcal{L}_{\P^1\bs\{0,1\}}(V)_{\leq 0}\ra L(V)_0,\quad \varphi\left(a\o \frac{z^n}{(z-1)^m}\right):=\begin{cases}
					a_{[\wt a-1]}& \mathrm{if}\ n-m=\wt a-1,\\
					0&\mathrm{if}\ n-m<\wt a-1. 
			\end{cases}\end{equation}
			In view of table \eqref{eq:listforindeal}, $\varphi$ sends the spanning elements corresponding to the pairs $(n,m)$ on the first row to $a_{[\wt a-1]}$, and others to $0$. In particular, we have $\varphi(\mathcal{L}_{\P^1\bs\{0,1\}}(V)_{<0})=0$. To show $\varphi$ is well-defined, it suffices to check $\varphi$ preserves the following relations:
			\begin{equation}\label{eq:quotientrel}
				L(-1)a\o \frac{z^n}{(z-1)^m}=-a\o \frac{n z^{n-1}}{(z-1)^m}+a\o \frac{mz^n}{(z-1)^{m+1}},\quad n-m\leq \wt (L(-1)a)-1.
			\end{equation}
			Clearly, $\varphi$  is $0$ on both sides of \eqref{eq:quotientrel} if $n-m<\wt a$. If $n-m=\wt a$, we have 
			\begin{align*}
				&\varphi\left(-a\o \frac{n z^{n-1}}{(z-1)^m}+a\o \frac{mz^n}{(z-1)^{m+1}}\right)=-n a_{[\wt a-1]}+m a_{[\wt a-1]}=-(\wt a)a_{[\wt a-1]}\\
				&=(L(-1)a)_{[\wt a]}=\varphi\left(L(-1)a\o \frac{z^n}{(z-1)^m}\right).
			\end{align*}
			Thus, $\varphi$ in \eqref{eq:quotientLieiso} induces an inverse map of $\psi$ in \eqref{eq:quotientiso0}, which shows the isomorphism of Lie algebras $\mathcal{L}_{\P^1\bs\{0,1\}}(V)_{\leq 0}/\mathcal{L}_{\P^1\bs\{0,1\}}(V)_{<0}\cong L(V)_0$. 
		\end{proof}
		
		By adopting a similar argument as the proof of Proposition~\ref{prop:spanningelt}, we can easily prove the following fact about the spanning elements of $\mathcal{L}_{\P^1\bs\{0,1\}}(V)_{<0}$. 
		
		\begin{lm}\label{lm:spanningofrestricted}
			The augmented ideal $\mathcal{L}_{\P^1\bs\{0,1\}}(V)_{<0}$ is spanned by the following elements 
			\begin{equation}\label{eq:spannideal}
				a\o\frac{z^{\wt a-1}}{z-1},\quad a\o z^{\wt a-k},\quad a\in V\ \mathrm{homogeneous},\ k\geq 2
			\end{equation}
		\end{lm}
		\begin{equation}\label{eq:listforindealspanning}
			\begin{matrix}
				\ds & (\wt a-3,-2) &(\wt a-2,-1)&(\wt a-1,0)&(\wt a,1)&(\wt a+1,2)& \ds \\
				\hline
				\ds &(\wt a-3,-1)&\textcolor{red}{	(\wt a-2,0)}&	\textcolor{red}{	(\wt a-1,1)}& (\wt a,2)& (\wt a+1,3)&\ds \\
				\ds &\textcolor{red}{	(\wt a-3,0)}&	(\wt a-2,1)&	(\wt a-1, 2) & (\wt a,3)& (\wt a+1,4)&\ds \\
				\textcolor{red}{		\ds}&(\wt a-3,1)& (\wt a-2,2)& (\wt a-1,3)& (\wt a, 4) &(\wt a+1,5)&\ds \\
				\ds 	& \vdots &\vdots &	\vdots&\vdots & \vdots &\dots
			\end{matrix}
		\end{equation}
		The pairs $(n,m)$ corresponding to the spanning elements \eqref{eq:spannideal} are marked in red in table \eqref{eq:listforindealspanning}.


		\section{Space of $\infty$-restricted three-pointed conformal blocks}\label{Sec:4}
	We introduce the notion of $\infty$-restricted three-pointed conformal blocks on $\P^1$ in this section using the $\infty$-restricted chiral Lie algebra $\mathcal{L}_{\P^1\bs\{0,1\}}(V)_{\leq 0}$. We then use the ideal $\mathcal{L}_{\P^1\bs\{0,1\}}(V)_{<0}$ to define the contracted tensor product $M^1\odot M^2$. This space is closely related to the left $A(V)$-module $A(M^1)\o_{A(V)}\Om(M^2)$ \cite{FZ92}; we discuss their relations at the end of this section.
	
	\subsection{$\infty$-restricted three-pointed conformal blocks on $\P^1$}
	
	We restrict the module $(M^3)'$ in the datum \eqref{eq:datum} to its bottom degree and obtain the following datum:
	\begin{equation}
		\Sigma(\Om((M^3)'), M^1, M^2)
		=\left(\P^1, (\infty, 1,0),(1/z, z-1,z), (\Om((M^3)'),M^1,M^2)\right).
	\end{equation}
	Note that $\Om((M^3)')$ is naturally a right module over the Zhu algebra $A(V)$ and a left module over $A(V)$ via the involution
	$\theta: A(V)\ra A(V)$, $[a]\mapsto [e^{L(1)}(-1)^{L(0)}a]$.
	
	Consider the $\infty$-restricted chiral Lie algebra $\mathcal{L}_{\P^1\bs\{0,1\}}(V)_{\leq 0}$ introduced in the previous section.
	It follows from Lemma~\ref{lm:inftychiralLiestab} that
	\[
	\mathcal{L}_{\P^1\bs\{0,1\}}(V)_{\leq 0}\cdot
	(\Om((M^3)')\o_\C M^1\o_\C M^2)
	\ssq
	(\Om((M^3)')\o_\C M^1\o_\C M^2).
	\]
	
	\begin{df}\label{def:restrictedcfb}
		The vector space
		\begin{equation}\label{eq:restrictedconformalblocks}
			\mathscr{C}_{\mathrm{res}}\left(\Sigma(\Om((M^3)'), M^1, M^2)\right)
			:= \left(
			\frac{\Om((M^3)')\otimes_\C M^1\otimes_\C M^2}
			{\mathcal{L}_{\P^1\bs\{0,1\}}(V)_{\leq 0}\cdot
				\left(\Om((M^3)')\otimes_\C M^1\otimes_\C M^2\right)}
			\right)^\ast
		\end{equation}
		is called the {\em space of (three-pointed) $\infty$-restricted conformal blocks on $\P^1$}. An element
		$\varphi \in \mathscr{C}_{\mathrm{res}}\left(\Sigma(\Om((M^3)'), M^1, M^2)\right)$
		is called a {\em (three-pointed) $\infty$-restricted conformal block} associated to the datum
		$\Sigma(\Om((M^3)'), M^1, M^2)$.
	\end{df}
	
	We want to express the right-hand side of \eqref{eq:restrictedconformalblocks} in terms of a Hom space. We observe the following facts from the representation theory of Lie algebras. The proof is standard, and we omit the details.
	
	\begin{lm}\label{lm:factabouthomspaceLie}
	Let $\g$ be a Lie algebra, and let $W$, $M$, and $U$ be $\g$-modules. Let $U^\ast$ be the dual $\g$-module of $U$. Then
		\begin{enumerate}
			\item There is an isomorphism of vector spaces:
			\begin{equation}\label{eq:liemodulehom}
				\Hom_{\C}(U\o_\C W/ \g.(U\o_\C W),\C)\cong \Hom_{\g}(W,U^\ast),
			\end{equation}
			where $U\o_\C W$ is the tensor product of $\g$-modules.
			\item Let $O(\g)\leq \g$ be an ideal of $\g$. Then $M/O(\g).M$ is a $\g/O(\g)$-module, and we have an isomorphism of vector spaces:
			\begin{equation}\label{eq:liequotient}
				M/\g.M\cong (M/O(\g).M)/(\g/O(\g)).(M/O(\g).M).
			\end{equation}
		\end{enumerate}
	\end{lm}
	
	We will apply Lemma~\ref{lm:factabouthomspaceLie} to the following datum:
	\begin{equation}\label{eq:lmdatum}
		\begin{aligned}
			&\g= \mathcal{L}_{\P^1\bs\{0,1\}}(V)_{\leq 0},&& W=M^1\o_\C M^2, && M=\Om((M^3)')\o_\C M^1\o_\C M^2, \\
			&U=\Om((M^3)'),&& O(\g)=\mathcal{L}_{\P^1\bs\{0,1\}}(V)_{< 0}.&&
		\end{aligned}
	\end{equation}
	
	By Proposition~\ref{prop:chiralLieaction}, $M^1\o_\C M^2$ is the usual tensor product module over the Lie algebra $\mathcal{L}_{\P^1\bs\{0,1\}}(V)_{\leq 0}$, with the module action given by
	
	\begin{equation}\label{eq:rho10}
		\begin{aligned}
			\rho_{1,0}:\mathcal{L}_{\P^1\bs\{0,1\}}(V)_{\leq 0}&\ra \mathfrak{gl}(M^1\o_\C M^2),\\
			\rho_{1,0}(X)(v_1\o v_2)&=\rho_1(X)(v_1)\o v_2+ v_1\o \rho_0(X)(v_2),
		\end{aligned}
	\end{equation}
	for $X\in \mathcal{L}_{\P^1\bs\{0,1\}}(V)_{\leq 0}$, $v_1\in M^1$, and $v_2\in M^2$, where $\rho_1$ and $\rho_0$ are given by \eqref{eq:1action} and \eqref{eq:0action}, respectively. Thus, the ideal $\mathcal{L}_{\P^1\bs\{0,1\}}(V)_{<0}$ also acts on $M^1\o_\C M^2$.
	
	\begin{df}\label{df:M1odotM2}
		Let $M^1$ and $M^2$ be $V$-modules. Define
		\begin{equation}\label{eq:defofodot}
			\begin{aligned}
				M^1\odot M^2:&=(M^1\o_\C M^2)/\mathcal{L}_{\P^1\bs\{0,1\}}(V)_{< 0}.(M^1\o_\C M^2).
			\end{aligned}
		\end{equation}
		We call it the {\em contracted tensor product of $M^1$ and $M^2$}.
	\end{df}
	
	We remark on the following basic facts about the contracted tensor product $M^1\odot M^2$:
	\begin{enumerate}
		\item By Lemma~\ref{lm:spanningofrestricted}, together with \eqref{eq:actionofchiralliealgebraelements}, $M^1\odot M^2$ is spanned by the symbols $v_1\odot v_2$, which are bilinear in $v_1$ and $v_2$, with $v_1\in M^1$ and $v_2\in M^2$, subject to the following relations:
		\begin{align}
			\sum_{j\geq 0}\binom{\wt a-1}{j} a(j-1)v_1\odot v_2&=v_1\odot \sum_{j\geq 0} a(\wt a-1+j)v_2,\label{eq: actionspanning1}\\
			\sum_{j\geq 0} \binom{\wt a-k}{j} a(j)v_1\odot v_2&=-v_1\odot a(\wt a-k)v_2,\quad k\geq 2,\label{eq: actionspanning2}
		\end{align}
		for any $a\in V$, $v_1\in M^1$, and $v_2\in M^2$.
		\item By Lemmas~\ref{lm:propertyofinftychiralLie} and \ref{lm:factabouthomspaceLie}, $M^1\odot M^2$ is a module over the Lie algebra $L(V)_0$, with the module action induced by the map $\rho_{1,0}$ in \eqref{eq:rho10}:
		\begin{equation}\label{eq:repofLV0onM1odotM2}
			\overline{\rho_{1,0}}: L(V)_0\cong \mathcal{L}_{\P^1\bs\{0,1\}}(V)_{\leq 0}/\mathcal{L}_{\P^1\bs\{0,1\}}(V)_{< 0}\ra \gl(M^1\odot M^2).
		\end{equation}
		In particular, $a_{[\wt a-1]}.(v_1\odot v_2)=\overline{\rho_{1,0}}(a\o z^{\wt a-1}+\mathcal{L}_{\P^1\bs\{0,1\}}(V)_{< 0})(v_1\odot v_2)$, and we can write
		\begin{equation}\label{eq: L(V)0action}
			a_{[\wt a-1]}.(v_1\odot v_2)=\sum_{j\geq 0}\binom{\wt a-1}{j} a(j)v_1\odot v_2+ v_1\odot o(a)v_2,
		\end{equation}
		for any $a_{[\wt a-1]}\in L(V)_0$, $v_1\in M^1$, and $v_2\in M^2$, where $o(a)=a(\wt a-1)$.
		\item More generally, for $s,t\in \Z$ such that $s+t>1$, we have $a\o \frac{z^{\wt a-s}}{(z-1)^t}\in \mathcal{L}_{\P^1\bs\{0,1\}}(V)_{< 0}$ by \eqref{eq:defofrestrictchiralLie}. In particular, the following relations hold in $M^1\odot M^2$:
		\begin{equation}\label{eq:actionfogeneralspann}
			\begin{aligned}
				\rho_1\left(a\o\frac{z^{\wt a-s}}{(z-1)^t}\right)(v_1)\odot v_2
				&=-v_1\odot \rho_0\left(a\o \frac{z^{\wt a-s}}{(z-1)^t}\right)(v_2),\quad \mathrm{equivalently,}\\
				\sum_{j\geq 0} \binom{\wt a-s}{j} a(j-t)v_1\odot v_2
				&=-v_1\odot \sum_{j\geq 0}\binom{-t}{j} (-1)^{t+j} a(\wt a-s+j)v_2.
			\end{aligned}
		\end{equation}
	\end{enumerate}
	
	Note that the left $A(V)$-module $\Om(M^3)$ is also a module over the Lie algebra $L(V)_0$ via the Lie algebra homomorphism
	$L(V)_0\ra A(V)_{\mathrm{Lie}}$, $a_{[\wt a-1]}\mapsto [a]$, see \cite{DLM98}.
	
		\begin{prop}\label{prop:inftycfbHom}
   Let $M^1,M^2,M^3$ be $V$-modules such that $\Om((M^3)')^\ast\cong\Om(M^3)$, which is the case when $M^3$ is an irreducible ordinary $V$-module.  
There is an isomorphism of vector spaces:
\begin{equation}\label{eq:inftyconfrmalblockshom}
	\mathscr{C}_{\mathrm{res}}\left(\Sigma(\Om((M^3)'), M^1, M^2)\right)\cong 
	\Hom_{L(V)_0}\left(M^1\odot M^2, \Om(M^3)\right).
\end{equation}
		\end{prop}
		\begin{proof}
			Apply Lemma~\ref{lm:factabouthomspaceLie} to the datum \eqref{eq:lmdatum}. We have
			\begin{align*}
				&\mathscr{C}_{\mathrm{res}}\left(\Sigma(\Om((M^3)'), M^1, M^2)\right)= \left(\frac{\Om((M^3)')\otimes_\C M^1\otimes_\C M^2}{\mathcal{L}_{\P^1\bs\{0,1\}}(V)_{\leq 0}.\left(\Om((M^3)')\otimes_\C M^1\otimes_\C M^2\right)}\right)^\ast\\
				&\cong \Hom_\C \left(\frac{(\Om((M^3)')\otimes_\C M^1\otimes_\C M^2)/ 	\mathcal{L}_{\P^1\bs\{0,1\}}(V)_{< 0}.(\Om((M^3)')\otimes_\C M^1\otimes_\C M^2)}{(\mathcal{L}_{\P^1\bs\{0,1\}}(V)_{\leq 0}/\mathcal{L}_{\P^1\bs\{0,1\}}(V)_{< 0}).\mathrm{(The\ Numerator)}}, \C\right)\quad (\mathrm{by}\ \eqref{eq:liequotient})\\
				&\cong \Hom_{\C}\left(\frac{(\Om((M^3)')\otimes_\C M^1\otimes_\C M^2)/\left(\Om((M^3)')\o_\C 	\mathcal{L}_{\P^1\bs\{0,1\}}(V)_{< 0}.(M^1\o M^2)\right)}{(L(V)_0).\mathrm{(The\ Numerator)}},\C\right) \quad (\mathrm{by}\ \eqref{eq:quotientiso})\\
				&\cong \Hom_\C \left(\frac{\Om((M^3)') \o _\C \left(M^1\odot M^2\right)}{(L(V)_0).\left(\Om((M^3)') \o _\C \left(M^1\odot M^2\right)\right)},\C \right) \quad (\mathrm{by}\ \eqref{eq:defofodot})\\
				&\cong\Hom_{L(V)_0}\left(M^1\odot M^2, \Om((M^3)')^\ast\right)\quad (\mathrm{by}\ \eqref{eq:liemodulehom})\\
                &\cong  \Hom_{L(V)_0}\left(M^1\odot M^2, \Om(M^3)\right),
			\end{align*}
			where we used the isomorphism $\mathcal{L}_{\P^1\bs\{0,1\}}(V)_{\leq 0}/\mathcal{L}_{\P^1\bs\{0,1\}}(V)_{< 0}\cong L(V)_0$ in Lemma~\ref{lm:propertyofinftychiralLie}.
		\end{proof}

		\subsection{The $L(V)_0$-module $M^1\odot M^2$ and the $A(V)$-action}
		We discuss some basic properties of the $L(V)_0$-module $M^1\odot M^2$ in Definition~\ref{df:M1odotM2}. In particular, we show that $M^1\odot M^2$ is also a left module over the Zhu algebra $A(V)$ if either $M^1$ or $M^2$ is generated by its bottom degree.
		
		\begin{lm}\label{lm:spannodot}
			Let $M^1,M^2$ be $V$-modules.
			\begin{enumerate}
				\item If $M^2$ is generated by $\Om(M^2)$, then
				\begin{equation}\label{eq:M1odotM2spann1}
					M^1\odot M^2=\spn\{v_1\odot v^2: v_1\in M^1, v^2\in \Om(M^2)\}.
				\end{equation}
				\item If $M^1$ is generated by $\Om(M^1)$, then
				\begin{equation}\label{eq:M1odotM2spann2}
					M^1\odot M^2=\spn\{v^1\odot v_2: v^1\in \Om(M^1), v_2\in M^2\}.
				\end{equation}
			\end{enumerate}
			In particular, \eqref{eq:M1odotM2spann1} (resp. \eqref{eq:M1odotM2spann2}) holds if $M^2$ (resp. $M^1$) is an irreducible $V$-module.
		\end{lm}
		\begin{proof}
			We first show (1). By assumption,
			$M^2=\spn\{a^1(n_1)\ds a^r(n_r)v^2: a^i\in V, n_i\in \Z, v^2\in \Om(M^2)\}$.
			Note that $a^i(\wt a^i-k_i)v^2$ is in $\Om(M^2)$ if $k_i\leq 1$. It follows that $M^2$ is spanned by the following elements:
			$$
			a^1(\wt a^1-k_1)\ds a^r(\wt a^r-k_r)v^2,\quad r\geq 0, a^i\in V,\ k_i\geq 2,\ v^2\in \Om(M^2).
			$$
			Then, by \eqref{eq: actionspanning2}, we have
			\begin{align*}
				&v_1\odot a^1(\wt a^1-k_1) a^2(\wt a^2-k_2)\ds a^r(\wt a^r-k_r)v^2\\
				&=-\sum_{j_1\geq 0} \binom{\wt a^1-k_1}{j_1} a^1(j_1)v_1\odot a^2(\wt a^2-k_2)\ds a^r(\wt a^r-k_r)v^2\\
				&\vdots\\
				&=(-1)^r \sum_{j_1\geq 0}\ds \sum_{j_r\geq 0}\binom{\wt a^1-k_1}{j_1}\ds \binom{\wt a^r-k_r}{j_r} a^r(j_r)\ds a^1(j_1)v_1\odot v^2,
			\end{align*}
			where $v_1\in M^1$, $a^i\in V$, and $k_i\geq 2$ for all $i$, and $v^2\in \Om(M^2)$. This proves \eqref{eq:M1odotM2spann1}.
			
			To show (2), we first note that $M^1$ is spanned by $a^1(\wt a^1-k_1)\ds a^r(\wt a^r-k_r)v^1$, with $r\geq 0$, $k_i\geq 2$, and $v^1\in \Om(M^1)$. By Definition~\ref{def:L<0}, for any $k\geq 2$, we have $a\o (1/(z-1)^{k-\wt a})\in \mathcal{L}_{\P^1\bs\{0,1\}}(V)_{< 0}$. Then it follows from \eqref{eq:actionfogeneralspann}, with $s=\wt a$ and $t=k-\wt a$, that
			\begin{align*}
				&a^1(\wt a^1-k_1)a^2(\wt a^2-k_2)\ds a^r(\wt a^r-k_r)v^1\odot v_2\\
				&= -a^2(\wt a^2-k_2)\ds a^r(\wt a^r-k_r)v^1\odot \sum_{j_1\geq 0}\binom{\wt a_1-k_1}{j_1} (-1)^{\wt a_1-k_1+j_1} a^1(j_1) v_2\\
				&\vdots\\
				&=(-1)^r v^1\odot \sum_{j_1\geq 0}\ds \sum_{j_r\geq 0} \binom{\wt a^1-k_1}{j_1}\ds \binom{\wt a^r-k_r}{j_r} (-1)^{\sum_{i=1}^r (\wt a_i-k_i+j_i)} a^r(j_r)\ds a^1(j_1)v_2.
			\end{align*}
			This proves \eqref{eq:M1odotM2spann2}.
		\end{proof}
		
		Let $M$ be an admissible $V$-module. We recall the construction of the bimodule $A(M)$ over the Zhu algebra $A(V)$ in \cite{FZ92}. By \cite[Definition 1.5.2]{FZ92}, $V$ has left and right $\ast$-actions on $M$. We can reinterpret these actions in terms of $\rho_1$ \eqref{eq:1action}:
		
	\begin{align}
		&V\times M\ra M, \quad (a,v)\mapsto a\ast v:=\rho_1\left(a\o \frac{z^{\wt a}}{z-1}\right)(v)=\sum_{j\geq 0}\binom{\wt a}{j} a(j-1)v, \label{eq:leftast}\\
		& M\times V\ra M,\quad (v,a)\mapsto v\ast a=\rho_1\left(a\o \frac{z^{\wt a-1}}{z-1}\right)(v)=\sum_{j\geq 0}\binom{\wt a-1}{j} a(j-1)v. \label{eq:rightast}
	\end{align}
	Observe that
	\[
	a\ast v-v\ast a=\sum_{j\geq 0} \binom{\wt a-1}{j} a(j)v=\rho_1(a\o z^{\wt a-1})(v).
	\]
	Using the $\ast$-symbol as above, we can rewrite the relations \eqref{eq: actionspanning1} and \eqref{eq: L(V)0action} in $M^1\odot M^2$ as follows:
	\begin{align}
		&(v_1\ast a)\odot v_2=v_1\odot \sum_{j\geq 0} a(\wt a-1+j)v_2,\label{eq:actionstarform}\\
		& a_{[\wt a-1]}.(v_1\odot v_2)=\left(a\ast v_1-v_1\ast a\right)\odot v_2+ v_1\odot o(a)v_2.\label{eq:L(V)0starform}
	\end{align}
	
	The following formulas are $(1.5.11)-(1.5.15)$ in \cite{FZ92}. They show that $A(M)=M/O(M)$, with
	$O(M)=\spn\{a\circ v=\sum_{j\geq 0}\binom{\wt a}{j} a(j-2)v: a\in V, v\in M\}$, is a bimodule over $A(V)$ \ref{def:AV}:
	\begin{align*}
		&O(V)\ast v\ssq O(M),&&v\ast O(V)\ssq O(M),\\
		& a\ast O(M)\ssq O(M),&& O(M)\ast a\ssq O(M),\numberthis\label{eq:starbimodule}\\
		& (a\ast b)\ast v-a\ast (b\ast v)\in O(M),&& (v\ast a)\ast b-v\ast (a\ast b)\in O(M),\\
		& (a\ast v)\ast b-a\ast (v\ast b)\in O(M).
	\end{align*}
	
	Next, we show that $A(V)$ has a well-defined left module action on the contracted tensor product $M^1\odot M^2$ if either $M^1$ or $M^2$ is generated by its bottom degree as a $V$-module.
	\begin{construction}\label{const:staraction}
		Note that $\iota: V\ra L(V)_0, a\mapsto a_{[\wt a-1]}$ is a well-defined linear map. Composing this map with \eqref{eq:repofLV0onM1odotM2}, we obtain a linear map
		$$
		V\xrightarrow{\iota} L(V)_0\cong \mathcal{L}_{\P^1\bs\{0,1\}}(V)_{\leq 0}/\mathcal{L}_{\P^1\bs\{0,1\}}(V)_{< 0}\xrightarrow{\overline{\rho_{1,0}}} \gl(M^1\odot M^2).
		$$
		By \eqref{eq: L(V)0action} and \eqref{eq:L(V)0starform}, the map $\overline{\rho_{1,0}}\circ \iota$ induces an action of $V$ on $M^1\odot M^2$, denoted by $\ast$:
		\begin{equation}\label{eq:VactiononM1oM2}
			\begin{aligned}
				V\times \left(M^1\odot M^2\right)&\ra M^1\odot M^2, \quad 	(a, v_1\odot v_2)\mapsto 	a\ast (v_1\odot v_2),\\
				a\ast (v_1\odot v_2):&=\overline{\rho_{1,0}}(a\o z^{\wt a-1}+\mathcal{L}_{\P^1\bs\{0,1\}}(V)_{< 0})(v_1\odot v_2)\\
				&=(a\ast v_1-v_1\ast a)\odot v_2+ v_1\odot o(a)v_2,
			\end{aligned}
		\end{equation}
		where $a\ast v_1$ and $v_1\ast a$ are given by \eqref{eq:leftast} and \eqref{eq:rightast}, respectively.
		
	\end{construction}
	
	\begin{remark}
		The expression of the $\ast$-action \eqref{eq:VactiononM1oM2} is not unique, due to the fact that $a\o z^{\wt a-1}+\mathcal{L}_{\P^1\bs\{0,1\}}(V)_{< 0}$ has multiple representatives in $L(V)_0\cong \mathcal{L}_{\P^1\bs\{0,1\}}(V)_{\leq 0}/\mathcal{L}_{\P^1\bs\{0,1\}}(V)_{< 0}$, and $M^1\odot M^2$ is a quotient module. We give an alternative expression of \eqref{eq:VactiononM1oM2} that is useful in the next theorem. Using the addition property \eqref{eq:additionproperty}, we see from table \eqref{eq:listforindeal} that
		$$a\o z^{\wt a-1}\equiv a\o\frac{z^{\wt a-1+q}}{(z-1)^q}\pmod{\mathcal{L}_{\P^1\bs\{0,1\}}(V)_{< 0}},\quad q\in \Z. $$
		In particular, let $q=-\wt a+1$. Then
		\begin{align*}
			a\ast (v_1\odot v_2)&=\overline{\rho_{1,0}}\left(a\o \frac{1}{(z-1)^{-\wt a+1}}+\mathcal{L}_{\P^1\bs\{0,1\}}(V)_{< 0}\right)(v_1\odot v_2)\\
			&=\rho_1\left(a\o (z-1)^{\wt a-1}\right)(v_1)\odot v_2+ v_1\odot \rho_0(a\o (z-1)^{\wt a-1})(v_2) \numberthis\label{eq:altastaction}\\
			&= o(a)v_1\odot v_2+ v_1\odot \Res_{z=0} Y_{M^2}(a,z)v_2(-1+z)^{\wt a-1}.
		\end{align*}
		Here we adopt the convention that $(-1+z)^n$ is the expansion of $z$ in $|z|<1$. 
	\end{remark}

		\begin{thm}\label{thm:AVaction}
	Let $M^1$ and $M^2$ be $V$-modules. If either $M^1$ or $M^2$ is generated by its bottom degree, then the $\ast$-action \eqref{eq:VactiononM1oM2} satisfies $O(V)\ast (M^1\odot M^2)=0$, and $M^1\odot M^2$ becomes a left $A(V)$-module with respect to the following action:
	\begin{equation}\label{eq:AVactiononodot}
    \begin{aligned}
  A(V)&\times \left(M^1\odot M^2\right)\ra M^1\odot M^2,\\
  [a].(v_1\odot v_2)&=(a\ast v_1-v_1\ast a)\odot v_2+ v_1\odot o(a)v_2,      
    \end{aligned}
	\end{equation}
	where $[a]\in A(V)$, $v_1\in M^1$, and $v_2\in M^2$.
\end{thm}

\begin{proof}
	Assume that $M^2$ is generated by $\Om(M^2)$. Then, by Lemma~\ref{lm:spannodot} (1), to show $O(V)\ast (M^1\odot M^2)=0$, it suffices to show that $O(V)\ast (v_1\odot v^2)=0$ for any $v^2\in \Om(M^2)$ and $v_1\in M^1$. We claim that
	\begin{equation}\label{eq:OM1odot}
		O(M^1)\odot v^2=0,\quad v^2\in \Om(M^2).
	\end{equation}
	Indeed, let $a\circ v_1=\sum_{j\geq 0}\binom{\wt a}{j} a(j-2)v_1\in O(M^1)$. Since $a\o \frac{z^{\wt a}}{(z-1)^2}\in  \mathcal{L}_{\P^1\bs\{0,1\}}(V)_{< 0}$ in view of table \eqref{eq:listforindealspanning}, we have
	\begin{align*}
		(a\circ v_1)\odot v^2&=\Res_{z=1}Y_{M^1}(a,z-1)v_1 \iota_{1,z-1}\left(\frac{z^{\wt a}}{(z-1)^2}\right)\odot v^2=\rho_1\left(a\o \frac{z^{\wt a}}{(z-1)^2}\right)(v_1)\odot v^2\\
		&= -v_1\odot \rho_0\left(a\o \frac{z^{\wt a}}{(z-1)^2}\right)v^2=-v_1\odot \left(\sum_{j\geq 0} j a(\wt a+j-1)v^2\right)=0.
	\end{align*}
	This proves \eqref{eq:OM1odot}. Recall that $o(O(V))\Om(M^2)=0$ (see \cite[Theorem 2.1.2]{Z96}). Then, by \eqref{eq:OM1odot} and \eqref{eq:starbimodule}, we have
	\begin{align*}
		O(V) \ast (v_1\o v^2)=(O(V)\ast v_1-v_1\ast O(V))\odot v^2+ v_1\odot o(O(V))v^2\ssq O(M^1)\odot v^2+0=0.
	\end{align*}
	Hence \eqref{eq:AVactiononodot} is well-defined. Now let $v_1\in M^1$ and $v^2\in \Om(M^2)$. Then, by \eqref{eq:actionstarform}, we can simplify the action \eqref{eq:AVactiononodot} as follows:
	\begin{align*}
		[a]. (v_1 \odot v^2)&=(a\ast v_1-v_1\ast a)\odot v^2+v_1\odot o(a)v^2\\
		&=(a\ast v_1)\odot v^2-v_1\odot o(a)v^2+v_1\odot o(a)v^2\numberthis\label{eq:simplifiedaction}\\
		&=(a\ast v_1)\odot v^2.
	\end{align*}
	In particular, for any $a,b\in V$, by \eqref{eq:starbimodule} and \eqref{eq:OM1odot} we have
	\begin{align*}
		&([a]\ast [b]).(v_1\odot v^2)-[a]. \left([b].(v_1\odot v^2)\right)\\
		&=\left((a\ast b)\ast v_1-a\ast (b\ast v_1)\right)\odot v^2 \in O(M^1)\odot v^2=0.
	\end{align*}
	Since $M^1\odot M^2=\spn\{v_1\odot v^2: v_1\in M^1, v^2\in \Om(M^2)\}$, this shows that $M^1\odot M^2$ is a left $A(V)$-module via \eqref{eq:AVactiononodot}.

	Now assume that $M^1$ is generated by $\Om(M^1)$. Instead of using the known properties of $O(V)$ and $O(M)$ in \eqref{eq:starbimodule}, we prove $O(V)\ast (M^1\odot M^2)=0$ directly.

	Let $v^1\odot v_2$ be a spanning element of $M^1\odot M^2$, with $v^1\in \Om(M^1)$ and $v_2\in M^2$ (see Lemma~\ref{lm:spannodot} (2)). Let $a\circ b\in O(V)$, with $a,b\in V$ homogeneous. It follows from \eqref{eq:altastaction}, the Jacobi identity, and relation \eqref{eq:actionfogeneralspann} that
	\begin{align*}
		&(a\circ b)\ast (v^1\odot v_2)\\
		&=o(a\circ b)(v^1)\odot v_2+ v^1\odot \sum_{j\geq 0} \Res_{z_2=0} \binom{\wt a}{j} Y_{M^2}(a(j-2)b,z_2)v_2 (-1+z_2)^{\wt a-j+\wt b}\\
		&=0+v^1\odot \Res_{z_2=0}\Res_{z_1-z_2=0} Y_{M^2}(Y(a,z_1-z_2)b,z_2)v_2 \frac{(-1+z_2)^{\wt b}}{(z_1-z_2)^2} (-1+z_1)^{\wt a}\\
		&= v^1\odot \sum_{i\geq 0} \binom{-2}{i}\Res_{z_1=0}\Res_{z_2=0} Y_{M^2}(a,z_1)Y_{M^2}(b,z_2)v_2 \frac{(-1+z_1)^{\wt a}}{z_1^{2+i}} (-1+z_2)^{\wt b}(-z_2)^i\\
		&\ \ \ -v^1\odot \sum_{i\geq 0} \binom{-2}{i} \Res_{z_2=0}\Res_{z_1=0} Y_{M^2}(b,z_2)Y_{M^1}(a,z_1)v_2 \frac{(-1+z_2)^{\wt b}}{z_2^{2+i}} (-1)^{2+i}(-1+z_1)^{\wt a}z_1^i\\
		&=-\sum_{i\geq 0}\binom{-2}{i} \rho_1\left(a\o \frac{z^{-2-i}}{(z-1)^{-\wt a}}\right)(v^1)\odot \Res_{z_2=0}Y_{M^2}(b,z_2)v_2(-1+z_2)^{\wt b} (-z_2)^i \\
		&\ \ \ + \sum_{i\geq 0} \binom{-2}{i} \rho_1\left(b\o \frac{z^{-2-i}}{(z-1)^{-\wt b}}\right)(v^1)\odot \Res_{z_1=0} Y_{M^2}(a,z_1)v_2 (-1)^{2+i} (-1+z_1)^{\wt a} z_1^i\\
		&=-\sum_{i\geq0}\sum_{k\geq 0} \binom{-2}{i} \binom{-2-i}{k} a(\wt a+k)v^1\odot \Res_{z_2=0}Y_{M^2}(b,z_2)v_2(-1+z_2)^{\wt b} (-z_2)^i \\
		&\ \ \ +\sum_{i\geq 0} \sum_{k\geq 0} \binom{-2}{i} \binom{-2-i}{k} b(\wt b+k)v^1\odot \Res_{z_1=0} Y_{M^2}(a,z_1)v_2 (-1)^{2+i} (-1+z_1)^{\wt a} z_1^i\\
		&=0,
	\end{align*}
	since $a(\wt a+k)v^1, b(\wt b+k)v^1\in M^1(-k-1)=0$ for $k\geq 0$. This shows $O(V)\ast (M^1\odot M^2)=0$, and \eqref{eq:altastaction} induces an action $A(V)\times \left(M^1\odot M^2\right)\ra M^1\odot M^2$, with
	\begin{equation}\label{eq:altaction2}
		[a].(v_1\odot v_2)=o(a)v_1\odot v_2+ v_1\odot \Res_{z=0} Y_{M^2}(a,z)v_2(-1+z)^{\wt a-1}.
	\end{equation}

	Finally, we show that $M^1\odot M^2$ is a left $A(V)$-module with respect to the action \eqref{eq:altaction2}. Let $a,b\in V$ be homogeneous, $v^1\in \Om(M^1)$, and $v_2\in M^2$. Then
	\begin{align*}
		&([a]\ast [b]).(v^1\odot v_2)=[a\ast b].(v^1\odot v_2)\\
		&=o(a\ast b)v^1\odot v_2+v^1\odot \sum_{j\geq 0}\Res_{z_2=0} \binom{\wt a}{j} Y_{M^2}(a(j-1)b,z_2)_2 (-1+z_2)^{\wt a-j-1+\wt b}\\
		&= o(a)o(b)v^1\odot v_2+v^1\odot \Res_{z_2=0}\Res_{z_1-z_2=0} Y_{M^2}(Y(a,z_1-z_2)b,z_2)v_2 \frac{(-1+z_1)^{\wt a}}{z_1-z_2}(-1+z_2)^{\wt b-1}\\
		&=\underbrace{ o(a)o(b)v^1\odot v_2}_{(A)}\\
		&\ \ \ + \underbrace{v^1\odot \sum_{j\geq 0} \Res_{z_1=0}\Res_{z_2=0} Y_{M^2}(a,z_1)Y_{M^2}(b,z_2)v_2 \frac{(-1+z_1)^{\wt a}}{z_1^{1+j}} (-1+z_2)^{\wt b-1} z_2^j}_{(B)} \\
		&\ \ \ +\underbrace{v^1\odot \sum_{j\geq0} \Res_{z_2=0}\Res_{z_1=0} Y_{M^2}(b,z_2)Y_{M^2}(a,z_1)v_2 \frac{(-1+z_2)^{\wt b-1}}{z_2^{1+j}} (-1+z_1)^{\wt a} z_1^j}_{(C)}\\
		&=(A)+(B)+(C).
	\end{align*}
	Write $\frac{(-1+z_1)^{\wt a}}{z_1}=-\frac{(-1+z_1)^{\wt a-1}}{z_1}+(-1+z_1)^{\wt a-1}$. Then, by \eqref{eq:actionfogeneralspann}, we can express $(B)$ as follows:
	\begin{align*}
		(B)&=\rho_1\left(a\o \frac{z^{-1}}{(z-1)^{-\wt a+1}}\right)(v^1)\odot \Res_{z_2=0} Y_{M^2}(b,z_2)v_2 (-1+z_2)^{\wt b-1} \\
		&\ \ \ +v^1\odot \Res_{z_1=0}\Res_{z_2=0} Y_{M^2}(a,z_1)Y_{M^2}(b,z_2)v_2 (-1+z_1)^{\wt a-1} (-1+z_2)^{\wt b-1} \\
		&\ \ \ -\sum_{j\geq 1} \rho_1\left(a\o \frac{z^{-1-j}}{(z-1)^{-\wt a}}\right)(v^1)\odot \Res_{z_2=0} Y_{M^2}(b,z_2)v_2 (-1+z_2)^{\wt b-1}\\
		&= \sum_{k\geq 0} \binom{-1}{k} a(\wt a-1+k)v^1\odot \Res_{z_2=0} Y_{M^2}(b,z_2)v_2 (-1+z_2)^{\wt b-1}\\
		&\ \ \ +v^1\odot \Res_{z_1=0}\Res_{z_2=0} Y_{M^2}(a,z_1)Y_{M^2}(b,z_2)v_2 (-1+z_1)^{\wt a-1} (-1+z_2)^{\wt b-1} \\
		&\ \ \ -\sum_{j\geq 1}\sum_{k\geq 0} \binom{-1-j}{k} a(\wt a+k)v^1\odot \Res_{z_2=0} Y_{M^2}(b,z_2)v_2 (-1+z_2)^{\wt b-1}\\
		&=o(a)v^1\odot  \Res_{z_2=0} Y_{M^2}(b,z_2)v_2 (-1+z_2)^{\wt b-1}\\
		&\ \ \ +v^1\odot \Res_{z_1=0}\Res_{z_2=0} Y_{M^2}(a,z_1)Y_{M^2}(b,z_2)v_2 (-1+z_1)^{\wt a-1} (-1+z_2)^{\wt b-1},
	\end{align*}
	noting that $a(\wt a+k)v^1=0$ for $k\geq 0$.
	Using \eqref{eq:actionfogeneralspann} again, we can express $(C)$ as follows:
	\begin{align*}
		(C)&=-\sum_{j\geq 0}\rho_1\left(b\o \frac{z^{-1-j}}{(z-1)^{-\wt b+1}}\right)(v^1)\odot \Res_{z_1=0}Y_{M^2}(a,z_1)v_2(-1+z_1)^{\wt a} z_1^j\\
		&=-\sum_{j\geq 0}\sum_{k\geq 0}\binom{-1-j}{k} b(\wt b-1+k)v^1\odot \Res_{z_1=0}Y_{M^2}(a,z_1)v_2(-1+z_1)^{\wt a} z_1^j\\
		&=o(b)v^1\odot \Res_{z_1=0}Y_{M^2}(a,z_1)v_2(-1+z_1)^{\wt a}\cdot  \frac{1}{-1+z_1}\\
		&=o(b)v^1\odot \Res_{z_1=0}Y_{M^2}(a,z_1)v_2(-1+z_1)^{\wt a-1}.
	\end{align*}
	Now it follows from \eqref{eq:altaction2} that
	\begin{align*}
		&([a]\ast [b]).(v^1\odot v_2)=(A)+(B)+(C)\\
		&=o(a)o(b)v^1\odot v_2+ o(a)v^1\odot  \Res_{z_2=0} Y_{M^2}(b,z_2)v_2 (-1+z_2)^{\wt b-1}\\
		&\ \ \ + v^1\odot \Res_{z_1=0}\Res_{z_2=0} Y_{M^2}(a,z_1)Y_{M^2}(b,z_2)v_2 (-1+z_1)^{\wt a-1} (-1+z_2)^{\wt b-1}\\
		&\ \ \ +o(b)v^1\odot \Res_{z_1=0}Y_{M^2}(a,z_1)v_2(-1+z_1)^{\wt a-1}\\
		&= [a].\left([b].(v^1\odot v_2)\right).
	\end{align*}
	This shows that $M^1\odot M^2$ is a left $A(V)$-module with respect to the action \eqref{eq:altaction2}.
\end{proof}

\begin{remark}
	Since the $A(V)$-action \eqref{eq:AVactiononodot} (or \eqref{eq:altaction2}) on $M^1\odot M^2$ is induced by the $L(V)_0$-action \eqref{eq: L(V)0action}, Theorem~\ref{thm:AVaction} also shows that the Lie algebra homomorphism $\overline{\rho_{1,0}}: L(V)_0\ra \gl(M^1\odot M^2)$ factors through the Lie algebra epimorphism $L(V)_0\ra A(V)_{\mathrm{Lie}},\ a_{[\wt a-1]}\mapsto [a]$:
	$$
	\begin{tikzcd}
		L(V)_0\arrow[r,"\rho"]\arrow[d,two heads]& \gl(M^1\odot M^2),\\
		A(V)_{\mathrm{Lie}}\arrow[ur,dashed,"\varphi"']
	\end{tikzcd}
	$$
	where $\rho$ is given by \eqref{eq: L(V)0action} and $\varphi$ is given by \eqref{eq:AVactiononodot}. In particular, under the assumptions of Theorem~\ref{thm:AVaction}, we have
	\begin{equation}\label{eq:homspaceequality}
		\Hom_{L(V)_0}(M^1\odot M^2, \Om(M^3))=\Hom_{A(V)}(M^1\odot M^2, \Om(M^3)),
	\end{equation}
	where $\Om(M^3)$ is the bottom degree of an admissible $V$-module $M^3$.
\end{remark}

\begin{coro}\label{corc:comparison}
	Assume that the $V$-module $M^2$ is generated by $\Om(M^2)$. The following map is an epimorphism of left $A(V)$-modules:
	\begin{equation}\label{eq:defpi}
		\pi: A(M^1)\o_{A(V)} \Om(M^2)\twoheadrightarrow M^1\odot M^2, \quad [v_1]\o v^2\mapsto v_1\odot v^2.
	\end{equation}
	In particular, we have the estimate
	\begin{equation}\label{eq:dimestimate}
		\dim \Hom_{L(V)_0}(M^1\odot M^2, \Om(M^3))\leq \dim \Hom_{A(V)}(A(M^1)\o_{A(V)} \Om(M^2), \Om(M^3)).
	\end{equation}
\end{coro}
\begin{proof}
	To show that $\pi$ is well-defined, we need to show that $\pi([O(M^1)]\otimes v^2)=0$ and that $\pi([v_1]\ast [a]\o v^2)=\pi(v_1\o o(a)v^2)$. Indeed, it follows from \eqref{eq:OM1odot} that
	\[
		\pi([O(M^1)]\otimes v^2)=O(M^1)\odot v^2=0.
	\]
	Moreover, by \eqref{eq:simplifiedaction}, we have
	\begin{align*}
		&\pi([v_1]\ast [a]\o v^2)=\pi([v_1\ast a]\o v^2)=v_1\ast a \odot v^2=v_1\odot o(a)v^2=\pi (v_1\o o(a)v^2),\\
		& \pi( [a].([v_1]\o v^2))=\pi([a\ast v_1]\o v^2)= (a\ast v_1)\odot v^2=[a].(v_1\odot v^2)=[a].\pi([v_1]\o v^2).
	\end{align*}
	Hence $\pi$ in \eqref{eq:defpi} is a well-defined $A(V)$-module homomorphism. By Lemma~\ref{lm:spannodot}, it is surjective. Finally, the estimate \eqref{eq:dimestimate} follows from \eqref{eq:homspaceequality}.
\end{proof}

\begin{remark}
	In Section \ref{Sec:6}, we will show that the epimorphism $\pi$ in \eqref{eq:defpi} is not necessarily injective, and so the estimate \eqref{eq:dimestimate} is sharp for certain examples of VOAs. On the other hand, if the VOA $V$ is rational and $C_2$-cofinite, then $\pi$ is in fact an isomorphism of left $A(V)$-modules.
\end{remark}


		\section{Extension of the $\infty$-restricted conformal blocks}\label{sec:5}
		In this Section, we show that an $\infty$-restricted conformal block $\varphi$ in $\mathscr{C}_{\mathrm{res}}\left(\Sigma(\Om((M^3)'), M^1, M^2)\right)$ can be extended to a regular three-pointed conformal block $\widetilde{\varphi}$ in $\mathscr{C}\left(\Sigma((M^3)', M^1, M^2)\right)$, where $(M^3)'\cong \bar{M}(\Om((M^3)'))$. This will lead to our Hom-space description $\Hom_{L(V)_0}(M^1\odot M^2, \Om(M^3))$ of the space of intertwining operators $I\fusion{M^1}{M^2}{M^3}$.

\subsection{Construction of the extended conformal blocks}

The following lemma shows that any regular conformal block can be restricted to an $\infty$-restricted conformal block. It also gives an estimate of the fusion rule, generalizing the estimate in \cite[Proposition 2.10]{Li99}.

\begin{lm}\label{lm:estimate}
	Assume the $V$-module $(M^3)'$ is generated by its bottom degree $\Om((M^3)')$, and $\Om((M^3)')^\ast\cong\Om(M^3)$. Then the restriction map
	\begin{equation}\label{eq:cfbrestriction}
		\begin{aligned}
			G:\mathscr{C}\left(\Sigma( (M^3)', M^1, M^2)\right)\ra  \mathscr{C}_{\mathrm{res}}\left(\Sigma(\Om((M^3)'), M^1, M^2)\right), \quad \psi\mapsto \psi|_{\Om((M^3)')\o M^1\o M^2},
		\end{aligned}
	\end{equation}
	is injective. In particular, $N\fusion{M^1}{M^2}{M^3}\leq \dim \Hom_{L(V)_0}(M^1\odot M^2,\Om(M^3))$.
\end{lm}
\begin{proof}
	By the remark in Section~\ref{sec:2.1.5}, $(M^3)'$ is spanned by $$\{b'(n)v'_3: b\in V, n\in \Z, v'_3\in \Om((M^3)')\},$$
    where
$
		b'(n)=\sum_{j\geq 0}\frac{(-1)^{\wt b}}{j!} (L(1)^jb)(2\wt b-n-j-2).
$
	Now assume $\psi\in \mathscr{C}\left(\Sigma( (M^3)', M^1, M^2)\right)$ satisfies $\psi|_{\Om((M^3)')\o M^1\o M^2}=0$. Then by \eqref{eq:actionofchiralliealgebraelements},
	\begin{align*}
		&\psi(b'(n)v'_3\o v_1\o v_2)\\
		&=-\psi\left((b\o z^n).(v'_3\o v_1\o v_2)\right)
		+\psi\!\left(v'_3\o \Res_{z=1}Y_{M^1}(b,z-1)v_1\o v_2\right)\iota_{1,z-1}(z^n)\\
		&\ \ \ + \psi\!\left(v'_3\o v_1\o \Res_{z=0} Y_{M^2}(b,z)v_2\right) z^n\\
		&=0,
	\end{align*}
	where the last equality follows from the facts that $\psi$ is invariant under the action of $b\o z^n$, and $\psi|_{\Om((M^3)')\o M^1\o M^2}=0$.
	Hence $\psi=0$ on $(M^3)'\o M^1\o M^2$, i.e., $G$ is injective. The estimate of the fusion rule follows from Proposition~\ref{prop:IOconformalblocks} and Proposition~\ref{prop:inftycfbHom}.
\end{proof}

Let $M^3$ be a $V$-module with bottom degree $\Om(M^3)$, and let $\bar{M}(\Om((M^3)'))$ be the generalized Verma module associated to the left $A(V)$-module $\Om((M^3)')$ via the contragredient action \eqref{eq:contraleftA(V)}.
Note that
\[
	\bar{M}(\Om((M^3)'))=\spn\{b'(n)v'_3: b\in V, n\in \Z, v'_3\in \Om((M^3)')\}.
\]

\begin{construction}
	Given an $\infty$-restricted conformal block $\varphi\in \mathscr{C}_{\mathrm{res}}\left(\Sigma(\Om((M^3)'), M^1, M^2)\right)$, define a linear map
$
		\widetilde{\varphi}: \bar{M}(\Om((M^3)'))\otimes_\C M^1\otimes_\C M^2\ra \C
$
	by setting
	\begin{equation}\label{eq:defextenedphi}
		\begin{aligned}
			\braket*{\widetilde{\varphi}}{b'(n)v'_3\o v_1\o v_2}
			&:=\braket*{\varphi}{v'_3\otimes \Res_{z=1}Y_{M^1}(b,z-1)v_1 \otimes v_2}\iota_{1,z-1}(z^n)\\
			&\ \ \ +\braket*{\varphi}{v'_3\otimes v_1\otimes \Res_{z=0} Y_{M^2}(b,z)v_2} z^n,
		\end{aligned}
	\end{equation}
	where $b\in V$, $n\in \Z$, and $v'_3\in \Om((M^3)')$.

	For a general spanning element $b'_1(n_1)b'_2(n_2)\ds b'_r(n_r)v'_3$ of $\bar{M}(\Om((M^3)'))$, where $b_i\in V$ and $n_i\in \Z$, with $-\wt b_i+n_i+1\geq 0$ for all $i$, we define the evaluation of $\wphi$ on $b'_1(n_1)b'_2(n_2)\ds b'_r(n_r)v'_3\o v_1\o v_2$ inductively by the following formula:
	\begin{equation}\label{eq:defextensionphi2}
		\begin{aligned}
			&\braket*{\wphi}{b'_1(n_1)b'_2(n_2)\ds b'_r(n_r)v'_3\o v_1\o v_2}\\
			&=\braket*{\wphi}{b'_2(n_2)\ds b'_r(n_r)v'_3\o \Res_{z=1}Y_{M^1}(b_1,z-1)v_1\o v_2}\iota_{1,z-1}(z^{n_1})\\
			&\ \ \ +\braket*{\wphi}{b'_2(n_2)\ds b'_r(n_r)v'_3\o v_1\o \Res_{z=0}Y_{M^2}(b_1,z)v_2}z^{n_1}.
		\end{aligned}
	\end{equation}
\end{construction}

\begin{remark}
	Observe that if $\varphi$ is an element of $\mathscr{C}\left(\Sigma(\bar{M}(\Om((M^3)')), M^1, M^2)\right)$, then it must satisfy the following equality by \eqref{eq:actionofchiralliealgebraelements}:
	\begin{equation}\label{eq:defextensionedphi0}
		\begin{aligned}
			&\braket*{\varphi}{\Res_{z=\infty} Y_{\bar{M}(\Om((M^3)'))}(\iota(b),z^{-1})v'_3\o v_1\o v_2}\iota_{z,1}\left(\frac{z^n}{(z-1)^m}\right)\\
			&=\braket*{\varphi}{v'_3\otimes \Res_{z=1}Y_{M^1}(b,z-1)v_1 \otimes v_2}\iota_{1,z-1}\left(\frac{z^n}{(z-1)^m}\right)\\
			&\ \ \ +\braket*{\varphi}{v'_3\otimes v_1\otimes \Res_{z=0} Y_{M^2}(b,z)v_2}\iota_{1,z}\left(\frac{z^n}{(z-1)^m}\right),
		\end{aligned}
	\end{equation}
	where $b\in V$, $v'_3\in\bar{M}(\Om((M^3)'))$, and $m,n\in \Z$. We construct $\widetilde{\varphi}$ in \eqref{eq:defextenedphi} and \eqref{eq:defextensionphi2} so that this equality holds.

	The only relations among the spanning elements $b'_1(n_1)b'_2(n_2)\ds b'_r(n_r)v'_3$ of the generalized Verma module $\bar{M}(\Om((M^3)'))$ are given by the Jacobi identity \cite{DLM98}. Note that \eqref{eq:defextensionedphi0} is essentially a variation of the Jacobi identity, which, in turn, also defines $\wphi$ in \eqref{eq:defextenedphi} and \eqref{eq:defextensionphi2}. Hence $\wphi$ is a well-defined element of $\left(\bar{M}(\Om((M^3)'))\otimes_\C M^1\otimes_\C M^2\right)^\ast$.
\end{remark}

		

	We want to show that $\widetilde{\varphi}$ is an element of the space of conformal blocks associated to the datum $\Sigma(\bar{M}(\Om((M^3)')), M^1, M^2)$. In other words, it is also invariant under the action of the chiral Lie algebra $\mathcal{L}_{\P^1\bs\{\infty,1,0\}}(V)$; see Definition~\ref{def:cfb}. For this purpose, we need the following lemma about the spanning elements of $\mathcal{L}_{\P^1\bs\{\infty,1,0\}}(V)$. The proof is similar to that of Proposition~\ref{prop:spanningelt}, and we omit the details.


\begin{lm}\label{lm:spanningchiralLie}
	The chiral Lie algebra $\mathcal{L}_{\P^1\bs\{\infty,1,0\}}(V)$ is spanned by the following elements:
	\begin{equation}\label{}
		a\o \frac{z^{\wt a-1}}{z-1},\quad a\o z^{\wt a-1},\quad a\o z^{\wt a-k},\quad a\o z^{\wt a+l},
	\end{equation}
	where $a\in V$ is homogeneous, $k\geq 2$, and $l\geq 0$.
\end{lm}

The following table for the pairs $(n,m)$ illustrates the spanning elements of $\mathcal{L}_{\P^1\bs\{\infty,1,0\}}(V)$:
\begin{equation}\label{}
	\begin{matrix}
		\ds &\vdots &\vdots&\vdots&\vdots&\vdots&\textcolor{red}{\ds}\\
		\ds & (\wt a-3, -4) & (\wt a-2,-3)& (\wt a-1,-2) &(\wt a,-1) & \textcolor{red}{(\wt a+1,0)} &\ds \\
		\ds &(\wt a-3,-3)& (\wt a-2,-2)&(\wt a-1,-1)&\textcolor{red}{(\wt a,0)}& (\wt a+1,1)& \ds \\
		\ds & (\wt a-3,-2) &(\wt a-2,-1)&\textcolor{red}{(\wt a-1,0)}& (\wt a,1)&(\wt a+1,2)& \ds \\
		\ds &(\wt a-3,-1)&\textcolor{red}{(\wt a-2,0)}& \textcolor{red}{(\wt a-1,1)}& (\wt a,2)& (\wt a+1,3)&\ds \\
		\ds &\textcolor{red}{(\wt a-3,0)}& (\wt a-2,1)& (\wt a-1, 2) & (\wt a,3)& (\wt a+1,4)&\ds \\
		\textcolor{red}{\ds} & \vdots &\vdots & \vdots&\vdots & \vdots &\dots
	\end{matrix}
\end{equation}

\subsection{The extension theorem for $\infty$-restricted conformal blocks}

We now prove our main theorem concerning the extension of $\infty$-restricted conformal blocks. A more general twisted version was proved in \cite[Theorems 5.18, 5.19]{GLZ24} using the Riemann--Roch theorem for algebraic curves. Here we give a purely algebraic proof.

\begin{thm}\label{thm:main}
	Let $\varphi\in \mathscr{C}_{\mathrm{res}}\left(\Sigma(\Om((M^3)'), M^1, M^2)\right)$. Then $\wphi$, given by \eqref{eq:defextenedphi} and \eqref{eq:defextensionphi2}, is invariant under the action of the three-pointed chiral Lie algebra $\mathcal{L}_{\P^1\bs\{\infty,1,0\}}(V)$.
	In particular, we have an isomorphism of vector spaces
	\begin{equation}\label{eq:isomofconformalblocks}
		F:\mathscr{C}_{\mathrm{res}}\left(\Sigma(\Om((M^3)'), M^1, M^2)\right)\cong \mathscr{C}\left(\Sigma(\bar{M}(\Om((M^3)')), M^1, M^2)\right),\quad \varphi\mapsto \widetilde{\varphi},
	\end{equation}
	whose inverse is the restriction map $G$ in \eqref{eq:cfbrestriction}.
\end{thm}

	\begin{proof}
	We need to show $\braket*{\wphi}{\mathcal{L}_{\P^1\bs\{\infty,1,0\}}(V). (\bar{M}(\Om((M^3)'))\otimes_\C M^1\otimes_\C M^2)}=0$. 
	By Lemma~\ref{lm:spanningchiralLie}, it suffices to show that 
	\begin{align}
		&\braket*{\wphi}{\left(a\otimes \frac{z^{\wt a-1}}{z-1}\right). (b'(n)v'_3\o v_1\o v_2)}=0, \label{eq:inv-1}\\
		&\braket*{\wphi}{\left(a\otimes z^{\wt a-1}\right).(b'(n)v'_3\o v_1\o v_2)}=0,\label{eq:inv-2}\\
		& \braket*{\wphi}{\left(a\otimes z^{\wt a-k}\right). (b'(n)v'_3\o v_1\o v_2)}=0,\quad k\geq 2,\label{eq:invgeneral1}\\
		& \braket*{\wphi}{\left(a\otimes z^{\wt a+l}\right). (b'(n)v'_3\o v_1\o v_2)}=0,\quad l\geq 0, \label{eq:invgeneral2}
	\end{align}
	for any homogeneous $a\in V$ and any $b'(n)v'_3\o v_1\o v_2\in  \bar{M}(\Om((M^3)'))\otimes_\C M^1\otimes_\C M^2$. To simplify notation, we omit the 
$dz$ symbol in the residue calculations below.

	Case I. Proof of \eqref{eq:inv-1}. 
	
	For any $v'_3\in \Om((M^3)')$, since $\deg (a'(\wt a-j-2))=-\wt a+\wt a-j-2+1=-j-1<0$ for any $j\geq 0$, we have 
	$$\Res_{z=\infty} Y_{\bar{M}(\Om((M^3)'))}(\vartheta(a),z^{-1})v'_3\iota_{z,1}\left(\frac{z^{\wt a-1}}{z-1}\right)=-\sum_{j\geq 0} a'(\wt a-j-2)v'_3=0.$$
	Moreover, let $M$ be a $V$-module, and let $M'$ be the contragredient module of $M$. For any $v'\in M'$, $v\in M$, $a,b\in V$, and $m,n\in \Z$, by the Jacobi identity for $Y_M$, we have 
	\begin{equation}\label{eq:primebracket}
		\begin{aligned}
			\<[b'(n),a'(m)]v',v\>&=\<v', [a(m),b(n)v]\>=-\sum_{i\geq 0}\binom{n}{i}\<v', (b(i)a)(m+n-i)v\>\\
			&=-\sum_{i\geq 0}\<(b(i)a)'(m+n-i)v',v\>.
		\end{aligned}
	\end{equation}
	It follows that for any $v'_3\in \Om((M^3)')$, 
	\begin{align*}
		&\Res_{z=\infty} Y_{\bar{M}(\Om((M^3)'))}(\vartheta(a),z_2^{-1})b'(n)v'_3 \iota_{z,1}\left(\frac{z^{\wt a-1}}{z-1}\right)\\
		&=\Res_{z=\infty}\left( b'(n)Y_{\bar{M}(\Om((M^3)'))}(\vartheta(a),z_2^{-1})v'_3- [b'(n),  Y_{\bar{M}(\Om((M^3)'))}(\vartheta(a),z_2^{-1})]v'_3 \right) \iota_{z,1}\left(\frac{z^{\wt a-1}}{z-1}\right)\\
		&=\Res_{z=\infty} Y_{\bar{M}(\Om((M^3)'))}(\vartheta(b(i)a),z_2^{-1})v'_3\iota_{z,1}\left(\frac{z^{\wt a-1}}{z-1}\right).
	\end{align*}
	Then, by \eqref{eq:defextensionedphi0} and \eqref{eq:actionofchiralliealgebraelements}, together with the equation above, we have 
	\begin{align*}
		&\braket*{\wphi}{\left(a\otimes \frac{z^{\wt a-1}}{z-1}\right). (b'(n)v'_3\o v_1\o v_2)}\\
		&=\braket*{\wphi}{\Res_{z=\infty} Y_{\bar{M}(\Om((M^3)'))}(\vartheta(a),z_2^{-1})b'(n)v'_3 \o v_1\o v_2}\iota_{z_2,1}\left(\frac{z_2^{\wt a-1}}{z_2-1}\right)\\
		&\ \ \ + \braket*{\wphi}{b'(n)v'_3\otimes \Res_{z_2=1}Y_{M^1}(a,z_2-1)v_1\o v_2} \iota_{1,z_2-1}\left(\frac{z_2^{\wt a-1}}{z_2-1}\right)\\
		&\ \ \ + \braket*{\wphi}{b'(n)v'_3\o v_1\o \Res_{z_2=0} Y_{M^2}(a,z_2)v_2}\iota_{1,z_2}\left(\frac{z_2^{\wt a-1}}{z_2-1}\right)\\
		&=\sum_{i\geq 0} \binom{n}{i}\braket*{\wphi}{\Res_{z=\infty} Y_{\bar{M}(\Om((M^3)'))}(\vartheta(b(i)a),z_2^{-1})v'_3\o v_1\o v_2} z_2^{n-i}\iota_{z_2,1}\left(\frac{z_2^{\wt a-1}}{z_2-1}\right)\\
		&\ \ \ +\braket*{\varphi}{v'_3 \o \Res_{z_1=1}\Res_{z_2=1}Y_{M^1}(b,z_1-1)Y_{M^1}(a,z_2-1)v_1\o v_2} \iota_{1,z_1-1}(z_1^n)\iota_{1,z_2-1}\left(\frac{z_2^{\wt a-1}}{z_2-1}\right)\\
		&\ \ \ +\braket*{\varphi}{v'_3\o \Res_{z_2=1}Y_{M^1}(a,z_2-1)v_1\o \Res_{z_1=0} Y_{M^2}(b,z_1)v_2} z_1^n \iota_{1,z_2-1}\left(\frac{z_2^{\wt a-1}}{z_2-1}\right)\\
		&\ \ \ + \braket*{\varphi}{v'_3\o \Res_{z_1=1}Y_{M^1}(b,z_1-1)v_1\o \Res_{z_2=0} Y_{M^2}(a,z_2)v_2} \iota_{1,z_1-1}(z_1^n) \iota_{1,z_2}\left(\frac{z_2^{\wt a-1}}{z_2-1}\right)\\
		&\ \ \ +\braket*{\varphi}{v'_3\o v_1\o \Res_{z_1=0}\Res_{z_2=0}Y_{M^2}(b,z_1)Y_{M^2}(a,z_2)v_2}z_1^n \iota_{1,z_2}\left(\frac{z_2^{\wt a-1}}{z_2-1}\right)\\
		&= -\underbrace{\sum_{i\geq 0}\binom{n}{i}\braket*{\varphi}{v'_3\otimes \Res_{z_2=1}Y_{M^1}(b(i)a,z_2-1)v_1\otimes v_2} \iota_{1,z_2-1}\left(\frac{z_2^{\wt a-1+n-i}}{z_2-1}\right)}_{(A)}\\
		&\ \ \ -\underbrace{\sum_{i\geq 0}\binom{n}{i}\braket*{\varphi}{v'_3\otimes v_1\otimes \Res_{z_2=0}Y_{M^2}(b(i)a,z_2)v_2} \iota_{1,z_2}\left(\frac{z_2^{\wt a-1+n-i}}{z_2-1}\right)}_{(B)}\\
		&\ \ \ +\underbrace{\braket*{\varphi}{v'_3 \o \Res_{z_1=1}\Res_{z_2=1}Y_{M^1}(b,z_1-1)Y_{M^1}(a,z_2-1)v_1\o v_2} \iota_{1,z_1-1}(z_1^n)\iota_{1,z_2-1}\left(\frac{z_2^{\wt a-1}}{z_2-1}\right)}_{(C)}\\
		&\ \ \ -\underbrace{\braket*{\varphi}{v'_3\o v_1\o \Res_{z_1=0} \Res_{z_2=0} Y_{M^2}(a,z_2) Y_{M^2}(b,z_1)v_2} z_1^n \iota_{1,z_2}\left(\frac{z_2^{\wt a-1}}{z_2-1}\right)}_{(D)}\\
		&\ \ \ - \underbrace{\braket*{\varphi}{v'_3\o \Res_{z_2=1}\Res_{z_1=1}Y_{M^1}(a,z_2-1)Y_{M^1}(b,z_1-1)v_1\o  v_2} \iota_{1,z_1-1}(z_1^n) \iota_{1,z_2-1}\left(\frac{z_2^{\wt a-1}}{z_2-1}\right)}_{(E)}\\
		&\ \ \ +\underbrace{\braket*{\varphi}{v'_3\o v_1\o \Res_{z_1=0}\Res_{z_2=0}Y_{M^2}(b,z_1)Y_{M^2}(a,z_2)v_2}z_1^n \iota_{1,z_2}\left(\frac{z_2^{\wt a-1}}{z_2-1}\right)}_{(F)}\\
		&=-(A)-(B)+(C)-(D)-(E)+(F),
	\end{align*}
	where the last equality follows from the fact that $\varphi\in  \mathscr{C}_{\mathrm{res}}\left(\Sigma(\Om((M^3)'), M^1, M^2)\right)$ is invariant under the action of $a\otimes \frac{z^{\wt a-1}}{z-1}\in \mathcal{L}_{\P^1\bs\{0,1\}}(V)_{\leq 0}$. Using the Jacobi identity for $Y_{M^1}$, we have 
	
	\begin{align*}
		&(C)-(E)\\
		&=\braket*{\varphi}{v'_3 \o \Res_{z_1=1}\Res_{z_2=1}Y_{M^1}(b,z_1-1)Y_{M^1}(a,z_2-1)v_1\o v_2} \iota_{1,z_1-1}(z_1^n)\iota_{1,z_2-1}\left(\frac{z_2^{\wt a-1}}{z_2-1}\right)\\
		&\ \ \ - \braket*{\varphi}{v'_3\o \Res_{z_2=1}\Res_{z_1=1}Y_{M^1}(a,z_2-1)Y_{M^1}(b,z_1-1)v_1\o  v_2} \iota_{1,z_1-1}(z_1^n) \iota_{1,z_2-1}\left(\frac{z_2^{\wt a-1}}{z_2-1}\right)\\
		&=\braket*{\varphi}{v'_3\o \Res_{z_2=1}\Res_{z_1-z_2=0} Y_{M^1}(Y(b,z_1-z_2)a,z_2-1)v_1\o v_2} \iota_{z_2,z_1-z_2}(z^n_1) \iota_{1,z_2-1}\left(\frac{z_2^{\wt a-1}}{z_2-1}\right)\\
		&=\sum_{i\geq 0} \binom{n}{i}\braket*{\varphi}{v'_3\o \Res_{z_2=1} Y_{M^1}(b(i)a,z_2-1)v_1\otimes v_2} \iota_{1,z_2-1}\left(\frac{z_2^{\wt a-1+n-i}}{z_2-1}\right)\\
		&=(A).  
	\end{align*}
	Similarly, using the Jacobi identity for $Y_{M^2}$, we have
	\begin{align*}
		&-(D)+(F)\\
		&=-\braket*{\varphi}{v'_3\o v_1\o \Res_{z_1=0} \Res_{z_2=0} Y_{M^2}(a,z_2) Y_{M^2}(b,z_1)v_2} z_1^n \iota_{1,z_2}\left(\frac{z_2^{\wt a-1}}{z_2-1}\right)\\
		& \ \ \ +\braket*{\varphi}{v'_3\o v_1\o \Res_{z_1=0}\Res_{z_2=0}Y_{M^2}(b,z_1)Y_{M^2}(a,z_2)v_2}z_1^n \iota_{1,z_2}\left(\frac{z_2^{\wt a-1}}{z_2-1}\right)\\
		&= \braket*{\varphi}{v'_3\o v_1\o \Res_{z_2=0}\Res_{z_1-z_2=0} Y_{M^2}(Y(b,z_1-z_2)a,z_2)v_2 } \iota_{z_2,z_1-z_2}(z_1^n) \iota_{1,z_2}\left(\frac{z_2^{\wt a-1}}{z_2-1}\right)\\
		&=\sum_{i\geq 0}\binom{n}{i}\braket*{\varphi}{v'_3\otimes v_1\otimes \Res_{z_2=0}Y_{M^2}(b(i)a,z_2)v_2} \iota_{1,z_2}\left(\frac{z_2^{\wt a-1+n-i}}{z_2-1}\right)\\
		&=(B).
	\end{align*}
	This shows \eqref{eq:inv-1} because
	$$\braket*{\wphi}{\left(a\otimes \frac{z^{\wt a-1}}{z-1}\right). (b'(n)v'_3\o v_1\o v_2)}=-(A)-(B)+(C)-(D)-(E)+(F)=0.$$
	
	Case II. Proof of \eqref{eq:inv-2}. It follows from \eqref{eq:defextenedphi} that
	\begin{align*}
		&\braket*{\wphi}{\left(a\otimes z^{\wt a-1}\right).(b'(n)v'_3\o v_1\o v_2)}\\
		&=-\underbrace{\braket*{\wphi}{a'(\wt a-1)b'(n)v'_3\o v_1\o v_2}}_{(G)} +\underbrace{\braket*{\wphi}{b'(n)v'_3\o \Res_{z_2=1}Y_{M^1}(a,z_2-1)v_1\o v_2} \iota_{1,z_2-1}(z_2^{\wt a-1})}_{(H)}\\
		&\ \ + \underbrace{\braket*{\wphi}{b'(n)v'_3\o v_1\o \Res_{z_2=0} Y_{M^2}(a,z_2)v_2} z_2^{\wt a-1}}_{(I)}\\
		&=-(G)+(H)+(I).
	\end{align*}
	Note that $a'(\wt a-1)v'_3=o(a)v'_3$ for $v'_3\in \Om((M^3)')$. Then, by \eqref{eq:primebracket} and \eqref{eq:defextenedphi}, we have
	\begin{align*}
		(G)&=\braket*{\wphi}{b'(n)o(a)v'_3\o v_1\o v_2}+\sum_{i\geq 0}\binom{n}{i}\braket*{\wphi}{(b(i)a)'(\wt a-1+n-i)v'_3\o v_1\o v_2}\\
		&=\underbrace{\braket*{\varphi}{o(a)v'_3\o \Res_{z_1=1}Y_{M^1}(b,z_1-1)v_1\o v_2} \iota_{1,z_1-1}(z_1^n)}_{(G1)}\\
		&\ \ \ +\underbrace{\braket*{\varphi}{o(a)v'_3\o v_1\o \Res_{z_1=0} Y(b,z_1)v_2}z_1^n}_{(G2)} \\
		&\ \ \ +\underbrace{\braket*{\varphi}{v'_3\o \Res_{z_2=1}\Res_{z_1-z_2=0} Y_{M^1}(Y(b,z_1-z_2)a,z_2-1)v_1\o v_2}\iota_{z_2,z_1-z_2}(z_1^n)z_2^{\wt a-1}}_{(G3)}\\
		&\ \ \ +\underbrace{\braket*{\varphi}{v'_3\o v_1\o \Res_{z_2=0}\Res_{z_1-z_2=0} Y_{M^2}(Y(b,z_1-z_2)a,z_2)v_2}\iota_{z_2,z_1-z_2}(z_1^n)z_2^{\wt a-1}}_{(G4)}\\
		&=(G1)+(G2)+(G3)+(G4). 
	\end{align*}
	On the other hand, using the invariance of $\varphi$ under the action of $a\otimes z^{\wt a-1}\in \mathcal{L}_{\P^1\bs\{0,1\}}(V)_{\leq 0}$, together with \eqref{eq:defextenedphi}, we have 
	
	\begin{align*}
		(H)&=\braket*{\varphi}{v'_3\o \Res_{z_1=1}\Res_{z_2=1}Y_{M^1}(b,z_1-1)Y_{M^1}(a,z_2-1)v_1\o v_2} \iota_{1,z_2-1}(z_2^{\wt a-1})\iota_{1,z_1-1}(z_1^{n})\\
		&\ \ \ +\braket*{\varphi}{v'_3\o \Res_{z_2=1}Y_{M^1}(a,z_2-1)v_1\o \Res_{z_1=0}Y_{M^2}(b,z_1)v_2} \iota_{1,z_2-1}(z_2^{\wt a-1}) z_1^n \\
		&= \underbrace{\braket*{\varphi}{v'_3\o \Res_{z_1=1}\Res_{z_2=1}Y_{M^1}(b,z_1-1)Y_{M^1}(a,z_2-1)v_1\o v_2} \iota_{1,z_2-1}(z_2^{\wt a-1})\iota_{1,z_1-1}(z_1^{n})}_{(H1)}\\
		&\ \ \ +\underbrace{\braket*{\varphi}{o(a)v'_3\o v_1\o \Res_{z_1=0} Y_{M^2}(b,z_1)v_2}z_1^n}_{(H2)}\\
		&\ \ \ -\underbrace{\braket*{\varphi}{v'_3\o v_1\o \Res_{z_2=0} \Res_{z_1=0}Y_{M^2}(a,z_2)Y_{M^1}(b,z_1)v_2}z_1^n z_2^{\wt a-1}}_{(H3)}\\
		&=(H1)+(H2)-(H3). \\
		(I)&=\braket*{\varphi}{v'_3\o \Res_{z_1=1}Y_{M^1}(b,z_1-1)v_1\o \Res_{z_2=0} Y_{M^2}(a,z_2)v_2} z_2^{\wt a-1} \iota_{1,z_1-1}(z_1^n)\\
		&\ \ \ + \braket*{\varphi}{v'_3\o v_1\o \Res_{z_1=0}\Res_{z_2=0} Y_{M^2}(b,z_1)Y_{M^1}(a,z_2)v_2} z_2^{\wt a-1} z_1^n\\
		&=\underbrace{\braket*{\varphi}{o(a)v'_3\o \Res_{z_1=1}Y_{M^1}(b,z_1-1)v_1\o v_2} \iota_{1,z_1-1}(z_1^n)}_{(I1)}\\
		&\ \ \ - \underbrace{\braket*{\varphi}{v'_3\o \Res_{z_2=1}\Res_{z_1=1} Y_{M^1}(a,z_2-1)Y_{M^1}(b,z_1-1)v_1\o v_2} \iota_{1,z_2-1}(z_2^{\wt a-1}) \iota_{1,z_1-1}(z_1^{n})}_{(I2)}\\
		&\ \ \ + \underbrace{\braket*{\varphi}{v'_3\o v_1\o \Res_{z_1=0}\Res_{z_2=0} Y_{M^2}(b,z_1)Y_{M^1}(a,z_2)v_2} z_2^{\wt a-1} z_1^n}_{(I3)}\\
		&=(I1)-(I2)+(I3).
	\end{align*}
	Note that $-(G1)+(I1)=0$ and $-(G2)+(H2)=0$. By the Jacobi identity for $Y_{M^1}$ and $Y_{M^2}$, we have $(H1)-(I2)=(G3)$ and $-(H3)+(I3)=(G4)$. 
	\begin{align*}
		&\braket*{\wphi}{\left(a\otimes z^{\wt a-1}\right).(b'(n)v'_3\o v_1\o v_2)}=-(G)+(H)+(I)\\
		&=-(G1)-(G2)-(G3)-(G4)+(H1)+(H2)-(H3)+(I1)-(I2)+(I3)=0.
	\end{align*}
	
	Case III. Proof of \eqref{eq:invgeneral1}. 
	
	Let $k\geq 2$. Note that $a'(\wt a-k)v'_3=0$ for  $v'_3\in \Om((M^3)')$. Similar to the argument above, using \eqref{eq:primebracket}, we have
	
	\begin{align*}
		& \braket*{\wphi}{\left(a\otimes z^{\wt a-k}\right). (b'(n)v'_3\o v_1\o v_2)}\\
		&=-\braket*{\wphi}{\sum_{i\geq 0}\binom{n}{i} (b(i)a)'(\wt a-k+n-i)v'_3\o v_1\o v_2}\\
		&\ \ \ +\braket*{\wphi}{b'(n)v'_3\o \Res_{z_2=1}Y_{M^1}(a,z_2-1)v_1\o v_2} \iota_{1,z_2-1}(z_2^{\wt a-k})\\
		&\ \ \ +\braket*{\wphi}{b'(n)v'_3\o v_1\o \Res_{z_2=0} Y_{M^2}(a,z_2)v_2} z_2^{\wt a-k}\\
		&= -\sum_{i\geq 0}\binom{n}{i}\braket*{\varphi}{v'_3\o \Res_{z_2=1} Y_{M^1}(b(i)a,z_2-1) v_1\o v_2} \iota_{1,z_2-1}(z_2^{\wt a-k+n-i})\\
		&\ \ \ -\sum_{i\geq 0} \binom{n}{i} \braket*{\varphi}{v'_3\o v_1\o \Res_{z_2=0} Y_{M^2}(b(i)a,z_2)v_2} z_2^{\wt a-k+n-i}\\
		&\ \ \ +\braket*{\varphi}{v'_3\o \Res_{z_1=1}\Res_{z_2=1}Y_{M^1}(b,z_1-1)Y_{M^1}(a,z_2-1)v_1\o v_2} \iota_{1,z_1-1}(z_1^n)  \iota_{1,z_2-1}(z_2^{\wt a-k})\\
		&\ \ \ + \braket*{\varphi}{v'_3\o \Res_{z_2=1}Y_{M^1}(a,z_2-1)v_1\o \Res_{z_1=0}Y_{M^2}(b,z_1)v_2} z_1^n\iota_{1,z_2-1}(z_2^{\wt a-k})\\
		&\ \ \ +\braket*{\varphi}{v'_3\o \Res_{z_1=1} Y_{M^1}(b,z_1-1)v_1\o \Res_{z_2=0}Y_{M^2}(a,z_2)v_2} \iota_{1,z_1-1}(z_1^n) z_2^{\wt a-k}\\
		&\ \ \ +\braket*{\varphi}{v'_3\o v_1\o \Res_{z_1=0}\Res_{z_2=0} Y_{M^2}(b,z_1)Y_{M^2}(a,z_2)v_2} z_1^n z_2^{\wt a-k}\\
		&=-\braket*{\varphi}{v'_3\o \Res_{z_2=1}\Res_{z_1-z_2=0}Y_{M^1}(Y(b,z_1-z_2)a,z_2-1) v_1\o v_2} \iota_{z_2,z_1-z_2}(z_1^n)\iota_{1,z_2-1}(z_2^{\wt a-k})\\
		&\ \ \ - \braket*{\varphi}{v'_3\o v_1\o \Res_{z_2=0}\Res_{z_1-z_2=0}Y_{M^1}(Y(b,z_1-z_2)a,z_2)v_2} \iota_{z_2,z_1-z_2}(z_1^n)z_2^{\wt a-k}\\
		&\ \ \ +\braket*{\varphi}{v'_3\o \Res_{z_1=1}\Res_{z_2=1}Y_{M^1}(b,z_1-1)Y_{M^1}(a,z_2-1)v_1\o v_2} \iota_{1,z_1-1}(z_1^n)  \iota_{1,z_2-1}(z_2^{\wt a-k})\\
		&\ \ \ -\braket*{\varphi}{v'_3\o v_1\o \Res_{z_2=0}\Res_{z_1=0} Y_{M^2}(a,z_2)Y_{M^2}(b,z_1)v_2} z_1^n z_2^{\wt a-k}\\
		&\ \ \ -\braket*{\varphi}{v'_3\o \Res_{z_2=1}\Res_{z_1=1}Y_{M^1}(a,z_2-1)Y_{M^1}(b,z_2-1)v_1\o v_2}  \iota_{1,z_1-1}(z_1^n)  \iota_{1,z_2-1}(z_2^{\wt a-k})\\
		&\ \ \ +\braket*{\varphi}{v'_3\o v_1\o \Res_{z_1=0}\Res_{z_2=0} Y_{M^2}(b,z_1)Y_{M^2}(a,z_2)v_2} z_1^n z_2^{\wt a-k}\\
		&=0. 
	\end{align*}
	
	Case IV. Proof of \eqref{eq:invgeneral2}. Since $\deg (a'(\wt a+l))=l+1\geq 1$, we have $ a'(\wt a+l)b'(n)v'_3\in \bar{M}(\Om((M^3)'))$. Then it follows from \eqref{eq:defextensionphi2} that 
	\begin{align*}
		&\braket*{\wphi}{\left(a\otimes z^{\wt a+l}\right). (b'(n)v'_3\o v_1\o v_2)}\\
		&=-\braket*{\wphi}{a'(\wt a+l)b'(n)v'_3\o v_1\o v_2}+\braket*{\wphi}{b'(n)v'_3\o \Res_{z_2=1} Y_{M^1}(a,z_2-1)v_1\o v_2} \iota_{1,z_2-1}(z_2^{\wt a+l})\\
		&\ \ \ +\braket*{\wphi}{b'(n)v'_3\o v_1\o \Res_{z_2=0}Y_{M^2}(a,z_2)v_2} z_2^{\wt a+l}\\
		&=0.
	\end{align*}
	Thus, $\braket*{\wphi}{\mathcal{L}_{\P^1\bs\{\infty,1,0\}}(V). (\bar{M}(\Om((M^3)'))\otimes_\C M^1\otimes_\C M^2)}=0$, and \[\wphi\in \mathscr{C}\left(\Sigma(\bar{M}(\Om((M^3)')), M^1, M^2)\right).\] 
	
	It remains to show that the induced map $F$ in \eqref{eq:isomofconformalblocks} is an isomorphism of vector spaces. Choose $b'(n)=\vac'(-1)$ in \eqref{eq:defextenedphi}; we have 
	\begin{align*}
		\braket*{\wphi}{v'_3\o v_1\o v_2}&=\braket*{\wphi}{\vac'(-1)v'_3\o v_1\o v_2}\\
		&=\sum_{j\geq 0} \binom{-1}{j}\braket*{\varphi}{v'_3\o \vac(j)v_1\o v_2}+\braket*{\varphi}{v'_3\o v_1\o \vac(-1)v_2}\\
		&=\braket*{\varphi}{v'_3\o v_1\o v_2}.
	\end{align*}
	Hence $\wphi|_{\Om((M^3)')\o M^1\o M^2}=\varphi$, and so $G\circ F=\Id$. Conversely, let $\psi\in \mathscr{C}\left(\Sigma(\bar{M}(\Om((M^3)')), M^1, M^2)\right)$. Using the invariance of $\psi$ under $b\otimes z^n\in \mathcal{L}_{\P^1\bs\{\infty,1,0\}}(V)$, with $-\wt b+n+1\geq 0$, we have 
	\begin{align*}
		&\psi(b'(n)v'_3\o v_1\o v_2)\\
		&=-\psi\left((b\o z^n).(v'_3\o v_1\o v_2)\right)+\psi(v'_3\o \Res_{z=1}Y_{M^1}(a,z-1)v_1\o v_2)\iota_{1,z-1}(z^n)\\
		&\ \ \ + \psi(v'_3\o v_1\o \Res_{z=0} Y_{M^2}(b,z)v_2)z^n\\
		&=0+ G(\psi)(v'_3\o \Res_{z=1}Y_{M^1}(a,z-1)v_1\o v_2)\iota_{1,z-1}(z^n)+G(\psi)(v'_3\o v_1\o \Res_{z=0} Y_{M^2}(b,z)v_2)z^n\\
		&= (F\circ G)(\psi) (b'(n)v'_3\o v_1\o v_2).
	\end{align*}
	Since $\bar{M}(\Om((M^3)'))$ is spanned by $b'(n)v'_3$ with $-\wt b+n+1\geq 0$, it follows that $F\circ G=\Id$. Thus, $F$ in \eqref{eq:isomofconformalblocks} is an isomorphism of vector spaces. 
\end{proof}

		\section{Tensor product of modules over $A(V)$}\label{sec:6}
In this section, we prove that if the VOA $V$ is strongly rational, then the contracted tensor product $\odot$ introduced in Definition~\ref{df:M1odotM2} induces a fusion tensor structure on the module category of the Zhu algebra $\Mod(A(V))$, by proving a variant of the fusion rules theorem.

\subsection{Fusion rules determined by $M^1\odot M^2$}
The following lemma follows from \cite[Theorem 8.1]{DLM98}; see Section~\ref{sec:2.1.5}.
\begin{lm}\label{lm6.1}
	Let $V$ be a rational VOA, and let $M$ be an irreducible (ordinary) $V$-module. Then the contragredient module $M'$ is isomorphic to the generalized Verma module $\bar{M}(\Om(M'))$.
\end{lm}
\begin{proof}
	$\Om(M)=M(0)$ is an irreducible $A(V)$-module, so $\Om(M)^\ast$ is an irreducible left $A(V)$-module via the anti-involution $\theta$. In other words,
	$$
	A(V)\times \Om(M)^\ast\ra \Om(M)^\ast,\quad ([a],v')\mapsto [a].v',
	$$
	where
	\[
	\braket{[a]. v'}{v}
	=\braket{v'.[\theta(a)]}{v}
	=\braket{v'}{[\theta(a)].v}
	=\braket{v'}{o(\theta(a))v},
	\]
	see \eqref{eq:contraleftA(V)}.
	
	On the other hand, the bottom degree $\Om(M')=M(0)^\ast=\Om(M)^\ast$ of the contragredient module $M'$ is a left $A(V)$-module, with
$
	\braket{[a]\ast v'}{v}
	=\braket{o(a)v'}{v}
	=\braket{v'}{o(\theta(a))v}.
$
	Hence $\Om(M')$ is isomorphic to $\Om(M)^\ast$ as a left $A(V)$-module. Since $
	(\bar{M}(-)\dashv \Om(-)):\Mod(A(V))\rightleftarrows \mathsf{Adm}(V)$
	is an adjoint equivalence when $V$ is rational, we have
$
	M'\cong \bar{M}(\Om(M)^\ast)
$
	as $V$-modules.
\end{proof}

Putting together Theorem~\ref{thm:AVaction}, the identification \eqref{eq:homspaceequality}, and Theorem~\ref{thm:main}, we obtain a hom-space description of intertwining operators in terms of the contracted tensor product $M^1\odot M^2$ from Definition~\ref{df:M1odotM2}.

\begin{thm}\label{thm:fusion}
	Let $M^1,M^2,M^3$ be ordinary $V$-modules. Suppose that the contragredient module $(M^3)'$ is isomorphic to the generalized Verma module $\bar{M}(\Om((M^3)'))$ associated with the left $A(V)$-module $\Om((M^3)')$, and that $\Om((M^3)')^\ast\cong \Om(M^3)$. Assume further that either $M^1$ or $M^2$ is generated by its bottom degree $\Om(M^1)$ or $\Om(M^2)$, respectively. Then
	\begin{equation}\label{eq:fusion}
		I\fusion{M^1}{M^2}{M^3}\cong \Hom_{A(V)}(M^1\odot M^2,\Om(M^3)).
	\end{equation}
	In particular, if $V$ is rational, then \eqref{eq:fusion} holds for any irreducible $V$-modules $M^1,M^2$, and $M^3$.
\end{thm}
\begin{proof}
	Equation \eqref{eq:fusion} follows from the commutative diagram
	\[
	\begin{tikzcd}
		I\fusion{M^1}{M^2}{(\bar{M}(\Om((M^3)')))'}\arrow[r,"\text{Prop.}~\ref{prop:IOconformalblocks}","\cong"']\arrow[d,dashed]
		& \mathscr{C}\left(\Sigma(\bar{M}(\Om((M^3)')), M^1, M^2)\right)\arrow[d,"\text{Thm.}~\ref{thm:main}","\cong"']\\
		\Hom_{L(V)_0}(M^1\odot M^2,\Om(M^3))
		& \mathscr{C}_{\mathrm{res}}\left(\Sigma(\Om((M^3)'), M^1, M^2)\right)\arrow[l,"\text{Prop.}~\ref{prop:inftycfbHom}"',"\cong"]
	\end{tikzcd}
	\]
	Now \eqref{eq:fusion} follows from \eqref{eq:homspaceequality}, since $M^1$ or $M^2$ is generated by its bottom degree. Finally, if $V$ is rational and $M^1,M^2,M^3$ are irreducible, then $(M^3)'\cong \bar{M}(\Om((M^3)'))$ by Lemma~\ref{lm6.1}, and both $M^1$ and $M^2$ are generated by their bottom degrees.
\end{proof}

		\begin{coro}\label{coro:IOtensor}
	Let $V$ be a rational VOA, and let $M^1,M^2$ be irreducible $V$-modules. Then
	\begin{equation}\label{eq:A(V)bimodules}
		A(M^1)\o_{A(V)}\Om(M^2)\cong M^1\odot M^2
	\end{equation}
	as left $A(V)$-modules.
\end{coro}

\begin{proof}
	Let $\mathscr{W}$ be the set of irreducible $V$-modules. Then $\mathscr{W}$ is finite \cite{Z96,DLM98}. 
	Let $M^3\in \mathscr{W}$ with $M^3(0)=\Om(M^3)$. Since $V$ is rational, the Frenkel--Zhu fusion rules theorem holds \cite{FZ92,Li99,Liu23}:
	\begin{equation}\label{eq:fusionrulestheorem}
		I\fusion{M^1}{M^2}{M^3}\cong \Hom_{A(V)}(A(M^1)\o_{A(V)}\Om(M^2), \Om(M^3)). 
	\end{equation}
	It follows from Theorem~\ref{thm:fusion} and \eqref{eq:fusionrulestheorem} that
	$$
	\Hom_{A(V)}(M^1\odot M^2,\Om(M^3))
	\cong I\fusion{M^1}{M^2}{M^3}
	\cong \Hom_{A(V)}(A(M^1)\o_{A(V)} \Om(M^2),\Om(M^3)).
	$$
	Since $V$ is rational, $A(V)$ is a semisimple associative algebra \cite{Z96,DLM98}, and $\{\Om(M^3): M^3\in \mathscr{W} \}$ is a complete list of irreducible $A(V)$-modules. By \eqref{eq:defpi}, there is an epimorphism of left $A(V)$-modules
	$$
	A(M^1)\o_{A(V)}\Om(M^2)\twoheadrightarrow M^1\odot M^2.
	$$
	Since the irreducible $A(V)$-module $\Om(M^3)$ has the same multiplicity in both $A(M^1)\o_{A(V)}\Om(M^2)$ and $M^1\odot M^2$, they are isomorphic as $A(V)$-modules.
\end{proof}

\begin{remark}
	Note that Corollary~\ref{coro:IOtensor} is a consequence of Theorem~\ref{thm:fusion} and the fusion rules theorem~\eqref{eq:fusionrulestheorem}. In other words, to show \eqref{eq:A(V)bimodules}, we need Theorem~\ref{thm:fusion} and \eqref{eq:fusionrulestheorem} to hold in the first place.
	Thus, it does not give a new proof of either of these theorems.
\end{remark}

\subsection{Fusion tensor product on $\Mod(A(V))$ for strongly rational VOAs}
Let $M^1$ and $M^2$ be ordinary $V$-modules. Huang--Lepowsky constructed a $P(z)$-tensor product $M^1\bt_{P(z)} M^2$, where $z$ is a point on $\P^1$ different from $0$ and $\infty$, and $P(z)$ represents the element $(\P^1,(\infty,z,0), (t_\infty,t_z,t_0))$ in the moduli space of three-pointed genus-zero smooth curves with local coordinates, denoted by  $K$ in  \cite[Remark 4.3]{HL95}. 
The space of coinvariants we are considering is associated to the point $(\P^1,(\infty,1,0), (1/z,z-1,z))$ in $K$ \eqref{eq:datum}, which is $P(1)$ in Huang--Lepowsky's notation. 

In the construction of the $P(z)$-tensor product, Huang--Lepowsky introduced the notion of a $P(z)$-intertwining operator:
\begin{align*}
   F:M^1\o_\C M^2&\ra \overline{M^3},\\
  x_0^{-1}\delta\left(\frac{x_1-z}{x_0}\right)Y_{M^3}(a,x_1)F(v_1\o v_2)
  &=  z^{-1}\delta\left(\frac{x_1-x_0}{z}\right) F(Y_{M^1}(a,x_0)v_1\o v_2)\\
  &\ \ \ +  x_0^{-1}\delta\left(\frac{z-x_1}{-x_0}\right) F(v_1\o Y_{M^2}(a,x_1)v_2).
\end{align*}
Denote the space of such intertwining operators by $I_{P(z)}\fusion{M^1}{M^2}{M^3}$. Then the universal property of the $P(z)$-tensor product can also be written as follows \cite[Definition 4.1]{HL95}:
$$
I_{P(z)}\fusion{M^1}{M^2}{M^3}\cong \Hom_{V}(M^1\bt_{P(z)}M^2,M^3).
$$
The space of (formal) intertwining operators $I\fusion{M^1}{M^2}{M^3}$ in Definition~\ref{def:IO} can be identified with the space of $P(z)$-intertwiners by letting $z$ be the complex number $z\in \P^1\bs\{\infty,0\}$ and fixing a branch of the logarithm
$l_p(z)=\log|z|+2p\pi i$ so that $z^h=e^{h l_{p}(z)}$, where $h=h_1+h_2-h_3$.
It follows from Proposition~\ref{prop:IOconformalblocks} that the space of conformal blocks
$\mathscr{C}\left(\Sigma((M^3)', M^1, M^2)\right)$ associated to $(\P^1,\infty,1,0)$ is isomorphic to
$I_{P(1)}\fusion{M^1}{M^2}{M^3}$.

Now, if $(\bar{M}(-)\dashv \Om(-)):\Mod(A(V))\rightleftarrows \mathsf{Adm}(V)$ is an adjoint equivalence, then it is clear that
$\Hom_{V}(M^1\bt_{P(1)}M^2,M^3)\cong \Hom_{A(V)}(\Om(M^1\bt_{P(1)}M^2),\Om(M^3))$ for any irreducible $V$-modules
$M^1,M^2,M^3$, and we have the following commutative diagram:
\begin{equation}\label{tablefusion}
    \begin{tikzcd}
 \mathscr{C}\left(\Sigma((M^3)', M^1, M^2)\right)\arrow[d,"\text{Thm.}~\ref{thm:main}","\cong"']\arrow[r,"\text{Prop.}~\ref{prop:IOconformalblocks}","\cong"']&
 I_{P(z)}\fusion{M^1}{M^2}{M^3}\cong \Hom_{V}(M^1\bt_{P(1)}M^2,M^3)\arrow[d,"\cong"] \\
  \mathscr{C}_{\mathrm{res}}\left(\Sigma(\Om((M^3)')^\ast, M^1, M^2)\right)\arrow[r,"\cong"]\arrow[d,"\text{Prop.}~\ref{prop:inftycfbHom}"',"\cong"]&
  \Hom_{A(V)}(\Om(M^1\bt_{P(1)}M^2),\Om(M^3))\\
  \Hom_{A(V)}(M^1\odot M^2,\Om(M^3)).\arrow[ur,dashed]
\end{tikzcd}
\end{equation}
Therefore, if $A(V)$ is semisimple, then $\Om(M^3)$ has the same multiplicity in both
$M^1\odot M^2$ and $\Om(M^1\bt_{P(1)}M^2)$.
By Lemma~\ref{lmrationality}, we obtain the following proposition.
\begin{prop}
Let $V$ be a rational VOA, and let $M^1,M^2$ be irreducible $V$-modules. Then we have
$$
\Om(M^1\bt_{P(1)}M^2)\cong M^1\odot M^2
$$
as left $A(V)$-modules, and
$$
\bar{M}(M^1\odot M^2)\cong M^1\bt_{P(1)} M^2
$$
as $V$-modules.
\end{prop}

The multiplicities of the simple $V$-modules in the $P(z)$-tensor product are given by the fusion rules. The
following fact was proved by Huang--Lepowsky \cite{HL95,H05}:

\begin{lm}\cite[Proposition 4.13]{HL95}\label{lm6.4}
	Let $V$ be a rational VOA such that the fusion rules among irreducible modules are all finite, and let $M^1,M^2$ be irreducible $V$-modules. Then the $P(z)$-tensor product exists and
\begin{equation}\label{fusiontensor}
    M^1\bt_{P(z)} M^2\cong \bigoplus_{W\in \mathscr{W}}N\fusion{M^1}{M^2}{W}\, W
\end{equation}
	as a $V$-module for any $z\in \P^1\bs\{\infty,0\}$.
\end{lm}

\begin{df}\cite{V88,MS89,Z96,DLM00}
A VOA $V$ is called {\em strongly rational} if $V$ is of CFT-type, rational, $C_2$-cofinite, simple, and self-dual. In this case, $V$ has finitely many irreducible modules, and the fusion rules among irreducible modules are all finite.

The tensor product~\eqref{fusiontensor} of irreducible modules over strongly rational VOAs is also referred to as the {\em fusion tensor product}.
\end{df}

 The associativity of the fusion tensor product is equivalent to the associativity of fusion rules:
\begin{equation}
    \sum_{W\in \mathscr{W}}N\fusion{M^1}{M^2}{W}\cdot N\fusion{W}{M^3}{M^4}
    =\sum_{W\in \mathscr{W}} N\fusion{M^2}{M^3}{W}\cdot N\fusion{M^1}{W}{M^4}.
\end{equation}
For strongly rational VOAs, associativity can be proved using the vector bundle of VOA-conformal blocks $\mathscr{C}(V;M^\bullet)$ on the moduli space $\overline{\mathcal{M}}_{0,4}$ \cite{KL93,DGT21,GL25}. See also \cite{HL95,H05} for a proof using compositions of intertwining operators.

In particular, under the assumptions of Lemma~\ref{lm6.4}, the bottom degree of $M^1\bt_{P(z)}M^2$ satisfies
\begin{equation}\label{Omtensor}
\Om(M^1\bt_{P(z)}M^2)=\bigoplus_{W\in \mathscr{W}}N\fusion{M^1}{M^2}{W}\, \Om(W)    
\end{equation}
as a left $A(V)$-module. Moreover, since $(\bar{M}(-)\dashv \Om(-)):\Mod(A(V))\rightleftarrows \mathsf{Adm}(V)$ is an adjoint equivalence of categories, the fusion tensor structure on $\mathsf{Adm}(V)$ can be pulled back to a fusion tensor structure on $\Mod(A(V))$.

\begin{thm}\label{coro:bottomdegree}
 Let $V$ be a strongly rational VOA. The fusion tensor product of two irreducible $A(V)$-modules $S,T$ is given by
	\begin{equation}\label{corotensor}
		S\bt T\cong\bar{M}(S)\odot \bar{M}(T).
	\end{equation}
\end{thm}
\begin{proof}
	Since $A(V)$ is semisimple and $\{\Om(W):W\in \mathscr{W}\}$ is the complete list of irreducible left $A(V)$-modules, it follows from \eqref{eq:fusion} and Schur's lemma that
	$$
	M^1\odot M^2
	=\bigoplus_{W\in \mathscr{W}}
	\dim_\C \Hom_{A(V)}(M^1\odot M^2,\Om(W))\cdot \Om(W)
	=\bigoplus_{W\in \mathscr{W}} N\fusion{M^1}{M^2}{W}\, \Om(W).
	$$
	Now $\Om(M^1\bt_{P(z)}M^2)\cong M^1\odot M^2$ follows from Lemma~\ref{lm6.4}. It is clear from the construction that the functor $\bar{M}(-)$ preserves direct sums of $V$-modules. Hence,
	$$
	\bar{M}(M^1\odot M^2)
	\cong \bigoplus_{W\in \mathscr{W}} N\fusion{M^1}{M^2}{W}\bar{M}(\Om(W))
	\cong M^1\bt_{P(z)} M^2.
	$$
	Under the adjoint equivalence $(\bar{M}(-)\dashv \Om(-)):\Mod(A(V))\rightleftarrows \mathsf{Adm}(V)$, the pull-back of the $P(z)$-tensor product from $\mathsf{Adm}(V)$ to $\Mod(A(V))$ is given by
	$$
	S \bt T=\Om(\bar{M}(S)\bt_{P(z)}\bar{M}(T))\cong \bar{M}(S)\odot \bar{M}(T),
	$$
    in view of the fact that $\Om(M^1\bt_{P(z)}M^2)\cong M^1\odot M^2$. 
\end{proof}

By Theorem~\ref{thm:AVaction}, Theorem~\ref{thm:fusion}, and Theorem~\ref{coro:bottomdegree}, we obtain the following universal property for the fusion tensor product on $\Mod(A(V))$.

\begin{prop}
	Let $V$ be a strongly rational VOA.
	Let $S^1,S^2$, and $S^3$ be irreducible $A(V)$-modules, whose corresponding irreducible admissible $V$-modules under the functor $\bar{M}(-): \Mod(A(V))\ra \mathsf{Adm}(V)$ are denoted by $M^1,M^2$, and $M^3$, respectively. Then
	\begin{equation}
		\Hom_{A(V)}(S^1\bt S^2,S^3)
		\cong I\fusion{M^1}{M^2}{M^3}
		\cong \Hom_{V}(M^1\bt_{P(z)}M^2, M^3),
	\end{equation}
	where $S^1\bt S^2$ is given by \eqref{corotensor}. 
\end{prop}
 The following corollary follows immediately from the corresponding properties of fusion rules in \cite[Section 5]{FHL93}.

\begin{coro}\label{coro6.9}
	Let $V$ be a strongly rational VOA. Denote the contragredient left $A(V)$-module $S^\ast$ by $S^\vee$. Then we have the following isomorphisms of left $A(V)$-modules:
\begin{enumerate}
	\item $S\bt T\cong T\bt S$.
	\item Denote $\Om(V)$ by $\vac$. Then $\vac\bt S\cong S\cong S\bt \vac$, that is, $\vac$ is the tensor identity.
	\item $\vac$ is an $A(V)$-submodule of $S\bt S^\vee$.
\end{enumerate}
\end{coro}

\begin{proof}
	By Theorem~\ref{coro:bottomdegree} and \eqref{Omtensor}, we have
	$$
	S\bt T\cong \bigoplus_{W\in \mathscr{W}} N\fusion{\bar{M}(S)}{\bar{M}(T)}{W}\cdot \Om(W).
	$$
	Now, (1) follows from the fact that
	$N\fusion{M^1}{M^2}{M^3}=N\fusion{M^2}{M^1}{M^3}$;
	(2) follows from the fact that
	$N\fusion{V}{M}{W}=\delta_{M,W}=N\fusion{M}{V}{W}$;
	and (3) follows from the fact that $V$ is self-dual and
	$N\fusion{M}{M'}{V}=N\fusion{M}{V'}{M''}=N\fusion{V}{M}{M}=1$,
	where $M^1,M^2,M^3,M,W\in \mathscr{W}$.
\end{proof}

\begin{remark}
The explicit form of the isomorphism
\[
B:S\bt T=\bar{M}(S)\odot \bar{M}(T)\xrightarrow{\simeq}\bar{M}(T)\odot \bar{M}(S)=T\bt S
\]
is the braiding isomorphism in the category $\Mod(A(V))$. This can be derived from the isomorphism
$I\fusion{M^1}{M^2}{M^3}\xrightarrow{\simeq} I\fusion{M^2}{M^1}{M^3}$
using the skew-symmetry formula \cite[eq.~(5.4.33)]{FHL93}, together with Proposition~\ref{prop:IOconformalblocks}. We will provide full details in a subsequent work.
\end{remark}


\section{Examples of the tensor products of $A(V)$-modules}\label{sec:7}
We determine the contracted tensor product $M^1\odot M^2$ for irreducible modules in several examples, including the universal affine VOA $V^k(\g)$, the level-one type-$A_1$ affine VOA $L_1(\mathfrak{sl}_2)$, the universal Virasoro VOA $\bar{V}(c,0)$, and the Ising model Virasoro VOA $L(\tfrac{1}{2},0)$.
		
		\subsection{ $M^1\odot M^2$ for the universal affine and universal Virasoro VOAs}\label{Sec:6}
		We first examine the contracted tensor product $M^1\odot M^2$ for some irrational VOAs. It turns out that the left $A(V)$-module $A(M^1)\o_{A(V)}\Om(M^2)$ may or may not be isomorphic to $ M^1\odot M^2$, unlike the rational case in Corollary~\ref{coro:IOtensor}.

		
	\subsubsection{Examples in the Heisenberg and universal affine VOAs} 
		
Let $V=V^k(\g)$ be the Heisenberg or universal affine VOA of level $k\in \C$, where $\g$ is an abelian or simple Lie algebra, respectively. Recall that $A(V)\cong U(\g)$ \cite{FZ92}. Let $U$ be an irreducible $A(V)$-module. The {\em generalized Verma module associated to $U$} \cite{DLM98,LL04} is given by the induced module
\begin{equation}
	V^k(\g, U)=U(\hat{\g})\o _{U(\hat{\g}_{\geq 0})}U\cong U(\hat{\g}_{-})\o _{\C} U,
\end{equation}
where $K.u=ku$ for some fixed $k\in \C$, $a(0).u=a.u$, and $a(n).u=0$ for any $a\in \g$, $n>0$, and $u\in U$. It was proved in \cite{FZ92,LL04} that $V^k(\g,U)$ is a $V$-module, and it is the generalized Verma module $\bar{M}(U)$ of the VOA $V^k(\g)$; see Section~\ref{sec:2.1.5}.

Let $U=L(\la)$, which is the finite-dimensional irreducible $\g$-module associated to a dominant weight $\la\in P_+$ if $\g$ is simple, and is $\C e^\la$ if $\g$ is abelian. Denote the generalized Verma module $V^k(\g,L(\la))$ by $V^k(\g,\la)$. Note that $V^k(\g,\la)$ is an irreducible $V^k(\g)$-module if $\g$ is abelian. If $\g$ is simple, it admits a unique maximal submodule $W^k(\g,\la)$ and an irreducible quotient module $L_k(\g,\la)=V^k(\g,\la)/W^k(\g,\la)$. The bottom degree $\Om(L_k(\g,\la))$ is the $\g$-module $L(\la)$. 
		
\begin{prop}
	Let $V=V^k(\g)$ be the Heisenberg or universal affine VOA of level $k\in \C$. Assume $M^1=L_k(\g,\la)=V^k(\g,\la)$ and $M^2=L_k(\g,\mu)=V^k(\g,\mu)$ are irreducible generalized Verma modules such that $\Om(M^1)=L(\la)$ and $\Om(M^2)=L(\mu)$ are finite-dimensional. Then we have
	\begin{equation}\label{6.7}
		M^1\odot M^2\cong  A(M^1)\o_{A(V)} \Om(M^2)\cong L(\la)\o_\C L(\mu)
	\end{equation}
	as left $A(V)\cong U(\g)$-modules.
\end{prop}

\begin{proof}
	By \cite[Theorem 3.2.1]{FZ92} and our assumptions on $M^1$ and $M^2$,
	$$
	A(M^1)\o_{A(V)} \Om(M^2)
	\cong (\Om(M^1)\o_\C A(V))\o _{A(V)} \Om(M^2)
	\cong \Om(M^1)\o _\C \Om(M^2).
	$$
	First, we consider the case when $\g=\h$ is an abelian Lie algebra and $V^k(\h)=M(k)$ is the Heisenberg VOA. In this case, any generalized Verma module $V^k(\h,\la)$ is irreducible.

	Since $L(\la)=\C e^{\la}$, we have $A(M^1)\o_{A(V)} \Om(M^2)\cong \C e^\la\o \C e^\mu$, and it surjects onto $M^1\odot M^2$ by \eqref{eq:defpi}. Moreover, $M^1\odot M^2\neq 0$, since $\Hom_{L(V)_0}(M^1\odot M^2, \C e^{\la+\mu})\neq 0$ in view of \eqref{eq:fusion}. Hence
	$$
	A(M^1)\o_{A(V)} \Om(M^2)\cong \C e^\la\o \C e^\mu\cong M^1\odot M^2.
	$$
			
	Next, we consider the case when $\g$ is a semisimple Lie algebra and $V=V^k(\g)$ is the universal affine VOA. Since $A(M^1)\o_{A(V)} \Om(M^2)$ is finite-dimensional and surjects onto $M^1\odot M^2$, by Weyl's complete reducibility theorem, both $A(M^1)\o_{A(V)} \Om(M^2)$ and $M^1\odot M^2$ are direct sums of finite-dimensional irreducible $\g$-modules.
			
	Let $U$ be a finite-dimensional irreducible $A(V^k(\g))\cong U(\g)$-module. Then, by \eqref{eq:homspaceequality} and Theorem~\ref{thm:fusion}, we have
	$$
	\Hom_{U(\g)}(M^1\odot M^2,U)
	\cong \Hom_{L(V)_0}(M^1\odot M^2,U)
	\cong I\fusion{M^1}{M^2}{(\bar{M}(U^\ast))'}.
	$$
	On the other hand, $M^2=V^k(\g,\mu)$ is also a generalized Verma module as a module over the VOA $V_{\hat{\g}}(k,0)$. Then, by \cite[Theorem 4.20]{Liu23} and \cite[Proposition 6.3]{GLZ23}, we have
	$$
	\Hom_{U(\g)}(A(M^1)\o_{A(V)}\Om(M^2), U)
	\cong I\fusion{M^1}{M^2}{(\bar{M}(U^\ast))'}.
	$$
	Thus, the multiplicities of the irreducible $\g$-module $U$ in $M^1\odot M^2$ and $A(M^1)\o_{A(V)}\Om(M^2)$ are the same. Hence,
	$
	A(M^1)\o_{A(V)} \Om(M^2)\cong M^1\odot M^2.
	$
\end{proof}

\begin{remark}
	If $M^1=L_k(\g,\la)$ and $M^2=L_k(\g,\mu)$ are irreducible but not generalized Verma modules, then $\Om(M^1)$ and $\Om(M^2)$ are not necessarily the irreducible $\g$-modules $L(\la)$ and $L(\mu)$ because of the existence of singular vectors in Verma modules. Hence, we do not have the identification \eqref{6.7}. Examples will be given in the next subsection.
\end{remark}




		\subsubsection{Examples in the Virasoro VOAs and Li's example}\label{Sec:5.4.2}
Let $V=M_c$ be the (universal) Virasoro VOA \cite{FZ92} of central charge $c$. Li gave an example in \cite[Section 2]{Li99} showing that the Frenkel--Zhu fusion rules theorem~\ref{eq:fusionrulestheorem} does not hold if $M^2$ and $M^3$ are not generalized Verma modules.
		
Let $M(c, h)$ be the Verma module over the Virasoro algebra of highest weight $h$ and central charge $c$. Recall that $M_c=M(c,0)/\<L(-1)v_{c,0}\>$, where $v_{c,0}$ is the highest-weight vector. Let $c$ be chosen in such a way that the VOA $M_c$ is simple. Then $\Om(M_c)=M_c(0)$. Any Virasoro Verma module $M(c,h)$ is a module over the VOA $M_c$. Li noticed that if $h_1\neq h_2$, then
$$
N\fusion{M(c,h_1)}{M_c}{M(c,h_2)}=0,
\quad \mathrm{but} \quad
\dim \Hom_{A(M_c)}(A(M(c,h_1))\o_{A(M_c)} \Om(M_c), \Om(M(c,h_2))=1.
$$
		
This is due to the fact that $M^2=M_c$, as a module over itself, is not a generalized Verma module. However, the formula in Theorem~\ref{thm:fusion} holds for this example since $M^3=M(c,h_2)$ is a generalized Verma module. We can also see this from the following proposition. 
		
\begin{prop}\label{prop;vir}
	The contracted tensor product $M(c,h_1)\odot M_c$ is spanned by the vector $v_{c,h_1}\odot \overline{v_{c,0}}$, with the $A(M_c)$ (or $L(M_c)_0$)-module action
	\begin{equation}\label{eq7.3}
		[\om].(v_{c,h_1}\odot  \overline{v_{c,0}})= \om_{[\wt \om -1]}.(v_{c,h_1}\odot  \overline{v_{c,0}})
		= h_1 \cdot (v_{c,h_1}\odot  \overline{v_{c,0}}).
	\end{equation}
	In particular, we have
	$\Hom_{A(M_c)}(M(c,h_1)\odot M_c, \Om(M(c,h_2))=0
	= I\fusion{M(c,h_1)}{M_c}{M(c,h_2)}$.
\end{prop}

\begin{proof}
	By Proposition~\ref{lm:spannodot},
	$$
	M(c,h_1)\odot M_c
	=\spn\{L(-n_1)\ds L(-n_r)v_{c,h_1}\odot  \overline{v_{c,0}}
	: n_1\geq \ds\geq n_r\geq 1 \}.
	$$
	Given $v_1\in M(c,h_1)$, by \eqref{eq:OM1odot} and \eqref{eq: actionspanning1} we have
	\begin{align*}
		&\left(L(-n-3)+2L(-n-2)+L(-n-1)\right)v_1\odot  \overline{v_{c,0}}=0,
		\quad n\geq 0,\\
		&L(-2)v_{1}\odot \overline{v_{c,0}}
		+L(-1)v_1\odot \overline{v_{c,0}}
		= v_1\odot L(0)\overline{v_{c,0}}=0.
	\end{align*}
	Hence $M(c,h_1)\odot M_c$ is spanned by elements of the form
	$L(-1)^m v_{c,h_1}\odot \overline{v_{c,0}}$, with $m\geq 0$. Moreover, applying \eqref{eq: actionspanning2} to $a=\om$ and $k=2$, we have
	$$
	L(-1)v_1\odot \overline{v_{c,0}}
	=-v_1\odot L(-1)\overline{v_{c,0}}
	=-v_1\odot \overline{L(-1)v_{c,0}}=0.
	$$
	This shows that
	$M(c,h_1)\odot M_c=\spn\{v_{c,h_1}\odot \overline{v_{c,0}}\}$.
	Finally, by \eqref{eq: L(V)0action} we have
	$$
	\om_{[\wt \om-1]}.(v_{c,h_1}\odot  \overline{v_{c,0}})
	= L(-1)v_{c,h_1}\odot  \overline{v_{c,0}}
	+L(0)v_{c,h_1}\odot  \overline{v_{c,0}}
	+v_{c,h_1}\odot  L(0)\overline{v_{c,0}}
	= h_1 \cdot (v_{c,h_1}\odot  \overline{v_{c,0}}).
	$$
	Since $\om_{[\wt \om-1]}.v_{c,h_2}=h_2 \cdot v_{c,h_2}$, we have
	$\Hom_{A(M_c)}(M(c,h_1)\odot M_c, \Om(M(c,h_2))=0$.
\end{proof}
		
\begin{remark}
	It was proved by Li that
	$A(M(c,h_1))\o_{A(M_c)}M_c(0)\cong \C[t_1]$; see \cite[Section 2]{Li99}. Proposition~\ref{prop;vir} shows that $A(M^1)\o_{A(V)}\Om(M^2)$ is {\em not} isomorphic to $M^1\odot M^2$ in general if $V$ is not a rational VOA. Moreover, in view of Corollary~\ref{corc:comparison}, we also have the sharp estimate
	$$
	\dim \Hom_{L(M_c)_0}(M(c,h_1)\odot M_c, \C v_{c,h_2})
	<
	\dim \Hom_{A(M_c)}(A(M(c,h_1))\o_{A(M_c)} M_c(0), \C v_{c,h_2}),
	$$
	Thus, the formula~\eqref{eq:fusion} holds under more general assumptions than the Frenkel--Zhu fusion rules theorem~\eqref{eq:fusionrulestheorem}, in the sense that we only require $(M^3)'$ to be a generalized Verma module. 
\end{remark}

\subsection{Fusion tensor product for the level-one type-$A_1$ affine VOA $L_1(\mathfrak{sl}_2)$}	

\subsubsection{Basics of the affine VOA $L_1(\sl_2)$}
	Consider the level-one type $A_1$ affine VOA $V=L_1(\mathfrak{sl}_2)$ \cite{FZ92}. This VOA is strongly rational \cite{DLM97}, with two irreducible modules
\begin{align*}
    V=L_1(\mathfrak{sl}_2)\quad \mathrm{and}\quad 
    W=L_1(\sl_2,\al/2). 
\end{align*}
According to the fusion rules in the WZNW model \cite{V88,Bea96}, the fusion tensor products of irreducible modules in the category $\adm(L_1(\sl_2))$ are given as follows:
\begin{equation}\label{affinefusion}
	V\bt V=V,\quad V\bt W=W,\quad W\bt W=V.
\end{equation}
		
The Zhu algebra of $L_1(\sl_2)$ can be identified as
\[
A(L_1(\mathfrak{sl}_2))\cong U(\sl_2)/\<e^2\>.
\]
This is a $5$-dimensional semisimple algebra, where $\<e^2\>$ is the two-sided ideal of $U(\sl_2)$ generated by $e^2$ \cite{FZ92}. 
One can easily show that $A(L_1(\mathfrak{sl}_2))$ is spanned by $\{1,e,f,h,h^2\}$, subject to the relations
\begin{equation}\label{3.26}
	eh+e=0;\quad h^2-h-2fe=0;\quad fh-f=0;\quad e^2=f^2=0,
\end{equation}
see \cite[Section 6]{Liu25}. The bottom degrees of $V$ and $W$ give all irreducible modules over $A(L_1(\sl_2))$ up to isomorphism:
\begin{align*}
	&S=\Om(V)=\C \vac,\quad \mathrm{with}\quad e.\vac=f.\vac=h.\vac=0,\\
	&T=\Om(W)=\C e^{+}\op \C e^{-},\quad \mathrm{with}\quad e.e^{+}=0,\ f.e^{+}=e^{-},\ h.e^{\pm }=\pm e^{\pm }.
\end{align*}
The relations in $S$ and $T$ match the relations \eqref{3.26} in $A(L_1(\sl_2))$. Moreover, $V$ and $W$ have the following description
\begin{align*}
    V&=L_1(\mathfrak{sl}_2)=\left(U(\widehat{\sl_2})\o_{U(\widehat{\sl_2})_{\geq 0}}S\right)/\<e(-1)^{2}\o \vac\>,\\
    W&=L_1(\sl_2,\al/2)=\left(U(\widehat{\sl_2})\o_{U(\widehat{\sl_2})_{\geq 0}}T\right)/\<e(-1)\o e^+\>.
\end{align*}

Since $(\bar{M}(-)\dashv \Om(-)):\Mod(A(L_1(\sl_2)))\rightleftarrows \mathsf{Adm}(L_1(\sl_2))$ is an equivalence of categories,
\begin{align*}
	\bar{M}(S)&\cong V=\spn\{a^1(-n_1)\ds a^r(-n_r)\vac: a^i\in \sl_2,\ n_1\geq \ds\geq n_r\geq 0,\ e(-1)^2\vac=0 \},\\
	\bar{M}(T)&\cong W=\spn\{b^1(-m_1)\ds b^k(-m_k)e^\pm: b^j\in \sl_2,\ m_1\geq \ds\geq m_k\geq 0,\ e(-1)e^+=0\}.
\end{align*}

\subsubsection{Fusion tensor product on $\Mod(A(L_1(\sl_2)))$}
We want to verify that our $\odot$-tensor product, together with \eqref{corotensor}, gives rise to the correct fusion tensor product \eqref{affinefusion} in the category $\Mod(A(L_1(\sl_2)))$. In other words,
\begin{equation}\label{A1fusiontensor}
	S\boxtimes S=S,\quad S\bt T=T,\quad T\bt T=S.
\end{equation}
We only prove $T\bt T=W\odot W\cong S$. The other cases are similar, and we omit the details.

		\begin{lm}
	As a vector space, $W\odot W$ is spanned by $\al \odot \b$, where $\al,\b\in T=\C e^+\op \C e^-$. 
\end{lm}

\begin{proof}
This lemma is an easy consequence of Lemma~\ref{lm:spannodot}. For the sake of completeness, we write out the full details of the proof.

	Let $a\in \sl_2=V_1$. Using \eqref{eq:actionfogeneralspann}, with $s=1$ and $t=n\geq 1$, we have
	\begin{equation}\label{addrel1}
		a(-n)u\odot v
		=-u\odot \sum_{j\geq 0} \binom{-n}{j}(-1)^{n+j} a(j)v,
		\quad u,v\in W. 
	\end{equation}
	In particular, we have
	\begin{align*}
		&a^1(-n_1)\ds a^r(-n_r) \al\odot b^1(-m_1)\ds b^k(-m_k)\b\\
		&=-a^2(-n_2)\ds a^r(-n_r) \al\odot 
		\sum_{j\geq 0} \binom{-n_1}{j}(-1)^{n_1+j} a^1(j) b^1(-m_1)\ds  b^k(-m_k)\b\\
		&\vdots\\
		&=\al\odot \sum \la\cdot  a^r(j_r)\ds a^1(j_1)b^1(-m_1)\ds b^k(-m_k)\b.
	\end{align*}
	On the other hand, using \eqref{eq: actionspanning2} with $k=1+n$, we have
	\begin{equation}\label{addrel2}
		\sum_{j\geq 0}\binom{-n}{j} a(j)u\odot v=-u\odot a(-n)v,
		\quad u,v\in W.
	\end{equation}
	For any $\al\in T$, we have $a(0)\al=a.\al\in T$ since $T$ is an $\sl_2$-module, and $a(j).\al=0$ if $j>0$. Then by \eqref{addrel2}, for any $\al,\b\in T$, $n_1\geq \ds \geq n_r\geq 1$, and $c^1,\ds, c^r\in \sl_2$, we have
	\begin{align*}
		&\al\odot c^1(-n_1)\ds c^r(-n_r)\b\\
		&=-\sum_{j\geq 0} \binom{-n_1}{j} c^1(j)\al\odot  c^2(-n_2)\ds c^r(-n_r)\b\\
		&=-c^1.\al\odot  c^2(-n_2)\ds c^r(-n_r)\b\\
		&\vdots\\
		&=(-1)^r (c^r\ds c^1).\al\odot \b\in T\odot T,
	\end{align*}
	with $(c^r\ds c^1).\al\in T$. 
	Hence $W\odot W$ is spanned by
	$\{\al\odot \b: \al,\b\in T=\C e^+\op \C e^-\}$.
\end{proof}

\begin{prop}\label{prop:affinelattice}
	The following relations hold in $W\odot W$:
	$$
	e^+\odot e^+=e^-\odot e^-=0,\quad
	e^-\odot e^++e^+\odot e^-=0.
	$$
	In particular,
	$$
	T\bt T=W\odot W=\C (e^-\odot e^+)\cong S
	$$
	as a left $A(L_1(\sl_2))$-module.
\end{prop}
\begin{proof}
	Note that $e(-1).e^+=0$ in $W$. By \eqref{addrel1} with $a=e$ and $n=1$, we have
	$$
	0=e(-1).e^+\odot e^-
	=-e^+\odot \sum_{j\geq 0}\binom{-1}{j} (-1)^{1+j} e(j).e^-
	=e^+\odot e^+.
	$$
	Now apply $f\in A(L_1(\sl_2))$ to $e^+\odot e^+$. Note that $f.e^+=e^-$ in $T$. Then by \eqref{eq:AVactiononodot}, we have
	$$
	0=f.(e^+\odot e^+)
	=f.e^+\odot e^+ + e^+\odot f.e^+
	=e^-\odot e^+ + e^+\odot e^-.
	$$
	Applying $f$ once more to the equation
	$e^-\odot e^+ + e^+\odot e^-=0$, we obtain $e^-\odot e^-=0$.
			
	Therefore,
	$$
	T\bt T=\spn\{e^{\pm}\odot e^{\pm}\}
	=\spn\{e^+\odot e^-=-e^-\odot e^+\}.
	$$
	It is easy to check that
	$e.(e^+\odot e^-)=f.(e^+\odot e^-)=h.(e^+\odot e^-)=0$.
	Furthermore, since $W\cong W'$ as an $L_1(\sl_2)$-module, we have
	$T\cong T^\vee$ as a left $A(L_1(\sl_2))$-module. Then by Corollary~\ref{coro6.9}, $\vac\subset T\bt T^\vee$.
	Hence $T\bt T\neq 0$, and it is isomorphic to $S$ as an $A(L_1(\sl_2))$-module.
\end{proof}

\subsection{Fusion tensor product for the critical Ising model Virasoro VOA $L(\frac{1}{2},0)$}	

\subsubsection{Basics of the Virasoro VOA $L(\frac{1}{2},0)$}
According to \cite{W93,DMZ94,DLM00}, the VOA $V=L\left(\tfrac{1}{2},0\right)$ is a strongly rational VOA, with three irreducible modules
\begin{equation}\label{eq:irrvir}
	L\left(\tfrac{1}{2},0\right),\quad
	L\left(\tfrac{1}{2},\tfrac{1}{2}\right),\quad
	L\left(\tfrac{1}{2},\tfrac{1}{16}\right).
\end{equation}
The corresponding bottom degrees are all one-dimensional.

The fusion tensor products among these irreducible modules are standard in the critical Ising model conformal field theory \cite{V88,MS89}:
\begin{equation}\label{eq:vir}
	\begin{aligned}
		&L\left(\tfrac{1}{2},0\right)\boxtimes W=W,\quad \forall\, W\in \mathscr{W},\\
		&L\left(\tfrac{1}{2},\tfrac{1}{2}\right)\boxtimes L\left(\tfrac{1}{2},\tfrac{1}{2}\right)=L\left(\tfrac{1}{2},0\right),\\
		&L\left(\tfrac{1}{2},\tfrac{1}{2}\right)\boxtimes L\left(\tfrac{1}{2},\tfrac{1}{16}\right)=L\left(\tfrac{1}{2},\tfrac{1}{16}\right),\\
		&L\left(\tfrac{1}{2},\tfrac{1}{16}\right)\boxtimes L\left(\tfrac{1}{2},\tfrac{1}{16}\right)
		= L\left(\tfrac{1}{2},0\right)+L\left(\tfrac{1}{2},\tfrac{1}{2}\right).
	\end{aligned}
\end{equation}

According to \cite{FF88}, the irreducible Virasoro algebra modules admit the following quotient descriptions. Denote the highest-weight vector $v_{c,0}$ by $\vac$:
\begin{align}
	L\left(\tfrac{1}{2},0\right)
	&=M\left(\tfrac{1}{2},0\right)/\<L(-1)\vac,\ v_s\>,\\
	L\left(\tfrac{1}{2},\tfrac{1}{2}\right)
	&=M\left(\tfrac{1}{2},\tfrac{1}{2}\right)/\<\bigl(L(-2)-\tfrac{3}{4}L(-1)^2\bigr)u\>,
	\quad u=v_{1/2,1/2}, \label{rel1/2} \\
	L\left(\tfrac{1}{2},\tfrac{1}{16}\right)
	&=M\left(\tfrac{1}{2},\tfrac{1}{16}\right)/\<\bigl(L(-2)-\tfrac{4}{3}L(-1)^2\bigr)v\>,
	\quad v=v_{1/2,1/16}, \label{rel1/16}
\end{align}
where
$v_s=(64L(-2)^3+93L(-3)^2-264L(-4)L(-2)-108L(-6))\vac$.

Moreover, the Zhu algebra $A(L\left(\tfrac{1}{2},0\right))$ admits the following description \cite{W93,DMZ94}:
\begin{equation}\label{viaA}
	A(L\left(\tfrac{1}{2},0\right))
	\cong \C[x]/\<x(x-\tfrac{1}{2})(x-\tfrac{1}{16})\>,
	\quad [\om=L(-2)\vac]\mapsto \bar{x}.
\end{equation}
The bottom-degree irreducible $A(L\left(\tfrac{1}{2},0\right))$-modules corresponding to these irreducible
$L\left(\tfrac{1}{2},0\right)$-modules are
\begin{align*}
	T(0)&=\Om(L\left(\tfrac{1}{2},0\right))=\C\vac,
	&&\bar{x}.\vac=0,\\
	T(\tfrac{1}{2})&=\Om(L\left(\tfrac{1}{2},\tfrac{1}{2}\right))=\C u,
	&&\bar{x}.u=\tfrac{1}{2}\cdot u,\\
	T(\tfrac{1}{16})&=\Om(L\left(\tfrac{1}{2},\tfrac{1}{16}\right))=\C v,
	&&\bar{x}.v=\tfrac{1}{16}\cdot v.
\end{align*}

\subsubsection{Fusion tensor product on $\Mod(A(L\left(\frac{1}{2},0\right))$}
Similar to the previous subsection, we verify that our $\odot$-tensor product, together with \eqref{corotensor}, gives rise to the correct fusion tensor product \eqref{affinefusion} in the category $\Mod(A(L\left(\frac{1}{2},0\right))$. In other words, we show that
\begin{align*}
	&T(\tfrac{1}{2})\bt T(\tfrac{1}{2})=T(0),\\
	&T(\tfrac{1}{2})\bt T(\tfrac{1}{16})=T(\tfrac{1}{16}),\numberthis\label{AIsingtensor}\\
	&T(\tfrac{1}{16})\bt T(\tfrac{1}{16})=T(0)+T(\tfrac{1}{2}).
\end{align*}

\begin{lm}
	Let $a,b\in\{\tfrac{1}{2},\tfrac{1}{16}\}$. Then
	\begin{equation}\label{virspan}
		T(a)\bt T(b)
		=L\left(\tfrac{1}{2},a\right)\odot L\left(\tfrac{1}{2},b\right)
		=\spn\{u\odot v,\ L(-1)u\odot v\},
	\end{equation}
	where $u$ and $v$ are the highest-weight vectors of
	$L\left(\tfrac{1}{2},a\right)$ and $L\left(\tfrac{1}{2},b\right)$, respectively.

	The action of $A(L\left(\frac{1}{2},0\right))=\C[\bar{x}]$ on $T(a)\bt T(b)$ is given by
	\begin{equation}\label{viraction}
		\bar{x}.(u_1\odot v_1)
		=L(-1)u_1\odot v_1+L(0)u_1\odot v_1+u_1\odot L(0)v_1,
	\end{equation}
	where $u_1\in L\left(\tfrac{1}{2},a\right)$ and
	$v_1\in L\left(\tfrac{1}{2},b\right)$; see also \eqref{eq7.3}.
\end{lm}

\begin{proof}
	By \eqref{eq:M1odotM2spann1}, we have
	\begin{equation}\label{spanvir}
		T(a)\bt T(b)=\spn\{u_1\odot v: u_1\in L\left(\tfrac{1}{2},b\right)\}.
	\end{equation}
	We first list several relations arising from \eqref{eq: actionspanning1} and \eqref{eq: actionspanning2}:
	\begin{align}
		&\left(L(-n-3)+2L(-n-2)+L(-n-1)\right)u_1\odot v=0,
		\quad n\geq 0,\ u_1\in L\left(\tfrac{1}{2},a\right), \label{eq:virrel1}\\
		&L(-2)u_1\odot v+L(-1)u_1\odot v
		=u_1\odot L(0)v=b\cdot u_1\odot v,
		\quad u_1\in L\left(\tfrac{1}{2},a\right), \label{eq:virrel2}\\
		&L(-1)u_1\odot v_1+u_1\odot L(-1)v_1=0,
		\quad u_1\in L\left(\tfrac{1}{2},a\right),\ v_1\in L\left(\tfrac{1}{2},b\right). \label{eq:virrel3}
	\end{align}
	The first relation arises from the action of
	$\om\o \frac{z^{\wt \om}}{(z-1)^2}$,
	the second from the action of
	$\om\o \frac{z^{\wt \om-1}}{z-1}$,
	and the third from the action of
	$\om\o (z-1)^{\wt \om -2}$;
	see \eqref{eq:spannideal}.

	It follows from \eqref{spanvir} and \eqref{eq:virrel1} that
	$$
	T(a)\bt T(b)=\spn\{L(-1)^k u\odot v: k\geq 0\}.
	$$
	Moreover, by \eqref{rel1/2} and \eqref{rel1/16}, we have
	$L(-2)u=\la\cdot L(-1)^2u$, where $\la=3/4$ or $4/3$.
	Then \eqref{eq:virrel2} implies
	$$
	T(a)\bt T(b)=\spn\{L(-1)u\odot v,\ u\odot v\}.
	$$
	Finally, for the generator $\bar{x}\in A(L\left(\frac{1}{2},0\right))$,
	by \eqref{viaA} and \eqref{eq:AVactiononodot}, we compute
	\begin{align*}
		\bar{x}.(u_1\odot v_1)
		&=[\om].(u_1\odot v_1)
		=\sum_{j\geq 0}\binom{\wt \om-1}{j} \om_{[j]}u_1\odot v_1
		+u_1\odot \om_{[\wt \om-1]}v_1\\
		&=\om_{[0]}u_1\odot v_1+\om_{[1]}u_1\odot v_1+u_1\odot \om_{[1]}v_1\\
		&=L(-1)u_1\odot v_1+L(0)u_1\odot v_1+u_1\odot L(0)v_1.
	\end{align*}
	This proves \eqref{viraction}.
\end{proof}

	\begin{prop}\label{prop:Ising}
	Write $u=v_{\tfrac{1}{2},\tfrac{1}{2}}$ and $v=v_{\tfrac{1}{2},\tfrac{1}{16}}$ for the highest-weight vectors of
	$L\left(\tfrac{1}{2},\tfrac{1}{2}\right)$ and $L\left(\tfrac{1}{2},\tfrac{1}{16}\right)$, respectively. Then
	\begin{align}
		&L(-1)u\odot u+ u\odot u=0
		\quad \mathrm{in}\quad T(\tfrac{1}{2})\bt T(\tfrac{1}{2}), \label{dimonerel}\\
		&L(-1)u\odot v+ \tfrac{1}{2}u\odot v=0
		\quad \mathrm{in}\quad T(\tfrac{1}{2})\bt T(\tfrac{1}{16}). \label{dimonerel2}
	\end{align}
	It follows that
	$$
	T(\tfrac{1}{2})\bt T(\tfrac{1}{2})\cong T(0),\quad
	T(\tfrac{1}{2})\bt T(\tfrac{1}{16})\cong T(\tfrac{1}{16}),
	$$
	as left $A(L\left(\frac{1}{2},0\right))$-modules.
	On the other hand, $L(-1)v\odot v$ and $v\odot v$ are linearly independent in
	$T(\tfrac{1}{16})\bt T(\tfrac{1}{16})$, and
	\begin{equation}\label{Atensor1616}
		T(\tfrac{1}{16})\bt T(\tfrac{1}{16})\cong T(0)\op T(\tfrac{1}{2})
	\end{equation}
	as left $A(L\left(\frac{1}{2},0\right))$-modules.
\end{prop}

\begin{proof}
	We first show that \eqref{dimonerel} and \eqref{dimonerel2} arise from the relation
	$\bar{x}\cdot (\bar{x}-\tfrac{1}{2})\cdot (\bar{x}-\tfrac{1}{16})=0$
	in the Zhu algebra $A(L\left(\frac{1}{2},0\right))$.
	Indeed, let $u=v_{\tfrac{1}{2},\tfrac{1}{2}}$. Then by \eqref{viraction},
	\begin{align*}
		\bar{x}.(u\odot u)&=L(-1)u\odot u+u\odot u,\\
		(\bar{x}-\tfrac{1}{2}).(u\odot u)&=L(-1)u\odot u+\tfrac{1}{2} u\odot u,\\
		(\bar{x}-\tfrac{1}{16}). (u\odot u)&=L(-1)u\odot u+\tfrac{15}{16} u\odot u.
	\end{align*}
	On the other hand, using the fact that
	$L(-2)u=\tfrac{3}{4} L(-1)^2u$ in $L\left(\tfrac{1}{2},\tfrac{1}{2}\right)$
	and \eqref{eq:virrel2}, we obtain
	\begin{align*}
		\bar{x}.(L(-1)u\odot u)
		&=L(-1)^2u\odot u+L(0)L(-1)u\odot u+L(-1)u\odot L(0)u\\
		&=\tfrac{4}{3} L(-2)u\odot u+2L(-1)u\odot u\\
		&=\tfrac{4}{3}\bigl(-L(-1)u\odot u+u\odot L(0)u\bigr)+2L(-1)u\odot u\\
		&=\tfrac{2}{3}\bigl(L(-1)u\odot u+u\odot u\bigr).
	\end{align*}
	Therefore,
	\begin{align*}
		(\bar{x}-\tfrac{1}{2}).\bar{x}.(u\odot u)
		&=\tfrac{2}{3}\bigl(L(-1)u\odot u+u\odot u\bigr)
		+\tfrac{1}{2}\bigl(L(-1)u\odot u+u\odot u\bigr)\\
		&=\tfrac{7}{6}\bigl(L(-1)u\odot u+u\odot u\bigr).
	\end{align*}
	Furthermore,
	\begin{align*}
		0&=(\bar{x}-\tfrac{1}{16}). (\bar{x}-\tfrac{1}{2}).\bar{x}.(u\odot u)\\
		&=\tfrac{7}{6}\Bigl(
		\tfrac{2}{3}\bigl(L(-1)u\odot u+u\odot u\bigr)
		+\bigl(L(-1)u\odot u+u\odot u\bigr)
		-\tfrac{1}{16}\bigl(L(-1)u\odot u+u\odot u\bigr)
		\Bigr)\\
		&=\tfrac{539}{288}\bigl(L(-1)u\odot u+u\odot u\bigr).
	\end{align*}
	In particular, $L(-1)u\odot u+u\odot u=0$ in
	$T(\tfrac{1}{2})\bt T(\tfrac{1}{2})$, proving \eqref{dimonerel}.

Now we have $T(\tfrac{1}{2})\bt T(\tfrac{1}{2})=\C\,u\odot u$, and by \eqref{viraction},
\[
\bar{x}.(u\odot u)=L(-1)u\odot u+u\odot u=0.
\]
Moreover, since $T(\tfrac{1}{2})$ is self-dual, it follows from Corollary~\ref{coro6.9} (3) that
$T(\tfrac{1}{2})\bt T(\tfrac{1}{2})\neq 0$.
	Hence $T(\tfrac{1}{2})\bt T(\tfrac{1}{2})\cong T(0)$.

	A similar calculation shows that the following relation holds in
	$T(\tfrac{1}{2})\bt T(\tfrac{1}{16})$:
	$$
	0=(\bar{x}-\tfrac{1}{16}). (\bar{x}-\tfrac{1}{2}).\bar{x}. (u\odot v)
	=\tfrac{385}{2304}\bigl(L(-1)u\odot v+\tfrac{1}{2}u\odot v\bigr).
	$$
	This proves \eqref{dimonerel2}, and
	$T(\tfrac{1}{2})\bt T(\tfrac{1}{16})=\C\,u\odot v$.
	Moreover, by \eqref{viraction},
	\begin{align*}
		\bar{x}.(u\odot v)
		&=L(-1)u\odot v+L(0)u\odot v+u\odot L(0)v\\
		&=L(-1)u\odot v+\tfrac{1}{2}u\odot v+\tfrac{1}{16} u\odot v\\
		&=\tfrac{1}{16} u\odot v.
	\end{align*}
	This shows that
	$T(\tfrac{1}{2})\bt T(\tfrac{1}{16})\cong T(\tfrac{1}{16})$
	as an $A(L\left(\frac{1}{2},0\right))$-module.

	Finally, for $v=v_{\tfrac{1}{2},\tfrac{1}{16}}$, we have
	\begin{equation}\label{1/161/16}
		\begin{aligned}
			\bar{x}.(v\odot v)&=L(-1)v\odot v+\tfrac{1}{8} v\odot v,\\
			\bar{x}.(L(-1)v\odot v)&=\tfrac{3}{8} L(-1)v\odot v+\tfrac{3}{64} v\odot v.
		\end{aligned}
	\end{equation}
	In particular, it is easy to check that
	$$
	(\bar{x}-\tfrac{1}{2})\cdot \bar{x}.(v\odot v)=0.
	$$
	Hence the relation in $A(L\left(\frac{1}{2},0\right))$ does not give rise to new relations in
	$T(\tfrac{1}{16})\bt T(\tfrac{1}{16})$.
	Therefore, $L(-1)v\odot v$ and $v\odot v$ are linearly independent in
	$T(\tfrac{1}{16})\bt T(\tfrac{1}{16})$.
	Let
	$$
	\al=L(-1)v\odot v-\tfrac{3}{8} v\odot v,\quad
	\b=L(-1)v\odot v+\tfrac{1}{8} v\odot v.
	$$
	Then it follows from \eqref{1/161/16} that
	$$
	\bar{x}.\al=0,\quad \bar{x}.\b=\tfrac{1}{2}\b.
	$$
	Hence
	$$
	T(\tfrac{1}{16})\bt T(\tfrac{1}{16})
	=\C \al\op \C \b
	\cong T(0)\op T(\tfrac{1}{2})
	$$
	as an $A(L\left(\frac{1}{2},0\right))$-module, proving \eqref{Atensor1616}.
\end{proof}


		\section{Acknowledgements}
		I wish to thank Xu Gao, Angela Gibney, and Danny Krashen for their valuable discussions and useful comments on this paper. 
		
		\section{Statements and Declarations - Data Availability/Conflict of Interest}
		
		On behalf of all authors, the corresponding author states that there is no
		conflict of interest and data sharing is not applicable to this article as no datasets were generated or analysed during the current study.

	\end{document}